\documentclass[11pt]{amsart}
\usepackage{amscd,amssymb}
\usepackage{graphicx}
\usepackage[all]{xy}

\theoremstyle{plain}

\newtheorem{theorem}[equation]{Theorem}
\newtheorem{thm}[equation]{Theorem}
\newtheorem{prop}[equation]{Proposition}
\newtheorem{cor}[equation]{Corollary}

\newtheorem{lemma}[equation]{Lemma}

\numberwithin{equation}{section}

\newcommand{\Z}{\mathbb Z}

\newcommand{\C}{\mathbb C}

\def\Hom{{\rm Hom}}

\def\Aut{{\rm Aut}}

\def\SL{{\rm SL}}

\def\GSp{{\rm GSp}}

\def\PGSp{{\rm PGSp}}
\def\PGSO{{\rm PGSO}}
\def\Sp{{\rm Sp}}
\def\Spin{{\rm Spin}}

\def\SU{{\rm SU}}

\def\GL{{\rm GL}}
\def\PGL{{\rm PGL}}

\def\SO{{\rm SO}}
\def\Ind{{\rm Ind}}
\def\Sp{{\rm Sp}}

\def\O{{\rm O}}

\def\Z{{\mathbb Z}}

\title[Howe Duality and Dichotomy] {Howe duality and Dichotomy \\ for exceptional theta correspondences}

\author{Wee Teck Gan and Gordan Savin}

\address{W.T.G.:   Department of Mathematics, National University of Singapore, 10 Lower Kent Ridge Road
Singapore 119076} \email{matgwt@nus.edu.sg}
\address{G. S.: Department of Mathematics, University of Utah, Salt Lake City, UT}\email{savin@math.utah.edu}
\begin{document}

\begin{abstract}
We study three exceptional theta correspondences for $p$-adic groups, where one member of the dual pair is the exceptional group $G_2$. 
We prove  the Howe duality conjecture for these dual pairs  and a dichotomy theorem, and  determine explicitly the theta lifts of all 
non-cuspidal representations.   
 \end{abstract}
 
 \maketitle

\section{\bf Introduction}

Let $F$ be a non-archimedean local field of characteristic $0$ and residue characteristic $p$. 
In this paper, we study the local theta correspondence furnished by the following diagram of dual pairs:
\vskip 5pt

\[
\xymatrix@R=2pt{
&&{\rm PGSp}_6\\
&G_2 \ar@{-}[ru]\ar@{-}[ld]\ar@{-}[rd]&\\
{\rm PD^{\times}}&&{\rm PGL_3\rtimes \Z/2\Z}
}
\]
where $D$ denotes a cubic division $F$-algebra, so that $PD^{\times}$ is the unique inner form of $\PGL_3$.  More precisely, one has the dual pairs
\[  \begin{cases} 
({\rm PGL}_3 \rtimes \Z/2\Z) \times G_2 \subset E_6 \rtimes \Z/2\Z  \\
PD^{\times} \times G_2 \subset E_6^D \\
G_2 \times \PGSp_6 \subset E_7 \end{cases} \]
where the exceptional groups of type $E$ are all of adjoint type.  In each of the three cases, the centralizer of $G_2$ is a group $H_J=\Aut(J)$, where 
$J$ is a Freudenthal-Jordan algebra of rank 3. 
One can thus consider the restriction of the minimal representation $\Pi$ \cite{GS05} of $E$ to the relevant dual pair and obtain a local theta correspondence.  
\vskip 5pt

More precisely, if  $\pi \in {\rm Irr}(G_2)$ is an irreducible smooth representation of $G_2$, then the maximal $\pi$-isotypic quotient of $\Pi$ 
\[ 
\Pi/\cap_{\phi\in {\rm Hom}_{G_2}(\Pi, \pi)} {\rm ker}(\phi). 
\] 
can be expressed as $\pi\otimes \Theta(\pi)$ for some smooth representation $\Theta(\pi)$ of $H_J$  \cite[Lemme III.4]{MVW}. 
The representation $\Theta(\pi)$ is called the big theta lift of $\pi$, and 
its maximal semi-simple quotient (cosocle) is denoted $\theta(\pi)$. We say that $\pi$ has nonzero theta lift to $H_J$ if $\Theta(\pi) \ne 0$, or equivalently $\Hom_{G_2}(\Pi, \pi) \ne 0$. 
Similarly, one can consider the theta lift from $H_J$ to $G_2$ and have the analogous notions. 

\vskip 5pt

The first main result of this paper is the following dichotomy theorem:
\vskip 5pt

\begin{thm}  \label{T:intro1}
Let $\pi \in {\rm Irr}(G_2(F))$. Then $\pi$ has nonzero theta lift to exactly one of $PD^{\times}$ or ${\rm PGSp}_6(F)$.
\end{thm}
\vskip 5pt
The group ${\rm PGL}_3 \rtimes \Z/2\Z$ is not featured in the dichotomy theorem, but it is needed for some finer aspects of the theta correspondences. For example, 
every irreducible discrete series representation of $G_2$ lifts to a discrete series representation of precisely one of the three groups. 
After the above dichotomy theorem, we consider the problem of  understanding these theta correspondences more precisely.  
These local theta correspondences have all been studied to some extent by
Maggard-Savin \cite{MS}, Gross-Savin \cite{GrS}, Gan \cite{G}, Savin \cite{Sa}, Gan-Savin \cite{GS99, GS04}  and Savin-Weissman \cite{SWe}. 
Though various neat results were obtained in the various cases, they fall short of determining the theta correspondences completely. One of the main results of this paper is the completion of the analysis begun in these papers.  
\vskip 5pt

The main difficulty in studying these exceptional theta correspondences is that, unlike the classical theta correspondence, one does not know a priori the analog of the Howe duality conjecture.  Namely, one does not know that $\Theta(\pi)$ has finite length with unique irreducible quotient (that is,  $\theta(\pi)$ is irreducible if $\Theta(\pi)$ is nonzero). In this paper, we show that 
 the analog of the Howe duality conjecture holds for these dual pairs.  To summarize, we have: 
\vskip 5pt

 \begin{thm}  \label{T:intro2}
 The Howe duality conjecture holds for the three dual pairs considered here. Namely, for $\pi_1, \pi_2 \in {\rm Irr}(G_2(F))$, $\Theta(\pi_i)$ has finite length and 
 \[   \dim \Hom_{H_J}(\theta(\pi_1), \theta(\pi_2)) \leq \dim \Hom_{G_2}(\pi_1, \pi_2). \]
 Likewise, for $\tau \in {\rm Irr}(H_J)$, $\Theta(\tau)$ has finite length with unique irreducible quotient (if nonzero).
 \vskip 5pt
 
 More precisely, we have:
 \vskip 5pt
 
 \noindent (i)  The theta correspondence for $PD^{\times} \times G_2$ defines an injective map
 \[  \theta_D :  {\rm Irr}^{\heartsuit} (PD^{\times} ) \hookrightarrow {\rm Irr} (G_2(F)), \]
 where  ${\rm Irr}^{\heartsuit}( PD^{\times}) \subset  {\rm Irr}(PD^{\times})$ is the subset of representations which have nonzero theta lift to $G_2$.  
 If $p\ne 3$, then ${\rm Irr}^{\heartsuit} (PD^{\times}) = {\rm Irr}(PD^{\times})$, so that one has an injective map:
 \[  \theta_D :  {\rm Irr} (PD^{\times})  \hookrightarrow {\rm Irr} (G_2(F))  \]
 \vskip 5pt
 
 \noindent (ii)  The theta correspondence for $(\PGL_3(F) \rtimes \Z/2\Z) \times G_2$ defines an injective map
 \[  \theta_B :  {\rm Irr} ^{\heartsuit}(\PGL_3(F) \rtimes \Z/2\Z)  \hookrightarrow {\rm Irr} (G_2(F)), \]
 where  ${\rm Irr} ^{\heartsuit}(\PGL_3(F) \rtimes \Z/2\Z) \subset  {\rm Irr}(\PGL_3(F) \rtimes \Z/2\Z)$ is the subset of representations which have nonzero theta lift to $G_2$.
 Moreover, one can determine the subset $ {\rm Irr} ^{\heartsuit}(\PGL_3(F) \rtimes \Z/2\Z)$ explicitly, and 
the image of $\theta_B$ is disjoint from that of $\theta_D$ by the dichotomy theorem.  
\vskip 5pt
 
 \noindent (iii) The theta correspondence for $G_2 \times \PGSp_6$ defines an injection
 \[  \theta: {\rm Irr}(G_2(F)) \smallsetminus {\rm Im}( \theta_D)  \hookrightarrow  {\rm Irr} (\PGSp_6(F)). \]

 \end{thm}
\vskip 10pt
 For an irreducible representation $\tau$ of $PD^{\times}$, the non-vanishing of $\theta_D(\tau)$
is equivalent to existence of non-zero vectors in $\tau$ fixed by a maximal torus in $PD^{\times}$.
The  existence of such vectors has been checked by Lonka and Tandon \cite{LT} in the tame case, where $p\neq 3$. Thus, if $p=3$, we do not know that all 
irreducible representations of $PD^{\times}$ lifts to $G_2$, but the lift is still one-to-one on the subset of those representations that have nonzero lift.

\vskip 5pt In fact,  for the three dual pairs, we determine the theta lift of \underline{all} non-supercuspidal representations of $G_2$, and the lift of supercuspidal representations 
whose lift is not supercuspidal.  The detailed statements are in the main text, and we simply state the following qualitative result here: 

\vskip 5pt 
\begin{thm} \label{T:intro3} The three theta correspondences satisfy the following properties: 
\vskip 5pt 
\noindent (i)  The correspondences preserve tempered representations. 

\vskip 5pt 
\noindent (ii)  Any discrete series representation of $G_2$ lifts to a discrete series representation of precisely one of the three groups. 
\vskip 5pt 

\noindent (iii) The correspondences are functorial for non-tempered representations.

\end{thm} 
\vskip 10pt 
As an illustration of our results, let $\pi_{\rm gen}$ be a Whittaker generic (henceforth simply generic) supercuspidal representation of $G_2$. Then we have the following 
description of $\Theta(\pi_{\rm gen})$, the big theta lift to $\PGSp_6$,  completing the results of \cite{SWe}: 
\vskip 5pt

\noindent (a) If $\pi_{\rm gen}$ is not a theta  lift from $\PGL_3(F) \rtimes \Z/2\Z$, then $\Theta(\pi_{\rm gen})$ is an irreducible  generic supercuspidal  representation of 
$\PGSp_6$.

\vskip 5pt

\noindent (b)  Otherwise, $\pi_{\rm gen}$ is a lift of a supercuspidal representation of $\PGL_3(F) \rtimes \Z/2\Z$. Now the irreducible representations of 
$\PGL_3(F) \rtimes \Z/2\Z$ can be described in terms of those of $\PGL_3$ as follows: 

\vskip 5pt 
\begin{itemize} 
\item ${\rm Ind}_{\PGL_3}^{\PGL_3 \rtimes \Z/2\Z} (\tau) \cong {\rm Ind}_{\PGL_3}^{\PGL_3 \rtimes \Z/2\Z} (\tau^{\vee}) $ where $\tau\ncong \tau^{\vee} \in {\rm Irr}(\PGL_3)$;  
\vskip 5pt 
\item $\tau^+ \ncong \tau^{-}$, two extensions of a self dual ($\tau\cong \tau^{\vee}$)  representation of $\PGL_3$. 
\end{itemize} 

\vskip 10pt 
Accordingly, we have the two following two cases: 
\vskip 5pt

\noindent (b$_1$) If $\theta_B( {\rm Ind}_{\PGL_3}^{\PGL_3 \rtimes \Z/2\Z} (\tau) )=\pi_{\rm gen}$, 
then 
\[ 
\Theta(\pi_{\rm gen})\cong  {\rm Ind}_{P_3}^{\PGSp_6} (\tau) \cong {\rm Ind}_{P_3}^{\PGSp_6} (\tau^{\vee}) 
\] 
where $P_3$ is a maximal parabolic of $\PGSp_6$ with Levi factor $\GL_3$. 
\vskip 5pt

\noindent (b$_2$) If $\theta_B(\tau^+) = \pi_{\rm gen}$,  then 
 $\theta_B(\tau^-)=\pi_{\rm deg}$ is an irreducible degenerate 
(i.e. nongeneric)   supercuspidal representation of $G_2$. Since $\tau$ is self-dual, 
the induced representation ${\rm Ind}_{P_3}^{\PGSp_6} (\tau)$ is reducible. It is a direct sum
\[ 
 {\rm Ind}_{P_3}^{\PGSp_6} (\tau)=  {\rm Ind}_{P_3}^{\PGSp_6} (\tau)_{\rm gen} \oplus  {\rm Ind}_{P_3}^{\PGSp_6} (\tau)_{\rm deg} 
 \] 
 of two summands, with the first summand  generic and  the second  degenerate. We have: 
\[ 
\Theta(\pi_{\rm gen})\cong {\rm Ind}_{P_3}^{\PGSp_6} (\tau)_{\rm gen} \text{ and }  \Theta(\pi_{\rm deg})\cong {\rm Ind}_{P_3}^{\PGSp_6} (\tau)_{\rm deg}. 
\] 
 \vskip 10pt

To conclude the introduction, we would like to explain the general idea and strategy for proving the Howe duality theorem.
We begin with a discussion of the statement:
\begin{itemize}
\item[(a)] $\Theta(\pi)$ has finite length.
\end{itemize}
 This finiteness result is fundamental and it was shown by Kudla \cite{K} for the classical theta correspondence. The main tools used are his computation of the Jacquet modules of the Weil representation (relative to maximal parabolic subgroups of the two members of the dual pair) and his exploitation of the doubling see-saw identity. One key consequence of the finite length of $\Theta(\pi)$ is  \vskip 5pt
\begin{itemize}
 \item[(b)]   If $\Theta(\pi) \ne 0$, then it has an irreducible quotient.
\end{itemize}
 For the dual pairs considered in this paper, we will in fact first prove statement (b) and then use it with other inputs to show (a).
 \vskip 5pt
 
 Let us elaborate on this slightly subtle point and our strategy of proof. By Bernstein's decomposition, we may decompose
 \[  \Theta(\pi) = \Theta(\pi)_{c} \oplus \Theta(\pi)_{nc} \]
 as the sum of its cuspidal part and non-cuspidal part.  If $\Theta(\pi)_c$ is nonzero, then it certainly has an irreducible quotient, since it is semisimple.  On the other hand, 
 we shall show using Jacquet module computations that 
 \begin{itemize}
 \item[(c)]  $\Theta(\pi)_{nc}$ has finite length and hence has an irreducible quotient if it is nonzero.
 \end{itemize}
 The necessary Jacquet module computations are already available in the literature \cite{MS, Sa} when $H_J = PD^{\times}$ or $\PGL_3$ and are partially available \cite{MS, GrS} for $H_J = \PGSp_6$. In \S \ref{S:jacquet}, we complete the remaining Jacquet module computations. We stress that the material in \S \ref{S:jacquet} is independent of the rest of the paper and could have been discussed earlier in the paper; we have refrained from doing so, as the computations are rather technical. Consequences of the results of \S \ref{S:jacquet} are then discussed in \S \ref{S:conseq}.
 \vskip 5pt
 
 In any case,  we show statement (b) by showing (c) via Jacquet modules; the proof is written in \S \ref{S:finite}. The statement (b) is used in the proof of the dichotomy theorem (i.e. Theorem \ref{T:intro1}) for tempered representations in \S \ref{S:dichotomy}. For nontempered representations, the Jacquet module computations (of \cite{MS, GrS} and \S \ref{S:jacquet}-\ref{S:conseq}) will tell us everything about their theta lifts. 
 \vskip 5pt
 
 For the statement (a) (finite length of $\Theta(\pi)$), it remains to show that $\Theta(\pi)_c$ is of finite length. We shall show this together with the Howe duality conjecture, by showing that $\Theta(\pi)_c$ is either irreducible or $0$. This part of the argument may be considered the analog of the doubling see-saw argument, though one would legitimately question  what that  means in the setting of exceptional dual pairs. 
 
 \vskip 5pt
 
  It will  be instructive to first recall the argument for 
a classical dual pair $\Sp(W) \times \O(V)$, where $W$ is a symplectic space and $V$ a quadratic space. The Howe duality theorem was shown by examining the so-called doubling see-saw diagram:
 \[
 \xymatrix{
  \O(V^{\square})  \ar@{-}[dr] \ar@{-}[d] & \Sp(W) \times \Sp(W)   \ar@{-}[d] \\
  \O(V) \times \O(V)  \ar@{-}[ur] & \Sp(W)^{\Delta}}
\]  
where $V^{\square} = V + V^-$ is the doubled quadratic space. Starting from $\pi , \pi' \in {\rm Irr}(\O(V))$, the resulting see-saw identity gives
\[   \dim \Hom_{\Sp(W)}(\theta(\pi'), \theta(\pi))  \leq  \dim \Hom_{\O(V) \times \O(V)} ( \Theta(1), \pi' \otimes \pi^{\vee}).   \]
where $\Theta(1)$ is the big theta lift of the trivial representation of $\Sp(W)^{\Delta}$ to $\O(V^{\square})$.  By the local analog of the Siegel-Weil formula, one identifies $\Theta(1)$ with a submodule of a certain degenerate principal series representation $I$ on $\O(V^{\square})$. This implies that, for $\pi$ outside a small family of representations, 
\[ \dim \Hom_{\O(V) \times \O(V)} ( \Theta(1), \pi' \otimes \pi^{\vee}) \subset \dim \Hom_{\O(V) \times \O(V)} ( I , \pi' \otimes \pi^{\vee}). \]
Using Mackey theory, one can analyze the latter space and show that, for $\pi$ outside another small family of representations, 
\[  \dim \Hom_{\O(V) \times \O(V)} ( I , \pi' \otimes \pi^{\vee})  \leq \dim \Hom_{\O(V)}(\pi', \pi). \]
Taken together, one obtains the desired inequality
\[  \dim \Hom_{\Sp(W)}(\theta(\pi'), \theta(\pi))  \leq   \dim \Hom_{\O(V)}(\pi', \pi) \]
for $\pi$ outside a small family of representations. For this small family of representations, one needs to do a separate argument. 
\vskip 5pt

Now for the  exceptional dual pairs $G \times H$ studied in this paper, there is no analog of the doubling see-saw; this is ultimately tied to the sporadic nature of the geometry underlying exceptional groups. There is thus no direct analog of the above argument. However, the above argument is a particular manifestation of a general principle: 
\[
\text{{\it Theta correspondence typically relates or  transfers  a period on $G$ to a period on $H$}.}
\]

More precisely, given a subgroup $H_1 \subset H$ and  $\tau\in {\rm Irr}(H)$, we may consider the space $\Hom_{H_1}(\tau, \chi)$ for some  one-dimensional character $\chi$ of $H_1$, typically trivial; let us call this Hom space the $H_1$-period for $\tau$. 
Now, given
\begin{itemize}
\item[-]  $\pi \in {\rm Irr}(G)$, with big theta lift $\Theta(\pi)$ on $H$,  
\item[-] $\tau \in {\rm Irr}(H)$  an irreducible quotient of $\Theta(\pi)$, 
%\item[-]  a period $Q_1$ for the group  $H$ (associated to a subgroup $H_1$ of $H$),
\end{itemize}
one typically obtains a statement of the form
\[  \text{$H_1$-period of $\tau$} \subset \text{$H_1$-period of $\Theta(\pi)$} \cong \text{$G_1$-period of $\pi$} \]
for  some subgroup $G_1$ of $G$. 
%Here, by $Q_1$-period of $\tau$, we mean the space $\Hom_{H_1}(\tau, \chi)$ for some one-dimensional character $\chi$ of $H_1$, typically trivial.
\vskip 5pt

 Now one can turn the table around. For   an irreducible quotient $\tau$ of $\Theta(\pi)$, one can consider the $G_1$-period of $\Theta(\tau)$, which has $\pi$ as an irreducible quotient. One typically gets a statement
\[   \text{$G_1$-period of $\pi$}  \subset  \text{$G_1$-period of $\Theta(\tau)$}  \cong \text{$H_2$-period of $\tau$} \]
for some subgroup $H_2$ on $H$. 
\vskip 5pt

Iterating this process, one obtains a family of periods relative to subgroups $G_i$ on $G$ and   $H_i$ on $H$ such that
\[ \text{$H_i$-period of $\tau$ }    \subset  \text{$H_i$-period of $\Theta(\pi)$ }   \cong \text{$G_i$-period of $\pi$} \]
and
\[  \text{$G_i$-period of $\pi$ }    \subset  \text{$G_i$-period of $\Theta(\tau)$ }   \cong  \text{$H_{i+1}$-period of $\tau$}, \]
leading to a chain of containment of periods of $\pi$ and $\tau$.   One may call this a game of ping-pong with periods. Now an (empirical) observation is that the subgroups  $G_i$ and $H_i$ become more and more reductive (as $i$ increases) and one ultimately obtains a reductive period. When that happens, the next iteration will result in a seesaw diagram analogous to that in the classical case above and the consideration of an appropriate  degenerate principal series representation.  
 \vskip 5pt
 
 Now the miracle is that a Mackey theory argument with this degenerate principal series representation then returns us the initial $H_1$-period! In other words, for some $i > 1$, one has $H_i = H_1$, and this allows one to complete the chain of containment of periods   into a cycle. In particular, if one of these period spaces is finite-dimensional, then this cycle of containment is a cycle of equalities. This is the key step in our proof of the Howe duality theorem for the dual pairs treated here. We shall play this game of ping-pong with periods on two occasions, in \S \ref{S:dichotomy} and \S \ref{S:howetemp}. This seems to us to be a rather robust method for proving the Howe duality conjecture and should be applicable to other exceptional dual pairs, though the precise details will undoubtedly be different in each case.

\vskip 5pt

To be honest, just as in the classical case, this last step will hold for representations of $G$ outside a small family, which needs to be treated separately. 
One can characterize this exceptional family precisely, but we prefer not to do it, and simply observe that tempered irreducible representations of $G$ do not lie in this exceptional family.   As mentioned earlier, the  theta lifts of nontempered representations can be explicitly determined using the Jacquet module computations of \cite{MS, GrS} and \S \ref{S:jacquet}-\ref{S:conseq}.
 \vskip 5pt
 
 The main motivation for showing the results of this paper is the application to the local Langlands correspondence for the exceptional group $G_2$. 
  We will treat this application in a sequel to this paper. 
  
\vskip 15pt
\section{\bf The Group $G_2$}  \label{S:G2}
We begin by introducing the algebraic group $G_2$ over $F$.
\vskip 5pt

\subsection{\bf Octonion algebra.} 
Let $\mathbb{O}$ be the split octonion algebra over $F$. Thus, $\mathbb{O}$ is an $8$-dimensional 
non-associative and non-commutative $F$-algebra. It comes equipped with a conjugation map $x \mapsto \overline{x}$ with associated norm $N(x)  = x \cdot \overline{x} = \overline{x} \cdot x$ and trace $Tr(x)  = x + \overline{x}$. Moreover, $N : \mathbb{O}  \rightarrow F$ is a nondegenerate quadratic form.   
\vskip 5pt

Every element $x$ of $\mathbb{O}$ satisfies its characteristic polynomial $t^2 - Tr(x) t + N(x)$.  
A nonzero element $x \in \mathbb{O}$ is said to be of rank $1$ if $N(x)  = 0$. Otherwise it is of rank $2$, in which case the subalgebra $F[x]$ of $\mathbb{O}$ generated by $x$ over $F$ is isomorphic to the separable quadratic $F$-algebra $F[t]/ (t^2 - Tr(x) t +N(x))$. We  denote by $\mathbb{O}_0$ the 7-dimensional subspace of trace $0$ elements in $\mathbb{O}$.
\vskip 5pt

\subsection{\bf Automorphism group.}
The group $G_2$ is the automorphism group of the $F$-algebra $\mathbb{O}$. It is a split simple linear algebraic group of rank $2$ which is both simply connected and adjoint. If we fix a maximal torus $T$ contained in a Borel subgroup $B$, then we obtain a system of simple roots $\{ \alpha,  \beta\}$ of $G_2$ relative to $(T,B)$, with $\alpha$ short and $\beta$ long. 
The resulting root system is given by the following diagram.
\vskip 5pt

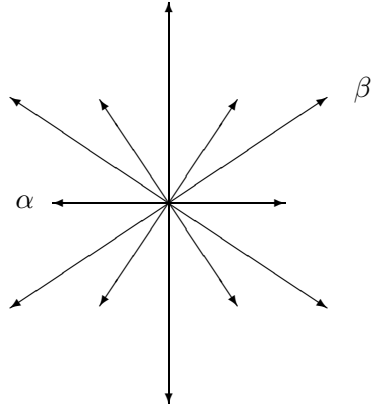
 \begin{picture}(300,200)(-100,0)

\put(100,100){\vector(1,0){44}}
\put(100,100){\vector(-1,0){44}}
\put(100,100){\vector(0,1){76}}
\put(100,100){\vector(0,-1){76}}
\put(100,100){\vector(3,2){60}}
\put(100,100){\vector(-3,2){60}}
\put(100,100){\vector(-3,-2){60}}
\put(100,100){\vector(3,-2){60}}
\put(100,100){\vector(2,-3){26}}
\put(100,100){\vector(2,3){26}}
\put(100,100){\vector(-2,3){26}}
\put(100,100){\vector(-2,-3){26}}
\put(42,98){$\alpha$} 
\put(170,140){$\beta$}
\end{picture}

\noindent  The highest root is $\beta_0 = 3 \alpha + 2\beta$.

\vskip 10pt
\subsection{\bf Maximal torus.}  \label{SS:torus}   Following Mui\'c, we will fix the isomorphism $T \cong \mathbb{G}_m^2$ by
\[  t \mapsto   ((2 \alpha + \beta)(t)  ,  (\alpha+\beta)(t)). \]
Any pair of  characters $(\chi_1, \chi_2)$ of $F^{\times}$ thus define a character $\chi_1 \times \chi_2$ of $T$ by composition with the above isomorphism.
\vskip 5pt

\subsection{\bf Parabolic subgroups.}
Up to conjugation, $G_2$ has 2 maximal parabolic subgroups which may be described as follows. 
Let $V_1 \subset V_2 \subset \mathbb{O}_0$ be subspaces of dimension $1$ and $2$ respectively on which the octonion multiplication is identically zero. Let $P$ and $Q$ be the stabiliser of $V_2$ and $V_1$ respectively. Then $P = MN$ and $Q = LU$ are the two maximal parabolic subgroups of $G_2$.  Moreover, their intersection 
$B = P \cap Q$ is a Borel subgroup of $G_2$. 
\vskip 5pt

The Levi factor $M$ of $P$ is given by
\[  M \cong \GL(V_2) \cong \GL_2 . \]
The isomorphism of $M$ with $\GL_2$ can be fixed so that the modulus character of $M$ is 
$\delta_P  = \det^3$.
Its unipotent radical $N$ is a 5-dimensional Heisenberg group with 1-dimensional center $Z =  U_{\beta_0}$. 
The action of $M$ on $N/Z$   is isomorphic to $Sym^3 (F^2) \otimes \det^{-1}$.  Moreover, the generic $M(F)$-orbits on $N(F)/Z(F)$ is naturally parametrized by the set of isomorphism classes of separable cubic $F$-algebras.
\vskip 5pt

The Levi factor $L$ of $Q$ is given by
\[  L \cong \GL(V_3/V_1)  \cong \GL_2 \]
where 
\[  V_3 = \{ x \in \mathbb{O}_0: x \cdot y = 0 \, \, \text{for all $y \in V_1$} \}. \]
The isomorphism of $L$ with $\GL_2$ can be fixed so that the modulus character of $L$ is $\delta_Q = \det^5$.
The unipotent radical $U$ is a 5-dimensional 3-step nilpotent group:
\[   U =U_0 \supset U_1 \supset U_2 \supset U_3 = \{1 \}, \]
such that
\[  U_0/U_1  = U_{\alpha} \times U_{\alpha + \beta} , \quad  U_1/U_2 \cong U_{2\alpha + \beta}  \quad 
U_2/U_3 = U_{\beta_0} \times U_{\beta_0 - \beta}. \]
As representation of $L$, one has
\[  U_0/U_1 \cong F^2, \quad U_1/U_2 \cong \det \quad U_2/U_3 \cong F^2 \otimes \det. \]

\vskip 10pt

\subsection{\bf The subgroup $\SL_3$.}

The subgroup of $G_2$ generated by the long root subgroups is isomorphic to $\SL_3$.  The normaliser of $\SL_3$ in $G_2$ is a semidirect product
$\SL_3 \rtimes \Z/2\Z$, with the nontrivial element of $\Z/2\Z$ acting on $\SL_3$ as a pinned outer automorphism. The subgroup $\SL_3$ is the pointwise stabilizer of a quadratic subalgebra of $\mathbb O$ which is isomorphic to $F \times F$, whereas the setwise stabilizer of such a subalgebra is $\SL_3 \rtimes \Z/2\Z$.  
\vskip 5pt

More generally, given a subalgebra of $\mathbb O$ which is isomorphic to a quadratic field extension of $F$, the pointwise stabilizer of this subalgebra is isomorphic to the quasi-split special unitary group $\SU_3^K$; the setwise stabilizer of this subalgebra is $\SU_3^K  \rtimes \Z/2\Z$.

\vskip 10pt

\subsection{\bf The dual group.}  
The Langlands dual group of $G_2$ is the complex Lie group $G_2(\C)$. 
In particular, one has the subgroups
\[  \SO_3(\mathbb C) \subset \SL_3(\C) \subset \SL_3(\C) \rtimes \Z/2\Z \subset G_2(\C). \]
The centralizer of $\SL_3(\mathbb C)$ in $G_2(\mathbb C)$ is $\mu_3$, and the centralizer of $\SO_3(\mathbb C)$ in $G_2(\mathbb C)$ is $S_3 = \mu_3 \rtimes \Z/2\Z$.  
Let $\SL_{2,l}(\mathbb C)$ be a long root $\SL_2$. Then the centralizer of $\SL_{2,l}(\mathbb C)$ in $G_2(\mathbb C)$  is  
$\SL_{2,s}(\mathbb C)$, a short root $\SL_2$, and vice versa. Thus we also have the subgroup 
\[ 
\SL_{2,l}(\mathbb C) \times_{\mu_2} \SL_{2,r}(\mathbb C) \cong \SO_4(\mathbb C) \subset G_2(\mathbb C). 
\]

\section{\bf Representations of $G_2$} \label{S:repG2}
In this section, we introduce some notations for the representations of $G_2(F)$. In particular, we shall describe all non-supercuspidal representations. The results in this section are 
almost entirely due to Mui\'c \cite{Mu}. 

\vskip 10pt

\subsection{\bf Principal series representations for $P$. }
We first consider the principal series representations for the Heisenberg parabolic subgroup $P = MN$, where $M \cong \GL_2$.
 Let $\tau$ be an irreducible  representation of $M$ with central character $\omega_{\tau}$ and set
\[ 
I_P(\tau)= {\rm Ind}_P^{G_2} \tau  \text{ and } 
 I_P(s, \tau)  = {\rm Ind}_P^{G_2}( |\det|^s  \cdot \tau)  \]
if we need to consider a family of induced representations. If $I_P(s,\tau)$ is a standard module, we will denote its unique Langlands quotient by $J_P(s,\tau)$.
Now  we have:
\vskip 5pt

\begin{prop} \label{P:PSP}
(i) If $\tau$ is a unitary supercuspidal representation, then $I_P(s, \tau)$ is reducible if and only if  $\tau^{\vee} \cong \tau$ (so $\omega_{\tau}^2=1$) and one of the following holds:
\vskip 5pt

\begin{itemize}
\item $\omega_{\tau}  =1$ and $s = 1/2$, in which case there is a non-split short exact sequence of length $2$, 
\[  \begin{CD}
0 @>>>  \delta_P(\tau)  @>>>  I_P(1/2, \tau)  @>>> J_P(1/2,\tau) @>>> 0, 
\end{CD} \]
where $\delta_P(\tau)$ is a generic discrete series representation. 
\vskip 5pt

\item $\omega_{\tau}  \ne 1$ and $s = 0$, in which case
\[  I_P(0,\tau)  = I_P(\tau)_{\mathrm {gen}}  \oplus I_{P}(\tau)_{\mathrm {deg}} \]
where $I_P(\tau)_{\mathrm {gen}}$ is generic. 
\end{itemize}
\vskip 5pt

\noindent (ii)  If $\tau = {\rm st}_{\chi}$ is a twisted Steinberg representation, then $I_P(s, \tau)$ is irreducible except for the following cases:
\vskip 5pt

\begin{itemize}
\item $\chi =1$ and $s = \pm 3/2$ or $\pm 1/2$, in which case one has:
\[  \begin{CD}
0 @>>> {\rm St}_{G_2} @>>> I_P(3/2, {\rm st}) @>>>  J_P(3/2, {\rm st})@>>>0,\end{CD} \]
with ${\rm St}_{G_2}$ the Steinberg representation. On the other hand,
 $I_P(1/2, {\rm st})$ has length $3$, with a unique irreducible submodule $\pi_{\mathrm {gen}}[1]$ which is a generic discrete series representation, a unique irreducible Langlands quotient $J_P(1/2, {\rm st})$ and a subquotient  $J_Q(1/2, {\rm st})$.
 
\vskip 5pt

\item$\chi^2  =1$ but $\chi \ne 1$ and $s  = \pm 1/2$, in which case one has:
\[  \begin{CD} 
0@>>> \pi_{\mathrm {gen}}[\chi] @>>>  I_P(1/2, {\rm st}_{\chi}) @>>>  J_P(1/2, {\rm st}_{\chi}) @>>> 0 \end{CD} \]
where $ \pi_{\mathrm {gen}}[\chi]$  is a generic discrete series representation. 
 \vskip 5pt

\item $\chi^3  =1$ but $\chi \ne 1$ and $s = \pm 1/2$, in which case one has:
\[  \begin{CD} 
0 @>>> \pi_{\mathrm {gen}}[\chi] @>>>  I_P(1/2, {\rm st}_{\chi}) @>>>  J_P(1/2, {\rm st}_{\chi}) @>>> 0. \end{CD} \]
where $\pi_{\mathrm {gen}}[\chi]\cong\pi_{\mathrm {gen}}[\chi^{-1}]$ is a generic discrete series representation.\end{itemize}
\vskip 5pt

\noindent (iii)  If $\tau = \chi$ is 1-dimensional unitary, then $I_P(s, \tau)$ is irreducible except in the following cases:
\begin{itemize}
\item $\chi =1$ and $s = \pm 1/2$ or $\pm 3/2$, in which case one has:
\[  \begin{CD}
0 @>>> J_Q(5/2, {\rm st}) @>>> I_P(3/2, 1) @>>>  1_{G_2} @>>> 0, \end{CD} \]
whereas $I_P(1/2, 1)$ is of length $3$, with a unique irreducible submodule $\pi_{\mathrm {deg}}[1]$ which is a nongeneric discrete series representation, a unique irreducible quotient $J_Q(1, \pi(1,1))$ and a subquotient  $J_Q(1/2, {\rm st})$.
 \vskip 5pt

\item $\chi^2=1$ but $\chi \ne 1$ and $s = \pm 1/2$, in which case one has:
\[  \begin{CD} 
0@>>> J_Q(1/2, {\rm st}_{\chi}) @>>>  I_P(1/2, \chi) @>>>  J_Q(1/2, \pi(1, \chi)) @>>> 0. \end{CD} \]
\vskip 5pt

\item  $\chi^3  =1$ but $\chi \ne 1$ and $s = \pm 1/2$, in which case one has:
\[  \begin{CD} 
0 @>>> J_P(1/2, {\rm st}_{\chi^{-1}}) @>>>  I_P(1/2, {\rm st}_{\chi}) @>>>  J_Q(1/2,  \pi(\chi, \chi^{-1})) @>>> 0. \end{CD} \]
\end{itemize}

\end{prop}

\vskip 10pt
\subsection{\bf Principal series representations for $Q$. }  \label{S:PSQ}
Now we consider the principal series representations for the 3-step parabolic subgroup $Q = LU$, where $L \cong \GL_2$. 
Let $\tau$ be an irreducible unitary representation of $L$ with L-parameter $\phi_{\tau}$ and set
\[  I_Q(\tau)= {\rm Ind}_Q^{G_2} \tau  \text{ and } 
 I_Q(s, \tau)  = {\rm Ind}_Q^{G_2} |\det|^s  \cdot \tau \]
if we need to consider a family of induced representations. As before, we let $J_Q(s,\tau)$ denote the unique Langlands quotient of $I_Q(s,\tau)$ if the latter is a standard module. Then we have:
\vskip 5pt

\begin{prop}  \label{P:PSQ}
(i)  If $\tau$ is unitary supercuspidal, then $I_Q(s, \tau)$ is reducible if and only if  $\tau^{\vee} \cong \tau$ (so $\omega_{\tau}^2=1$) and one of the following holds:
\vskip 5pt
\begin{itemize}
\item $\omega_{\tau} =1$ and $s = \pm 1/2$, in which case one has: 
\[  \begin{CD}
0 @>>>  \delta_Q(\tau) @>>> I_Q(1/2, \tau) @>>>  J_Q(1/2, \tau) @>>> 0, \end{CD} \]
where $\delta_Q(\tau)$ is a generic discrete series representation. 
\vskip 5pt 

 \item $\omega_{\tau}  \ne 1$ (so $\tau$ is dihedral),  ${\rm Im}(\phi_{\tau}) = S_3$  and $s=\pm 1$,in which case one has: 
   \[  \begin{CD}
0 @>>>  \pi_{\mathrm {gen}}[\tau] @>>> I_Q(1, \tau) @>>>  J_Q(1, \tau) @>>> 0, \end{CD} \]
where $\pi_{\mathrm {gen}}[\tau] $ is a generic discrete series representation.

\vskip 5pt

\item $\omega_{\tau} \ne 1$, ${\rm Im}(\phi_{\tau})  \ne S_3$  and $s= 0$, in which case one has:

\[   I_Q(0,\tau)  = I_Q(\tau)_{\mathrm {gen}}  \oplus I_{Q}(\tau)_{\mathrm {deg}} \]
where $I_Q(\tau)_{\mathrm {gen}}$ is generic. 

\end{itemize}

\vskip 5pt

\noindent (ii) If $\tau = {\rm st}_{\chi}$ is a twisted Steinberg representation, the $I_Q(s,\tau)$ is irreducible except for the following cases:
\vskip 5pt
\begin{itemize}
\item $\chi=1$ and $s = \pm 5/2$ or $\pm 1/2$, in which case one has
\[  \begin{CD}
0 @>>> {\rm St}_{G_2} @>>>  I_Q(5/2, {\rm st}) @>>>  J_Q(5/2,  {\rm st}) @>>> 0, \end{CD} \]
and
\[  \begin{CD}
0 @>>>  \pi_{\mathrm {gen}}[1]  \oplus \pi_{\mathrm {deg}}[1]  @>>> I_Q(1/2, {\rm st}) @>>>  J_Q(1/2, {\rm st}) @>>> 0. \end{CD} \]
Here $ \pi_{\mathrm {gen}}[1]$ is the generic discrete series representation already defined in Proposition \ref{P:PSP}(ii) (first bullet point) and $\pi_{\mathrm {deg}}[1]$   is the nongeneric discrete series representation already defined in Proposition \ref{P:PSP}(iii) (first bullet point).
\vskip 5pt

\item $\chi^2=1$ but $\chi \ne 1$ and $s = \pm 1/2$, in which case one has:
 \[  \begin{CD} 
0@>>> \pi_{\mathrm {gen}}[\chi] @>>>  I_Q(1/2, {\rm st}_{\chi}) @>>>  J_Q(1/2, {\rm st}_{\chi}) @>>> 0. \end{CD} \]
Here, $ \pi_{\mathrm {gen}}[\chi] $ is the generic discrete series representation defined in Proposition \ref{P:PSP}(ii) (second bullet point).
\end{itemize}
\vskip 5pt

\noindent (iii)  If $\tau = \chi$ is 1-dimensional unitary, then $I_Q(s, \tau)$ is irreducible except in the following cases:
\begin{itemize}
\item $\chi =1$ and $s = \pm 1/2$ or $\pm 5/2$, in which case one has:
\[  \begin{CD}
0 @>>> J_P(3/2, {\rm st}) @>>> I_Q(5/2, 1) @>>>  {1}_{G_2} @>>>  0 , \end{CD} \]
whereas $ I_Q(1/2, 1)$ is of length $3$, with unique irreducible submodule $J_Q(1/2, {\rm st})$, a unique irreducible quotient $J_Q(1, \pi(1,1))$ and subquotient $ J_P(1/2, {\rm st})$.
 \vskip 5pt

\item $\chi^2=1$ but $\chi \ne 1$ and $s = \pm 1/2$, in which case one has:
\[  \begin{CD} 
0@>>> J_P(1/2, {\rm st}_{\chi}) @>>>  I_Q(1/2, \chi) @>>>  J_Q(1/2, \pi(1, \chi)) @>>> 0 . \end{CD} \]
 \end{itemize}

\end{prop}

\vskip 10pt
\subsection{\bf Principal series representations for $B$. }  \label{SS:PSB}
We now consider the principal series representations induced from the Borel subgroup $B$. More precisely, suppose that $\pi$ is a Langlands quotient of a standard module
\[  I(s_1, s_2,  \chi_1, \chi_2)  \twoheadrightarrow \pi \]
with
\[  s_1 \geq s_2 \geq 0 \]
and $\chi_i$ unitary characters of $F^{\times}$.  Here, recall the convention about characters of $T$ which we have fixed in \S \ref{SS:torus}.
Then
\[  \pi  \hookrightarrow I(-s_1, -s_2, \chi_1^{-1}, \chi_2^{-1})  = I_P( \pi(\chi_1^{-1} |-|^{-s_1} , \chi_2^{-1} |-|^{-s_2})). \]
Now the representation $\pi(\chi_1^{-1} |-|^{-s_1} , \chi_2^{-1} |-|^{-s_2})$ of $M \cong \GL_2$ is reducible if and only if  
\[ \chi_2/\chi_1   \cdot |-|^{s_2-s_1}  =  | -| ^{-1}, \quad \text{i.e. $\chi_1 = \chi_2$ and  $s_1 = s_2 +1 \geq 1$,}\]
in which case one has
\[  \pi  \hookrightarrow I_P( -s + \frac{1}{2} ,  \chi_1^{-1}) ,  \quad \text{with $s \geq 1$}. \]

\vskip 5pt 
There is another, convenient, way to bookkeep the principal series ${\rm Ind}_B^{G_2}(\chi)$. Let $\beta_1, \beta_2, \beta_3$ be three long roots 
such that $\beta_1 + \beta_2 + \beta_3=0$.  This triple is unique up the  action of the Weyl group of $G_2$. 
Then the corresponding co-roots $\beta_i^{\vee} : F^{\times} \rightarrow T$ generate $T$, in particular, 
the character $\chi$ defines three characters of $F^{\times}$ by $\chi_i= \chi\circ \beta_i^{\vee}$ (and is determined by them). Clearly, these characters satisfy 
$\chi_1 \cdot \chi_2 \cdot \chi_3=1$. 

\begin{prop} 
 The induced representation ${\rm Ind}_B^{G_2}(\chi)$  
 is irreducible unless one of the following two conditions hold: 
\begin{itemize} 
\item $\chi_i= |\cdot| ^{\pm 1}$ for some $i$  or $\chi_i/\chi_j=|\cdot |^{\pm 1}$ for a pair $i\neq j$. 

\vskip 5pt 
\item The three characters $\chi_i$ are quadratic, non-trivial and pairwise different. Then 
\[ 
{\rm Ind}_B^{G_2}(\chi)= {\rm Ind}_B^{G_2}(\chi)_{\rm gen} \oplus  {\rm Ind}_B^{G_2}(\chi)_{\rm deg}
\] 
where ${\rm Ind}_B^{G_2}(\chi)_{\rm gen}$ is generic.  
\end{itemize} 
\end{prop}

\vskip 10pt

\subsection{\bf Conjectural L-packets of $G_2$.}   \label{SS:partial-LLC}
The above results allow one to give an enumeration of the non-cuspidal representations of $G_2$.  Using the desiderata of the conjectural local Langlands correspondence  (LLC) for $G_2$, we explain how one can assign L-parameters to the noncuspidal representations of $G_2$, and hence partition them into L-packets. Recall that an $L$-parameter of $G_2$ is 
an admissible  homomorphism 
\[ 
\varphi : WD_F \longrightarrow G_2^{\vee} = G_2(\mathbb C) 
\] 
of the Weil-Deligne group $WD=W_F \times \SL_2(\mathbb C)$ to the dual group $G_2(\C)$, taken up to conjugacy by $G_2(\C)$. Let 
\[  A_{\varphi} = \pi_0(Z_{G_2}(\varphi)) \]
be the associated component group of $\varphi$. Then one expects that there should be an   L-packet 
\[ 
\Pi_{\varphi} = \{ \pi(\rho) : \rho \in \hat{A}_{\varphi} \} \subset {\rm Irr}(G_2) \] 
associated to each $\varphi$,  whose members are indexed by the characters of $A_{\varphi}$, such  that
\[  {\rm Irr}(G_2) = \bigcup_{\varphi} \Pi_{\varphi}. \]

\vskip 5pt

The non-tempered irreducible representations of  $G_2$  are uniquely realized as Langlands quotients of standard modules, so have the form $J_P(\tau)$, $J_Q(\tau)$ or $J_B(\chi)$.  The Levi factors of the parabolic subgroups $P$, $Q$ aand $B$ are isomorphic to $\GL_2$ and $\GL_1 \times \GL_1$. Since the LLC for these groups are known, one can assign L-parameters to the nontempered representations.  For example, if $\pi = J_P(\tau)$, and $\varphi_{\tau} : WD_F \longrightarrow M^{\vee} = \GL_2(\C)$ is the L-parameter of $\tau$, then the L-parametrer of $\pi = J_P(\tau)$ is the composite
\[   \varphi_{\pi} : WD_F \longrightarrow M^{\vee} \hookrightarrow G_2^{\vee} = G_2(\C). \]
Since the L-packets on the Levi subgroups are singletons, we see also that the nontempered L-packets of $G_2$ are singletons, and $A_{\varphi_{\pi}}$ is correspondingly trivial. 

\vskip 5pt
In other words, the non-tempered irreducible representations of $G_2$ are naturally parametrized by the nontempered L-parameters of $G_2$; these are the L-parameters $\varphi$ such that $\varphi(W_F)$ is unbounded.  
In the following, we will use this partial LLC to describe the effect of the various theta correspondences on nontempered representations.

 \vskip 5pt

By the same token, since irreducible tempered representations which are not square-integrable are uniquely realized as summands of principal series representations induced from   unitary square-integrable representations of Levi factors, one can attach L-parameters to these tempered (but not square-integrable) representations of $G_2$. The resulting L-parameters $\varphi$ have the property that $\varphi(W_F)$ is bounded but  $\varphi(WD_F)$ is contained in a proper Levi subgroup. 
The size of such a tempered L-packet now depends on the number of irreducible summands in the corresponding parabolically induced representations. From the results recalled in this section, one sees  that the size of a tempered L-packet $\Pi_{\varphi}$ is $1$ or $2$.  One can verify that this is the same as the size of $A_{\varphi}$. 
Moreover, in each tempered L-packet, there is a unique generic representation, and this is assigned to the trivial character of $A_{\varphi}$. 
Thus, the LLC for tempered non-discrete series representations of $G_2$ is also known, and we may refer to this partial LLC for describing these representations.
\vskip 5pt

Hence, the main issue with the LLC for $G_2$ comes down to the classification of the square-integrable or discrete series representations by discrete series L-parameters; these are the L-parameters $\varphi$ which do not factor through any proper Levi subgroup, or equivalently whose centralizer $C_{\varphi} = Z_{G_2}(\varphi)$ is finite. Guided by the desiderata of the LLC, we can now describe  the various families of discrete series L-parameters, according to $\varphi(\SL_2)$, and list all non-supercuspidal members. 
\vskip 5pt
\begin{enumerate}
\item $\varphi(\SL_2)$ is the principal $\SL_2$. Then $A_{\varphi}=1$ and the packet consists of the Steinberg representation: 
\[ 
\Pi_{\varphi} = \{ {\rm St}_{G_2}  \} 
\] 

\vskip 5pt 
\item $\varphi(\SL_2)= \SO_3 \subseteq \SL_3 \subseteq G_2$; this is the subregular $\SL_2$. The centralizer of $\SO_3$ in $G_2$ is the finite symmetric group $S_3$, so that  $\varphi$ gives by restriction a map   $\phi: W_F \rightarrow S_3$.  There are four cases to discuss: 
\vskip 5pt

\begin{itemize} 
\item $\phi(W_F)= 1$. Then $A_{\varphi}=S_3$. Let $1,r, \epsilon$ be the three irreducible representations of $S_3$: the trivial, 2-dimensional and  
 the sign character respectively. Then 
\[
\Pi_{\varphi} =\{ \pi(1)=\pi_{\rm gen} [1],\, \pi(r)=\pi_{\rm deg}[1], \, \pi(\epsilon)= \pi_{\mathrm{sc}}[1] \}
\] 
where $\pi_{\rm gen}[1]$ is defined in Proposition \ref{P:PSP}(ii) (first bullet point) and $\pi_{\rm deg}[1]$ is given in Proposition \ref{P:PSP}(iii) (first b.p.).
The representation $\pi(\epsilon)$ is a depth 0 supercuspidal representation induced from a
cuspidal unipotent representation of $G_2(\mathbb F_q)$, inflated to a hyperspecial maximal compact group. The cuspidal unipotent representation is denoted in the literature by 
$G_2[1]$ and hence our notation $\pi_{\mathrm{sc}}[1]$.  

\vskip 5pt 
\item $\phi(W_F)= \mu_2$. 
Then, by the local class field theory, $\phi$ defines a quadratic character $\chi$ of $F^{\times}$. 
Let $1$ and $-1$ denote the trivial and non-trivial characters of $A_{\varphi}=\mu_2$. Then 
\[
\Pi_{\varphi} =\{ \pi(1)=\pi_{\rm gen} [\chi], \, \pi(-1)\}, 
\] 
where $\pi_{\rm gen}[\chi]$ is as defined in Proposition \ref{P:PSP}(ii) (second b.p.). If the character $\chi$ is unramified,  then $\pi(-1)=\pi_{\rm sc}[-1]$  is a depth 0 supercuspidal representation. It is induced from a 
cuspidal unipotent representation of $G_2(\mathbb F_q)$, denoted by $G_2[-1]$, inflated to a hyperspecial maximal compact group.

\vskip 5pt 
\item $\phi(W_F)= \mu_3$. 
Then, by  local class field theory, $\phi$ defines a cubic character $\chi$ of $F^{\times}$. 
Let $1$, $\omega$ and $\omega^2$ denote the characters of $A_{\varphi}=\mu_3$. Then 
\[
\Pi_{\varphi} =\{ \pi(1)=\pi_{\rm gen} [\chi], \, \pi(\omega), \, \pi(\omega^2)\},
\] 
where $\pi_{\rm gen} [\chi]$ is as defined in Proposition \ref{P:PSP}(ii) (third b.p.).
If the character $\chi$ is unramified,  then $\pi(\omega)=\pi_{\rm sc}[\omega]$ and  $\pi(\omega^2)=\pi_{\rm sc}[\omega^2]$ are induced from a 
cuspidal unipotent representations of $G_2(\mathbb F_q)$, denoted by $G_2[\omega]$ and $G_2[\omega^2]$, inflated to a hyperspecial maximal compact group.
\vskip 5pt 
\item $\phi(W_F)= S_3$. Then $r \circ \phi$  corresponds to a supercuspidal representation $\tau$ of $\GL_2$ (where we recall that  $r$ denotes the two-dimensional irreducible representation of $S_3$). In this case $A_{\varphi}$ is trivial and 
\[
\Pi_{\varphi} =\{ \pi_{\rm gen}[\tau]\},  \]
where  $\pi_{\rm gen}[\tau]$ is as defined in Proposition \ref{P:PSQ}(i) (second b.p.).
\end{itemize} 

\vskip 5pt 
\item $\varphi(\SL_2)=\SL_{2,s}$, a short root $\SL_2$. The centralizer of $\SL_{2,s}$ in $G_2$ is $\SL_{2,l}$, a long root $\SL_2$. Then $\varphi$ gives, 
 by restriction, a map form 
the Weil group $\phi: W_F \rightarrow \SL_{2,l}$, that corresponds to supercuspidal representation $\tau$ of $\GL_2$ with the trivial central character 
(and hence $\tau\cong \tau^{\vee}$). In this case $A_{\varphi}=\mu_2$, and 
\[
\Pi_{\varphi} =\{ \pi(1)= \delta_P(\tau), \, \pi(-1)\},
\] 
where $ \delta_P(\tau)$ is as defined in Proposition \ref{P:PSP}(i) (first b.p.) and $\pi(-1)$ is supercuspidal.
\vskip 5pt 

\item $\varphi(\SL_2)=\SL_{2,l}$, a long root $\SL_2$. The centralizer of $\SL_{2,l}$ in $G_2$ is $\SL_{2,s}$, a long root $\SL_2$. Then $\varphi$ gives, 
 by restriction, a map form 
the Weil group $\phi: W_F \rightarrow \SL_{2,s}$, that corresponds to supercuspidal representation $\tau$ of $\GL_2$ with the trivial central character 
(and hence $\tau\cong \tau^{\vee}$). In this case $A_{\varphi}=\mu_2$, and 
\[
\Pi_{\varphi} =\{ \pi(1)= \delta_Q(\tau), \, \pi(-1)\},
\] 
where $ \delta_Q(\tau)$ is as defined in Proposition \ref{P:PSQ}(i) (first b.p.)  and $\pi(-1)$ is supercuspidal.
\vskip 5pt 
\item $\varphi(SL_2)= 1$. Then $\varphi: W_F \rightarrow G_2(\mathbb C)$ gives rise to an L-packet consisting entirely of supercuspidal representations of $G_2$.
\end{enumerate} 
\vskip 5pt 

There has been some work towards the above conjectural LLC for $G_2$, most notably \cite{SWe} and \cite{HKT}.  At the moment, we simply wish to point out that all the noncuspidal discrete series representations are fully accounted for by the above classification scheme.

\vskip 10pt

\subsection{\bf Local Fourier coefficients.}
It will be useful to consider the twisted Jacquet modules of a representation $\pi$ of $G_2$ along the unipotent radical $N$ of $P$.  
The $M$-orbits of 1-dimensional characters of $N$ are naturally indexed by cubic $F$-algebras, with the generic orbits corresponding to \'etale cubic $F$-algebras.  
For any such cubic $F$-algebra $E$, we shall write $\psi_E$ for a character of $N$ in the corresponding $M$-orbit.  Then one may consider $\pi_{N, \psi_E}$.  
In particular, we note: 
\vskip 5pt

\begin{lemma}
For any irreducible, nongeneric, infinite dimensional  representation $\pi$ of $G_2$, there exists an  \'etale cubic $F$-algebra $E$ such that $\pi_{N,\psi_E}  \ne 0$.  
Moreover,  $\pi_{N,\psi_E}$ is finite-dimensional for any \'etale $E$.  
\end{lemma}

\begin{proof} Since $\pi$ is not Whittaker generic, and not the trivial representation, 
its wave-front set is supported on subregular nilpotent orbits, which are paramterized by cubic \'etale algebras, see \cite{HMS} and \cite{LS}. The non-vanishing of
$\pi_{N,\psi_E}$ for some $E$ and the finite-dimensionality of these spaces is the main result of \cite{MW} and \cite{Va}. 
\end{proof}

\vskip 10pt

\section{\bf Exceptional Dual Pairs}
In this section, we briefly describe the dual pairs  which intervene in this paper and some structural results which will be important in the study of the associated theta correspondences.

 \subsection{\bf The group $M_J$.} 
 Let $J$ be a Freudenthal-Jordan $F$-algebra \cite{KMRT}. 
 The algebra $J$  comes equipped with a cubic norm form $N_J$, and we let
 \[  M_J = \{ g \in \GL(J):   N_J \circ g  = N_J \}.   \]
  It contains the automorphism group $\Aut(J)$ as a subgroup. 
   Now we consider the $F$-vector space
 \[ \mathfrak{g}_J =   \mathfrak{sl}_3 \oplus Lie(M_J) \oplus (F^3 \otimes J) \oplus (F^3 \otimes J)^* \]
 Then $\mathfrak{g}_J$ can be given the structure of a simple  exceptional Lie algebra (see, for example, \cite{GS05}). We have the following cases of interest: 
 \vskip 5pt

\begin{center}
\begin{tabular}{|c|c|c|c|c|}
\hline  
$\dim J$  & $1$  & $3$ & $9$ & $15$\\
\hline 
$\mathfrak{g}_J$ & $G_2$ & $D_4$  & $E_6$ & $E_7$ \\
\hline
\end{tabular}
\end{center}
\vskip 10pt
 
 We observe:
 \vskip 5pt
 
 \begin{itemize}
 \item  If $\dim J=3$, then $J$ is a cubic \' etale $F$-algebra $E$.
 \vskip 5pt
 
 \item If $\dim J=9$ ,then $J$ corresponds to a pair $(B_K, \tau)$ where $B_K$ is a central simple algebra over an \' etale quadratic $F$-algebra $K$ 
and $\iota$ is an involution of the second kind. Thus, $J =  B_K^{\iota}$ is the subspace of $\iota$-symmetric elements. The split version of this algebra is 
$M_3(F^2)=M_3(F)\oplus M_3(F)$ where $\iota$ permutes two factors. Thus $J=M_3(F)$, and $\Aut(M_3) = \PGL_3 \rtimes \Z/2\Z$. 

\vskip 5pt

\item If $\dim J=15$, then $J$ is $H_3(B_F)$  is the space of all 
$3\times 3$ hermitian-symmetric matrices, where $B_F$ is a quaternion algebra over $F$. The split version is when $B_F=M_2(F)$.  
\end{itemize}

\vskip 5pt 
Let $G_J$ be the identity component of $\Aut(\mathfrak g_J)$. If $\dim J=9$ then 
 \[  \begin{CD}
 1 @>>> G_J @>>> \Aut(\mathfrak g_J) @>>> \Z/2\Z @>>> 1. \end{CD} \]
 This short exact sequence need not split  in general.
 \vskip 5pt

 \subsection{\bf Dual pair $G_2 \times \Aut(J)$.}
 We can now describe some dual pairs in $G_J$ or rather in $\Aut(\mathfrak g_J)$. It will be easier to do this on the level of Lie algebras. 
 \vskip 5pt
 
 The centralizer of $\Aut(J)$ in $\mathfrak g_J$ is 
 \[  \mathfrak{sl}_3 \oplus F^3 \otimes 1_J \oplus F^3 \otimes 1_J \]
 which one recognizes to be $\mathfrak g_F$ (i.e. taking $J=F$). Thus this is a Lie subalgebra of type $G_2$, and  we have a dual pair
 \[  G_2 \times \Aut(J) \subset \Aut(G_J).  \]
If $\dim J=9$,  we recall that $\Aut(J)$ sits in a short exact sequence
\[  \begin{CD}
1 @>>> \Aut(J)^0 @>>>\Aut(J) @>>> \Z/2\Z @>>> 1. \end{CD} \]
If $J$ is associated to a pair $(B_K, \tau)$, then $\Aut(J)^0 = PGU(B_K, \tau)$ is an adjoint group of type $A_2$. 
\vskip 5pt

\subsection{\bf Dual pair $\Aut(i:E \rightarrow J)  \times G_E$.}

Now we fix an embedding $i: E \longrightarrow J$, where $E$ is a cubic \'etale $F$-algebra. 
 We have the subgroup
\[   \Aut(i: E \rightarrow J) \subset \Aut(J).  \]
If $\dim J=9$ a detailed description of this group is in \cite{GS14}. Its identity component is a 2-dimensional torus. 
 The centralizer of  $\Aut(i: E \rightarrow J)$ in $\mathfrak g_J$ contains 
\[  \mathfrak{g}_E = \mathfrak{sl}_3 \oplus  \mathfrak{t}_E
\oplus F^3 \otimes E \oplus (F^3 \otimes E)^* \]
where $E \hookrightarrow J$ via $i$ and $\mathfrak t_E \cong E^0$ is the toral Lie subalgebra of trace 0 elements in $E$. 
 This Lie algebra is isomorphic to $Lie(G_E)$ (where $G_E$ is the simply connected quasi-split group $Spin_8^E$), and we have
the dual pair 
 \[ \Aut(i: E \rightarrow J) \times  G_E \longrightarrow \Aut(G_J).\]
 Note that this map is not injective.
 
 \vskip 5pt
 
 \subsection{\bf A see-saw diagram} 
  The two dual pairs we described above fit together into a see-saw diagram:
 \vskip 10pt

\begin{equation} \label{D:seesaw}
 \xymatrix{
 G_E:=  \Spin_4^E  \ar@{-}[dr] \ar@{-}[d] & \Aut(J) =: H_J
     \ar@{-}[d] \\
  G_2 \ar@{-}[ur] &   \Aut(i: E \rightarrow J) =: H_{J,E}}
\end{equation}
in ${\rm Aut}(G_J)$.  The various $J$'s of interest in this paper, and the corresponding groups  $H_J = \Aut(J)$ and $H_{J,E} = {\rm Aut}(i: E \rightarrow J)$ are given in the table below.
 
\begin{center}
\begin{tabular}{|c|c|c|c|}
\hline  
$J$  & $D^+$  & $M_3(F)^+$ & $H_3(M_2(F))$ \\
\hline 
$H_J$ & $PD^{\times}$  & $\PGL_3 \rtimes \Z/2\Z$ & $\PGSp_6$ \\
\hline
$H_{J,E}$ & $PE^{\times}$ & $PE^{\times} \rtimes \Z/2\Z$  & $\SL_2(E)/ \mu_2$ \\
\hline
\end{tabular}
\end{center}
\vskip 10pt
Here, note that $D^+$ denotes the Jordan algebra associated to a cubic division $F$-algebra $D$, in which case $E$ is necessarily a field.

\vskip 10pt

\section{\bf The See-Saw Argument}
In this section, we shall consider the see-saw identity arising from the seesaw diagram (\ref{D:seesaw}) and pursue some of its consequences.

\vskip 5pt

\subsection{\bf See-saw identity.} 
Suppose that $\pi \in {\rm Irr}(G_2)$. Then we have the see-saw identity associated with the seesaw (\ref{D:seesaw}):
\begin{equation} \label{E:seesaw}
  \Hom_{H_{J,E}}(\Theta(\pi), \C)  \cong \Hom_{G_2} (R_J(E),   \pi) \end{equation}
where 
\[  R_J(E)  :=  \Theta(1)  \]
is the big theta lift of the trivial representation of $H_{J,E}$.  To make use of this see-saw identity, we need to understand the representation $R_J(E)$ of $\Spin_8^E$. This has been studied in \cite{GS15} and we recall the relevant results there.
\vskip 5pt

\subsection{\bf Degenerate principal series of $\Spin_8^E$.} 
 Let $P_E = M_E \cdot N_E \subset \Spin_8^E$ be the Heisenberg parabolic subgroup, so that its Levi factor is
 \[  M_E  \cong  \GL_2(E)^0. \]
 Then the determinant map defines an algebraic  character  $M_E \rightarrow \mathbb{G}_m$ which is a basis element of $\Hom(M_E, \mathbb{G}_m)$. We may consider the degenerate principal series representation
 \[  I_E(s)  =  {\rm Ind}_{P_E}^{\Spin_8^E}   |\det|^s.  \]
 In \cite{GS15, S},   the module structure of this family of degenerate principal series representation has been determined. In particular, we have:
 \vskip 5pt
 
 \begin{prop} \label{P:deg-ps}
 \[  R_J(E)  \hookrightarrow I_E (s_J)  \]
 where  
 \[  s_J = \begin{cases}
 -1/2, \text{  if $J= D^+$ or $M_3(F)^+$;} \\
 1/2, \text{  if $J = H_3(M_2(F))$.} \end{cases} \]
 The representation $I_E(1/2)$ has length $3$ when $E$ is a field and has length $2$ otherwise.
 More precisely, it has a unique irreducible submodule $V$ with quotient isomorphic to $ R_{M_3(F)}(E) \oplus R_D(E)$ (where $R_D(E)$ is interpreted to be $0$ when $E$ is not a field). Indeed, one has the short exact sequence:
  \[  \begin{CD}
 0 @>>> R_{H_3(M_2(F))}(E)  @>>> I_E(1/2) @>>>  R_D(E) @>>> 0. \end{CD} \]
 and
 \[ \begin{CD}
 0 @>>> V @>>>   R_{H_3(M_2(F))}(E)   @>>> R_{M_3(F)}(E)  @>>> 0. \end{CD} \]
 In particular, when $E$ is not a field, $I_E(1/2) =  R_{H_3(M_2(F))}(E)$.
 \end{prop}
 \vskip 5pt
  As a consequence of the above discussion, we see that it is useful to understand the Hom space
  \[  \Hom_{G_2}(I_E(s), \pi)  \quad \text{ for $\pi \in {\rm Irr}(G_2)$.} \]
We shall study this in two ways.
\vskip 10pt

 \subsection{\bf Vanishing of an ${\rm Ext}^1$.} 
 In view of the proposition, we see that there is an exact sequence
 \[  \begin{CD}
 0 @>>>  \Hom_{G_2}( R_D(E), \pi)  @>>>  \Hom_{G_2}(I_E(1/2), \pi) @>>>  \Hom_{G_2}( R_{H_3(M_2(F))}(E), \pi) \\
  @. @.   @. @VVV  \\
 @. @. @.   {\rm Ext}^1_{G_2}( R_D(E),  \pi).
 \end{CD} \]
 Now we have the following useful fact:
 \vskip 5pt
 
 \begin{prop}  \label{P:temp-sum}
 Assume that $E$ is a field. 
 If $\pi \in {\rm Irr}(G_2)$ is tempered or has cuspidal support different from  $\pi_{\rm deg}[1]$, then 
 \[   {\rm Ext}^1_{G_2}( R_D(E),  \pi)  = 0, \]
 so that one has a short exact sequence
  \[  \begin{CD}
 0 \rightarrow  \Hom_{G_2}( R_D(E), \pi)  @>>>  \Hom_{G_2}(I_E(1/2), \pi)  @>>> \Hom_{G_2}( R_{H_3(M_2(F))}(E), \pi) \rightarrow 0. \end{CD} \]
% In particular, the only $\pi$'s which are left out are the nontempered constituents of $I_P(1/2, 1)$ and $I_P(1/2, {\rm st})$, namely the three representations $J_P(1/2, {\rm st})$, $J_Q(1/2, {\rm st})$ and $J_Q(1, \pi(1,1))$. 
 \end{prop}
 \vskip 5pt
 
 \begin{proof}
 One needs to understand $R_D(E)$ as a representation of $G_2$, and this is essentially done in \cite{Sa}, where the dual pair correspondence for $PD^{\times} \times G_2$ was studied.  We shall recall the results of \cite{Sa} in greater detail later on. At this point, we simply note that 
 as a representation of $G_2$, $R_D(E)$ is the direct sum of a supercuspidal representation (of infinite length) and the irreducible discrete series representation
  $\pi_{\rm deg}[1]$, which is a constituent of $I_Q(1/2,{\rm st})$.  
 From this, the vanishing of   ${\rm Ext}^1_{G_2}( R_D(E),  \pi)$ for those $\pi$ with different cuspidal support from  $\pi_{\rm deg}[1]$ follows immediately. On the other hand, if $\pi$ is tempered, then one also has  ${\rm Ext}^1(\pi_{\rm deg}[1], \pi) =0$ since discrete series representations are projective in the category of  tempered representations.
 \end{proof}

 \vskip 5pt
 
 \subsection{\bf $I_E(s)$ as $G_2$-module.}
 On the other hand, we may understand the restriction of $I_E(s)$ to $G_2$ using Mackey theory.
   The following is a key technical result:
 \vskip 5pt
 
 \begin{prop}  \label{P:Homspace}
 As a representation of $G_2$, $I_E(s)$ admits an equivariant filtration
 \[  0  \subset I_0 \subset I_1 \subset I_2 \subset I_3 \subset I_4  \]
 with successive quotients described as follows:
 \vskip 5pt
 
 \begin{itemize}
 \item $I_0  \cong  {\rm ind}_N^{G_2}  \bar\psi_E $; 
 \vskip 5pt
 
 \item $J_1 := I_1/I_0 \cong   I_P(\frac{s}{2} + \frac{1}{4}, C^{\infty}_c(\PGL_2))$.
 \vskip 5pt
 
 \item $J_2:= I_2/I_1  \cong  m_E \cdot I_P(\frac{s}{2} + \frac{1}{4}, {\rm ind}_N^{\PGL_2} \psi)$
 
 \vskip 5pt
 
 \item $J_3:= I_3/I_2  \cong  m_E \cdot  I_Q(s+1)$.
 
 \vskip 5pt
 
 \item $J_4:= I_4/  I_3  \cong I_P(s+1)$.  
 \end{itemize}
 Here,
 \[  m_E  = \begin{cases} 
 3, \text{  if $E = F^3$;} \\
 1, \text{  if $E= F \times K$;} \\
 0, \text{  if $E$ is a field.} \end{cases} \]
 \end{prop}
 
 \vskip 5pt
The proposition implies that one has a short exact sequence
 \[
\begin{CD}
0 @>>>  {\rm ind}_N^{G_2}  \bar\psi_E @>>> I_E(s) @>>> \Sigma_E(s) @>>> 0. \end{CD}, \]
from which one deduces an exact sequence:
\[  \begin{CD}
0 \rightarrow  \Hom_{G_2}(\Sigma_E(s), \pi)  \rightarrow \Hom_{G_2}(I_E(s), \pi) \rightarrow \Hom_N(\pi^{\vee}, \psi_E) \rightarrow  {\rm Ext}^1_{G_2}(\Sigma_E(s),  \pi).
\end{CD} \]

 \vskip 5pt
 We now specialize to $s=1/2$, where we need to be more precise.

 \vskip 5pt
 
  \begin{prop}  \label{P:bij}
Suppose that $\pi \in {\rm Irr}(G_2)$ is tempered or has cuspidal support along $Q$. Then
 \[   \Hom_{G_2}(I_E(1/2), \pi)   \cong \Hom_N(\pi^{\vee}, \psi_E). \]
 \vskip 5pt
  \end{prop}
   \begin{proof}
   We need to prove 
   \[ \Hom_{G_2}(\Sigma_E(1/2), \pi) = 0 = {\rm Ext}^1_{G_2}(\Sigma_E(1/2), \pi).\] 
  To that end, it suffices to prove the following lemma: 
   \vskip 5pt
   
 \begin{lemma}
For all $i$ and $j$,   ${\rm Ext}^i_{G_2}(J_j, \pi)  = 0$, with $\pi$ as in Proposition \ref{P:bij}.
 \end{lemma}
 
 \begin{proof}
 Consider $J_1$ firstly. By the Frobenius reciprocity, 
 \[  {\rm Ext}^i_{G_2}(J_1, \pi) = {\rm Ext}^i_M( |\det|^{1/2}\cdot  C^{\infty}_c(\PGL_2),  r_{\overline{P}}(\pi)).   \]
 Since $\pi$ is tempered, the center of $M=\GL_2$ acts on  $R_{\overline{P}}(\pi))$ by characters $\chi$ such that 
 $|\chi(z)| =|z|^t$ where $t\leq 0$. On the other hand, the center of $M$ acts on $|\det|^{1/2}\cdot  C^{\infty}_c(\PGL_2)$ by 
 $|z|$. Thus the right hand side is 0. The other cases are dealt with in the same way.
 
 \end{proof}
 \vskip 5pt
 
 This completes the proof of Proposition \ref{P:bij}.
 \end{proof}
 \vskip 10pt
 
 \section{\bf Dichotomy}  \label{S:dichotomy}
The goal of this section is to prove the following theorem:
\vskip 5pt
  \begin{thm}  \label{T:dicho}
 For any  representation $\pi \in {\rm Irr}(G_2)$,   $\pi$ has nonzero theta lift to exactly one of $PD^{\times}$ or $\PGSp_6$.  
 \end{thm}

\vskip 5pt

To prove this dichotomy theorem, we need some  preliminary results which are consequences of the computation of the Jacquet modules of the minimal representation $\Pi_J$ with respect to the various maximal parabolic subgroups of $G_2$ and $H_J$. The required Jacquet module computations were carried out in \cite{Sa} when $H_J = PD^{\times}$ and in \cite{MS, GS04} when $H_J = \PGL_3$. For $H_J = \PGSp_6$, the Jacquet module computations for some parabolic subgroups were carried out in \cite{MS}. The remaining ones will be done in \S \ref{S:jacquet} and some implications of these computations are discussed in \S \ref{S:conseq}. We note that \S \ref{S:jacquet} is a self-contained section independent of the rest of this paper. Hence, we first record some results
from \S \ref{S:jacquet}-\ref{S:conseq} and the earlier references  \cite{Sa, MS, GS04} that we will use.

\vskip 5pt

\subsection{\bf Consequences of Jacquet module computations.}
We first note:

\vskip 5pt 
\begin{lemma} \label{L:basic}  
Consider the theta correspondence for $G_2 \times H_J$ for the 3 cases of $J$. 
\vskip 5pt

(i) Let $\pi \in {\rm Irr}(G_2)$ and write
\[  \Theta_J(\pi) =  \Theta_J(\pi)_{c} \oplus \Theta_J(\pi)_{nc} \]
as a sum of its cuspidal and noncuspidal components. Then $\Theta_J(\pi)_{nc}$ has finite length. In particular, if $\Theta_J(\pi) \ne 0$, then it has an irreducible quotient.
 
\vskip 5pt

(ii) Likewise, let $\tau \in {\rm Irr}(H_J)$ and write 
\[  \Theta_J(\tau) = \Theta_J(\tau)_c  \oplus \Theta_J(\tau)_{nc}.  \]
Then $\Theta_J(\tau)_{nc}$ has finite length. In particular, if $\Theta_J(\tau) \ne 0$, then it has an irreducible quotient.
\end{lemma} 

\begin{proof}  
 The case of $J=D^+$ follows from results of \cite{Sa}, whereas that for $J= M_3(F)^+$ follows from \cite{GS04}. 
The case of $J = H_3(M_2(F))$ is shown in \S \ref{S:conseq}, based on the results of \S \ref{S:jacquet}. 
  \end{proof}
\vskip 5pt

In fact, the Jacquet module computations allow one to determine the theta lift of non-tempered representations explicitly (see Theorem \ref{T:nont}). We  simply note the following here:
\vskip 5pt

\begin{lemma} \label{L:basic2}
(i)  Let $\pi \in {\rm Irr}(G_2)$ and $\tau \in {\rm Irr}(H_J)$  be  such that $\pi \otimes \tau$ is a quotient of $\Pi_J$. Then
\[  \text{$\pi$ is tempered} \Longleftrightarrow \text{$\tau$ is tempered}. \]

\vskip 5pt

(ii) Let $\pi \in{\rm Irr}(G_2)$ be non-tempered. Then $\pi$ has nonzero theta lifting to $\PGSp_6$.
\end{lemma}
\vskip 5pt

\begin{proof}
For $H_J = PD^{\times}$ or $\PGL_3 \rtimes \Z/2\Z$, the desired results have been verified in \cite{Sa} and \cite{GS04} respectively. For the case 
when $H_J = \PGSp_6$, this is shown in Theorem \ref{T:nont} in \S \ref{S:conseq}. 
\end{proof}

 \vskip 5pt
 
 \subsection{\bf Reduction to non-generic tempered case.}
With the above inputs in place,  we can now begin the proof of the dichotomy theorem.
 We note:
 \vskip 5pt
 
 \begin{itemize}
 \item The dichotmy theorem holds for nontempered $\pi$. Indeed, if $\pi$ is non-tempered, then Lemma \ref{L:basic2}(ii) says that $\pi$ has nonzero theta lift to $\PGSp_6$, whereas \cite{Sa} shows that $\pi$ has zero theta lift to $PD^{\times}$.
 \vskip 5pt
 
 \item The dichotmy theorem holds for generic $\pi$.  Indeed,  it was shown in \cite[Cor. 20]{GS04} that a generic $\pi$ has nonzero theta lift to $\PGSp_6$  (see also Cor. \ref{C:GS-Whit} below) and it was shown in \cite{Sa} that $\pi$ has  zero theta lift to $PD^{\times}$. 
 \end{itemize}

 Thus, to prove the dichotomy theorem,  it remains to deal with non-generic tempered $\pi$.   
 
   \vskip 5pt

 \subsection{\bf Weak dichotomy.}
  We first prove that  a non-generic tempered $\pi$ has nonzero theta lift to one of $PD^{\times}$ or $\PGSp_6$.
   Since $\pi$ is non-generic and infinite-dimensional,  
    there exists an \'etale cubic $F$-algebra $E$ such that $\Hom_N(\pi^{\vee}, \psi_E) \ne 0$.  
    By Proposition \ref{P:bij}, we have an isomorphism 
   \[  \Hom_{G_2}(I_E(1/2), \pi) \twoheadrightarrow \Hom_N(\pi^{\vee}, \psi_E) \ \ne 0.  \]
 This implies, by Proposition \ref{P:temp-sum}, 
  \[  \Hom_{G_2}(R_D(E), \pi) \ne 0 \quad \text{or} \quad \Hom_{G_2}(R_{H_3(M_2(F)}(E)  , \pi) \ne 0. \]
 By the see-saw identity  (\ref{E:seesaw}), we deduce that 
 \[     \Hom_{H_E}( \Theta_J(\pi),  \C) \ne 0 \]
 for $J = D^+$ or $H_3(M_2(F))$. In particular, $\Theta_J(\pi) \ne 0$ for  $J = D^+$ or $H_3(M_2(F))$.
We have thus verified that $\pi$ has nonzero theta lift to at least one of $PD^{\times}$ or $\PGSp_6$.

\vskip 10pt

\subsection{\bf Curious chain of containments.}
 It remains to show that a nongeneric tempered $\pi$ cannot lift to both $PD^{\times}$ and $\PGSp_6$.
  Let $\bar \pi\cong \pi \otimes_{\mathbb C} \mathbb C$ be the complex conjugate of $\pi$.  If $\pi$ is unitarizable (e.g. if $\pi$ is tempered), then $\bar\pi \cong \pi^{\vee}$. 
 Thus 
 \[ 
 \pi_{\bar \psi_E} \cong \bar \pi_{ \psi_E} \cong( \pi^{\vee})_{\psi_E} 
 \] 
 where, in the second isomorphism, we assume that $\pi$ is unitarizable. 
 Since the minimal representation $\Pi_J$ used in this paper is defined over $\mathbb R$, 
 we have a canonical isomorphism $\bar \Pi_J \cong \Pi_J$. It follows that $\Theta_J(\bar \pi)$ is the complex conjugate of $\Theta_J(\pi)$.

 \vskip 5pt
 
We shall make use of the curious chain of containment given in the following lemma; this is the first instance of the game of ping-pong with periods discussed at the end of the introduction.
\vskip 5pt

\begin{lemma}  \label{L:curious}
Let $\pi \in {\rm Irr}(G_2)$ be tempered. 
% If $E$ is a field,  assume that $\pi$ is  different from the 3 exceptional representations in Proposition \ref{P:temp-sum}.
For $J = D^+$, $M_3(F)$  or $H_3(M_2(F))$, let  $\tau \in {\rm Irr}(H_J)$  be  tempered and such that 
\[  \Hom_{G_2 \times H_J}(\Pi_J,   \pi \boxtimes \tau) \ne 0. \]
Then we have the following natural inclusions 
 \begin{align} \label{E:chain}
  \Hom_N(\pi, \psi_E)  \subseteq \Hom_N(\Theta(\tau), \psi_E)  \cong \Hom_{H_{J,E}}(\tau^{\vee},\C)  \subseteq
 \Hom_{H_{J,E}} (\Theta(\pi^{\vee}), \C)   \cong  \Hom_{G_2}(R_J(E),  \pi^{\vee} ). \notag 
 \end{align}
 If one of these spaces is finite-dimensional, then all inclusions are isomorphisms.
 \end{lemma}
 \vskip 5pt
 \begin{proof} The first inclusion arises from $\Theta(\tau) \twoheadrightarrow \pi$. 
 The second follows from
 \[ 
  \Hom_N(\Theta(\tau), \psi_E)  \cong \Hom_{H_{J}}((\Pi_{J})_{N,\psi_E},  \tau) 
 \] 
 combined with (see Lemma 2.9 on page 213 in \cite{GrS}) 
 \[ 
 (\Pi_{J})_{N,\psi_E} \cong \mathrm{ind}_{H_{J,E}}^{H_J} (1)
 \] 
  and the Frobenius reciprocity. For the third, observe that $ \Theta_J(\bar \pi)$ is the complex conjugate of $\Theta_J(\pi)$. 
 Since $\bar\pi \cong \pi^{\vee}$ and $\bar \tau\cong \tau^{\vee}$ and we have 
$\Theta(\pi^{\vee}) \twoheadrightarrow \tau^{\vee}$. The fourth is the see-saw identity (\ref{E:seesaw}).

\vskip 5pt

 If any of the spaces is finite-dimensional,  then $ \Hom_N(\pi, \psi_E)$ is finite-dimensional. 
If this space is finite dimensional then,  since $\pi$ is tempered, by Propositions  \ref{P:temp-sum}  and \ref{P:bij}, one has
\begin{equation} \label{E:chain1}
 \dim  \Hom_{G_2}(R_J(E),  \pi^{\vee}) \leq \dim \Hom_{G_2}(I_E(1/2), \pi^{\vee} ) = \dim \Hom_N(\pi, \psi_E). \end{equation}
 It follows that all spaces have the same dimension and the lemma is proved. 
 \end{proof}
\vskip 5pt

\subsection{\bf Conclusion of proof.}
Using the lemma, we can now conclude the proof of Theorem \ref{T:dicho}. 
\vskip 5pt

Assume $\pi$ is tempered nongeneric and has nonzero theta lift to $PD^{\times}$.   Since $PD^{\times}$ is compact, 
one can find $\tau \in {\rm Irr}(PD^{\times})$ such that $\tau$ is an irreducible quotient of $\Theta_D(\pi)$. Moreover $\tau$ is tempered.
Choose $E$ so that $\Hom_N(\pi, \psi_E) \ne 0$. We may now apply Lemma \ref{L:curious} with the chosen $\tau$ and $E$ to deduce  that
\[  d: = \dim \Hom_{G_2}(I_E(1/2), \pi) = \dim \Hom_{G_2}(R_D(E), \pi) = \dim \Hom_N(\pi, \psi_E) \ne 0 \]
\vskip 5pt

Similarly, if $\pi$ has nonzero theta lift to $\PGSp_6$,  then  we may find a tempered $\tau \in {\rm Irr}(\PGSp_6)$ such that $\tau$ is an irreducible quotient of  $\Theta(\pi)$ (by Lemma \ref{L:basic} and Lemma \ref{L:basic2}(i)). With $E$ as above, an application of Lemma \ref{L:curious} shows that
\[   d = \dim \Hom_{G_2}(I_E(1/2), \pi) = \dim \Hom_{G_2}(R_{H_3(M_2(F))}(E), \pi) = \dim \Hom_N(\pi, \psi_E) \ne 0 \]
 Moreover, since all these dimensions are finite, one deduces by Proposition \ref{P:temp-sum} that
 \[ d=  \dim \Hom_{G_2}(I_E(1/2), \pi) =  \dim \Hom_{G_2}(R_D(E), \pi) +  \dim \Hom_{G_2}(R_{H_3(M_2(F))}(E), \pi)  =2d. \]
This gives the desired contradiction and completes the proof of Theorem \ref{T:dicho}.
\vskip 10pt

\subsection{\bf Uniqueness results}
As further applications of Lemma \ref{L:curious}, we may now derive two multiplicity one statements which will play a key role in the reminder of the paper. These statements are the first steps towards the proof of the Howe duality theorem.

\vskip 5pt 

\begin{prop} \label{P:one-to-one I} 
Let $\tau \in {\rm Irr}(H_J)$ be tempered. Let $\pi \in {\rm Irr}(G_2)$ be a tempered, non-generic quotient of $\Theta_J(\tau)$.  Then $\pi$ is the unique 
irreducible tempered quotient of $\Theta_J(\tau)$. 
\end{prop}  
\begin{proof} Since $\pi$ is non-generic, there exists $E$ such that $d= \dim\Hom_N(\pi, \psi_E)$ is finite and non-zero. By Lemma \ref{L:curious}, 
$\dim \Hom_N(\Theta(\tau), \psi_E) =d$. Let $\pi'$ be another tempered  irreducible   quotient of $\Theta_J(\tau)$. We apply Lemma \ref{L:curious} to 
$\pi'$  and $\tau$ to deduce that $\dim \Hom_N(\pi', \psi_E)=d$. Since $\pi\oplus \pi'$ is a quotient of $\Theta_J(\tau)$, 
\[ 
d= \dim \Hom_N(\Theta_J(\tau), \psi_E) \geq \dim \Hom_N(\pi, \psi_E) + \dim \Hom_N(\pi', \psi_E) =2d, 
\] 
a contradiction. 
\end{proof} 
\vskip 5pt

%We note that by Lemma \ref{L:basic2}(i), we know that, in the context of the proposition,  any irreducible quotient of $\Theta_J(\tau)$ is in fact tempered.   The following is a companion statement.
\vskip 5pt 

\begin{prop} \label{P:one-to-one II}  
Let $\pi \in {\rm Irr}(G_2)$ be tempered and non-generic. 
Then $\Theta_J(\pi)$  cannot have two tempered irreducible quotients. In particular, the cuspidal representation $\Theta_J(\pi)_c$ is irreducible or $0$.
\end{prop}  
\begin{proof}  Let $\tau_1, \tau_2  \in {\rm Irr}(H_J)$, irreducible tempered,  such that $\Theta_J(\pi) \twoheadrightarrow \tau_1 \oplus \tau_2$. 
  Since $\pi$ is non-generic there exists $E$ such that $d= \dim\Hom_N(\pi^{\vee}, \psi_E)$ is finite and non-zero. By Lemma \ref{L:curious}, applied to 
$\pi^{\vee}$, $\tau_1^{\vee}$ and then to $\pi^{\vee}$, $\tau_2^{\vee}$, 
\[ 
d= \dim \Hom_{H_{J,E}}(\tau_1, \mathbb C) =\dim \Hom_{H_{J,E}}(\Theta_J(\pi), \mathbb C) =\dim \Hom_{H_{J,E}}(\tau_2, \mathbb C). 
\] 
 Since $\tau_1\oplus \tau_2$ is a quotient of $\Theta(\pi)$, 
\[ 
d= \dim \Hom_{H_{J,E}}(\Theta_J(\pi), \mathbb C) \geq \dim \Hom_{H_{J,E}}(\tau_1, \mathbb C) + \dim \Hom_{H_{J,E}}(\tau_2, \mathbb C) =2d, 
\] 
a contradiction. 
\end{proof} 
%Again, by Lemma \ref{L:basic2}(i), we know that, in the context of the proposition, any irreducible quotient of $\Theta_J(\pi)$ is tempered.
\vskip 5pt

Combining Proposition \ref{P:one-to-one I} and \ref{P:one-to-one II} with Lemmas \ref{L:basic}(i) and \ref{L:basic2}(i), we deduce the following corollary which may be considered as a first step towards the Howe duality theorem.
\vskip 5pt

\begin{cor}  \label{C:howe}
Let $\pi \in {\rm Irr}(G_2)$ be tempered and non-generic. Then $\Theta_J(\pi)$ has finite length. If $\Theta_J(\pi) \ne 0$, then it has a unique irreducible quotient $\theta(\pi)$ and $\theta(\pi)$ is tempered. Moreover, for $\pi_1,  \pi_2 \in {\rm Irr}(G_2)$ tempered and non-generic,
\[  \theta(\pi_1) \cong \theta(\pi_2) \Longrightarrow \pi_1 \cong \pi_2.\] 
\end{cor}

\begin{proof}
Writing $\Theta_J(\pi) = \Theta_J(\pi)_c \oplus \Theta_J(\pi)_{nc}$, Proposition \ref{P:one-to-one II} tells us that $\Theta_J(\pi)_c$ is irreducible or $0$, whereas Lemma \ref{L:basic}(i) 
tells us that $\Theta_J(\pi)_{nc}$ has finite length. Hence $\Theta_J(\pi)$ has finite length, so that its cosocle $\theta_J(\pi)$ is a finite sum of irreducible representations. Moreover, Lemma \ref{L:basic2}(i) says that $\theta_J(\pi)$ is tempered, and Proposition \ref{P:one-to-one II} then shows the irreducibility of $\theta_J(\pi)$ if it is nonzero.
The final implication now follows by Proposition \ref{P:one-to-one I}.
\end{proof}
\vskip 5pt

In the rest of the paper, we shall examine each of the 3 dual pairs $G_2 \times H_J$ in turn and complete the proof of the Howe duality conjecture.

\vskip 10pt

\section{\bf  Theta Correspondence for $PD^{\times} \times G_2$}
In this section, we discuss the theta correspondence for the dual pair $PD^{\times} \times G_2$. 
A preliminary study of this dual pair correspondence has been carried out by the second author in \cite{Sa}. We first recall the results established there.

 \vskip 5pt
 
Let $\Pi_D$ be the minimal representation of $PD^{\times} \times G_2$. Then we have
\[  \Pi_D  = \bigoplus_{\tau \in {\rm Irr}(PD^{\times})} \tau \boxtimes \Theta(\tau). \]
The following was shown in \cite{Sa}:
\vskip 5pt

\begin{prop}
(i) If $\tau = 1$ is the trivial representation of $PD^{\times}$, then 
\[ 
\Theta(1)  = \pi_{\rm deg}[1],
\] the unipotent discrete series representation introduced in Proposition \ref{P:PSP}(iii) (first bullet point).
\vskip 5pt

\noindent (ii) If $\tau$ is not the trivial representation, then $\Theta(\tau)$ is nongeneric supercuspidal of finite length (possibly zero).
\vskip 5pt

\noindent (iii) If $\tau=\chi$ is a nontrivial  unramified cubic character,    then 
\[ 
\Theta(\chi)=\pi_{\rm sc}[\chi] \text{ and }  \Theta(\chi^2)=\pi_{\rm sc}[\chi^2]
\]  
 the two depth 0 supercuspidal representations  introduced in \S \ref{SS:partial-LLC} (2) (third bullet point).
\vskip 5pt

\noindent (iv)  For each cubic field extension $E/F$, 
\[  \Hom_N(\Theta(\tau), \psi_E)  \cong \Hom_{PE^{\times}}(\tau, \C).\]
\end{prop}
\vskip 5pt

We can now easily extend the above results. More precisely, 
\vskip 5pt

\begin{thm}

\noindent 
(i) For any $\tau \in {\rm Irr}(PD^{\times})$, $\Theta(\tau)$ is an irreducible representation of $G_2$ if it is nonzero. 
\vskip 5pt

\noindent 
(ii) If $\tau_1,\tau_2 \in {\rm Irr}(PD^{\times})$ are such that $\Theta(\tau_1)\cong\Theta(\tau_2) \neq 0$, then $\tau_1\cong\tau_2$. 
\vskip 5pt

\noindent  (ii) If $p\neq 3$, then the map $\tau \mapsto \Theta(\tau)$ defines an injection
\[  {\rm Irr}(PD^{\times}) \hookrightarrow {\rm Irr}(G_2). \]
\vskip 5pt

Hence, the Howe duality theorem holds for $PD^{\times} \times G_2$, so that
\[\dim \Hom_{G_2}(\theta(\tau_1) , \theta(\tau_2)) \leq   \dim \Hom_{PD^{\times}}(\tau_1, \tau_2) \]
for any $\tau_1,\tau_2 \in {\rm Irr}(PD^{\times})$. In particular, for any $\pi \in {\rm Irr}(G_2)$, the representation $\Theta(\pi)$  of $PD^{\times}$ is irreducible or zero. 

\end{thm}
\begin{proof} The first two parts follow from Propositions \ref{P:one-to-one I} and \ref{P:one-to-one II}. As for (iii) we use 
\[  \Hom_N(\Theta(\tau), \psi_E)  \cong \Hom_{PE^{\times}}(\tau, \C), \]
so it suffices to show that there exists a field $E$ such that $\Hom_{PE^{\times}}(\tau, \C)\neq 0$. If $p\neq 3$, this was proved for  
all irreducible $\tau$ by \cite{LT}, Theorem 2.4. 

\end{proof} 

\vskip 5pt

%\noindent{\bf Remark}: Observe that all $\pi \in {\rm Irr}(G_2)$ which participate in the theta correspondence with $PD^{\times}$ are discrete series representations, and are in particular tempered.  Hence to extend Theorem \ref{T:dicho} to all representations, it suffices to show that all non tempered $\pi \in {\rm Irr}(G_2)$ has nonzero theta lift to $\PGSp_6$. 
%This will be accomplished in Section \ref{S:jacquet}. 

\vskip 10pt

\section{\bf Theta Correspondence for $(\PGL_3\rtimes \Z/2\Z) \times G_2$}

In this section, we consider the theta correspondence for dual pair  $(\PGL_3\rtimes \Z/2\Z) \times G_2$ and prove various results
analogous to those in the last section. In fact, the theta correspondence for $\PGL_3 \times G_2$ was almost completely studied in \cite{GS04}. 
But the treatment there ignores the outer automorphism group of $\PGL_3$; this is akin to working with special orthogonal groups instead of orthogonal groups in classical theta correspondence and is of course undesirable. Thus, we shall complete the results of \cite{GS04} in their natural setting here. We extend the minimal representation of 
$E_6$ to $E_6  \rtimes \Z/2\Z$ so that $\Z/2\Z$ fixes the spherical vector. 

\vskip 5pt

\subsection{\bf Representations of $H = \PGL_3 \rtimes \Z/2\Z$.}  We realize $\Z/2\Z$, acting on $\GL_3$ as a pinned automorphism preserving the standard pinning, i.e. acting via  
\[  A \mapsto w_0 \cdot {^t}A^{-1} \cdot w_0^{-1} \quad \text{with} \quad 
\left( \begin{array}{ccc}
 & & 1\\
 & -1&  \\
 1 & & 
\end{array} \right)
\]
Let $U \subset \GL_3$ be the maximal unipotent subgroup of upper triangular matrices and let $\psi$ be a $\Z/2\Z$-invariant Whittaker character of $U$. 
Then $\psi$ extends to two characters of $U\rtimes \Z/2\Z$. Let 
$\psi\otimes 1$ be the extension such that $\Z/2\Z$ acts trivially, and let $\psi\otimes {\rm sign}$ be the other extension. 
If $\tau \in {\rm Irr}(\PGL_3)$, then there are two possibilities:
\vskip 5pt

\begin{itemize}
\item if $\tau^{\vee} \ncong \tau$, then 
\[  \tau^+  := {\rm Ind}_{\PGL_3}^{H}\tau \cong  {\rm Ind}_{\PGL_3}^{H} \tau^{\vee}  \]
is irreducible. If $\tau$ is generic then 
\[  \dim\Hom_{U \rtimes \Z/2\Z}(\tau^+,   \psi \otimes 1) = \dim \Hom _{U \rtimes \Z/2\Z}(\tau^+,   \psi \otimes {\rm sign}) =1. \]

\vskip 5pt

\item if $\tau \cong \tau^{\vee}$, then $\tau$ has two extensions to $H$, which differ from each other by twisting with the unique quadratic character 
${\rm sign}: H \longrightarrow \langle \pm 1 \rangle$ of $H$. When $\tau$ is generic (for example when $\tau$ is tempered), we let $\tau^+$ denote the unique extension of $\tau$ such that
\[  \dim\Hom_{U \rtimes \Z/2\Z}(\tau^+,   \psi \otimes 1)   =1, \]
and let $\tau^-$ denote the other extension. 
\vskip 5pt
The only nongeneric and self-dual representations of $\PGL_3$ are the 
Langlands quotients $J_B(\mu)$, where $B=TU$ is the normalizer of $U$ and 
\[ \mu = \chi|-|^{1/2} \otimes 1 \otimes \chi|-|^{-1/2} 
\] 
is a character of $T$ such that $\chi^2=1$.  In this case, $\theta(\tau)$ is irreducible by \cite{GS04}, and we  define $\tau^+$ by setting 
$\theta(\tau^+)=\theta(\tau)$ and  $\theta(\tau^-)=0$. 
\end{itemize}
It follows from the above discussion that any irreducible representation of $H$ is self-dual.

\vskip 10pt

\subsection{\bf Whittaker models.}
The following lemma summarizes some basic computations.
\vskip 5pt

\begin{lemma} \label{L:twisted_W} 
Let $\Pi$ be the minimal representation of  split $E_6 \rtimes \Z/2\Z$. 
\vskip 5pt
\noindent (i)  Let $\psi_V : V \rightarrow \mathbb C^{\times}$  be a Whitaker character for $G_2$ (so $V$ is a maximal unipotent subgroup of $G_2$). Then
\[  \Pi_{V, \psi_V}  \cong {\rm ind}_{U \rtimes \Z/2\Z}^H \psi \otimes 1. \]
In particular, for any $\tau^{\epsilon} \in {\rm Irr}(H)$, 
\[  \Hom_{V}(\Theta(\tau^{\epsilon}),  \psi_V)  \cong \Hom_{U \rtimes \Z/2\Z}( \tau^{\epsilon}, \psi \otimes 1) 
%= \begin{cases} 
%\C, \text{ if $\epsilon = +$} \\
%0, \text{  if $\epsilon = -$.} \end{cases} 
\] 
\vskip 5pt

\noindent (ii) For any \'etale cubic $F$-algebra, we have:
\[  \Hom_N(\Theta(\tau^{\epsilon}), \psi_E) \cong \Hom_{PE^{\times} \rtimes \Z/2\Z}( \tau^{\epsilon}, \C). \]
\end{lemma}
\vskip 10pt

\subsection{\bf Our earlier results.}
The following is a simple combination of results  \cite{GS04} and the previous discussion: 
\vskip 5pt

\begin{thm} \label{T:fromGS} 
For $\tau \in {\rm Irr}(\PGL_3)$, let $\phi_{\tau}: WD_F \longrightarrow \SL_3(\C)$ denote the L-parameter of $\tau$. 
If $\tau$ is non-supercuspidal, then $\Theta(\tau)$ has finite length. If $\tau \ncong \tau^{\vee}$ is supercuspidal, then $\Theta(\tau)$ is irreducible supercuspidal.
(This covers all $\tau \in {\rm Irr}(\PGL_3)$ if $ p \ne 2$). In these cases, set $\theta(\tau^{\epsilon})$  to be the maximal semisimple quotient of $\Theta(\tau^{\epsilon})$ for $\epsilon = \pm$.
\vskip 5pt

More precisely,  we have:
\vskip 5pt

\noindent (i)  If $\tau \ne \tau^{\vee}$, then $\theta (\tau^+)$ is irreducible and nonzero. 
If $\tau$ is generic, or supercuspidal, or a discrete series representation, or tempered, so is $\theta(\tau^+)$. 
When $\tau$ is not supercuspidal, then $\theta(\tau^+)$ is not supercuspidal and its L-parameter   is obtained by composing $\phi_{\tau}$ with the inclusion 
$\SL_3(\mathbb C) \subset G_2(\mathbb C)$. 
\vskip 5pt

\noindent (ii)  If  $\tau = \tau^{\vee}$ and the parameter $\phi_{\tau}$ contains the trivial representation, then $\theta(\tau^-)=0$ and 
$\theta (\tau^+)$ is nonzero irreducible, non-discrete-series and its L-parameter is obtained by composing $\phi_{\tau}$ with the inclusion 
$\SL_3(\mathbb C) \subset G_2(\mathbb C)$. 

\vskip 5pt

\noindent(iii) If  $\tau = \tau^{\vee}$ and the parameter $\phi_{\tau}$ does not contain the trivial representation,  then we have the following cases: 
\begin{itemize} 
\item  $\tau = {\rm St}$, the Steinberg representation. Then 
\[  \theta({\rm St}^+) \oplus \theta({\rm St}^-)  = \pi_{\rm gen}[1] \oplus \pi_{\mathrm sc}[1] \]
where $\pi_{gen}[1]$ is the generic discrete series representation introduced in Proposition \ref{P:PSP}(ii) and $\pi_{sc}[1]$ is the depth 0 supercuspidal representation introduced in \S \ref{SS:partial-LLC} (2).
\vskip 5pt 
\item $\tau$ is a tempered representation induced from a supercuspidal representation  $\sigma\cong \sigma^{\vee}$  of 
$\GL_2$ with a non-trivial central character. Then 
\[
  \theta(\tau^+) \oplus \theta(\tau^-)  = {\rm Ind}_{P}^{G_2}(\sigma) =  {\rm Ind}_{P}^{G_2}(\sigma)_{\rm gen} \oplus {\rm Ind}_{P}^{G_2}(\sigma)_{\rm deg} 
  \]
\vskip 5pt 
\item $\tau$ is a tempered principal series induced from a triple of  non-trivial quadratic characters $(\chi_1, \chi_2,  \chi_3)$ such that 
$\chi_1 \cdot \chi_2 \cdot \chi_3=1$ then 
\[
  \theta(\tau^+) \oplus \theta(\tau^-)  = {\rm Ind}_B^{G_2}(\chi)= {\rm Ind}_B^{G_2}(\chi)_{\rm gen} \oplus  {\rm Ind}_B^{G_2}(\chi)_{\rm deg} 
 \]
 where $\chi$ is the quadratic character of $T$ determined by $(\chi_1,\chi_2,\chi_3)$ as in \S \ref{SS:PSB}.   
\vskip 5pt 
\item  $\tau$ is a self-dual supercuspidal representation (so $p=2$). Then  $\Theta(\tau^{\epsilon})$ is supercuspidal and 
\[
  \Theta(\tau^+) \oplus \Theta(\tau^-)  = \pi_{\rm gen} \oplus \pi_{\rm deg} 
  \]
where $\pi_{\rm gen}$ is a generic irreducible supercuspidal representation, while $\pi_{\rm deg}$ is a non-generic supercuspidal representation of unknown length. 
 \end{itemize} 
\end{thm}
\vskip 5pt

Observe that  the only case for which we do not know that $\Theta(\tau^{\epsilon})$ has finite length (and hence $\theta(\tau^{\epsilon})$ is defined) is when $\tau$ is a self-dual supercuspidal representation (so $p=2$). In this case,  however, the last bullet point states that $\Theta(\tau^{\epsilon})$ is supercuspidal and hence semisimple. Hence, even in this exceptional case, we may set $\theta(\tau^{\epsilon}) = \Theta(\tau^{\epsilon})$. Moreover, observe that if $\tau^{\epsilon}$ is nontempered, then $\theta(\tau^{\epsilon})$ is irreducible nontempered and is completely determined by Theorem \ref{T:fromGS}.  On the other hand, when $\tau^{\epsilon}$ is tempered, then so is every irreducible summand of 
$\theta(\tau^{\epsilon})$. In particular, the results highlighted in Lemma \ref{L:basic} and \ref{L:basic2} hold in this case.
\vskip 5pt

In the rest of the section, we shall complete the results above by completing the unresolved parts of the theorem.

\vskip 5pt

\subsection{\bf A miracle of Oberwolfach}  \label{SS:ober}
 Let $\tau \in {\rm Irr}(\PGL_3)$ be a self-dual supercuspidal  representation. 
The goal here is to show that $\Theta(\tau^-) \neq 0$.  Let $Q=LU$ be the 3-step maximal parabolic subgroup of $G_2$. 
Recall that the group $U$ has the  3-step filtration 
\[ 
U \supset [U,U] \supset Z_U 
\] 
where $Z_U $ is the 2-dimensional center of $U$ and $U/Z_U$ is a 3-dimensional Heisenberg group with the center $[U,U]/Z_U$. 
 Let $\psi$  be a non-trivial character of $[U,U]$, trivial on $Z_U$.  Then  $\Pi_{[U,U], \psi}$ is naturally a $(\PGL_3\rtimes \Z/2\Z) \times \SL_2$-module, 
 where $\SL_2= [L,L]$. In order to describe $\Pi_{[U,U], \psi}$,  we need some additional notation. 
 \vskip 5pt
 
 Consider the action of the group $\GL_2\rtimes \Z/2\Z$  on $M_2(F)$, with elements in $\GL_2(F)$ acting by conjugation and the nontrivial element of $\Z/2\Z$ acting via:
 \[  X \mapsto  \left( \begin{array}{cc}
 0 & 1 \\
 1 & 0 \end{array} \right) \cdot X^t \cdot \left( \begin{array}{cc}
 0 & 1 \\
 1 & 0 \end{array} \right).  \]
  This action  preserves the determinant (quadratic) form on $M_2(F)$ and descends 
to the quotient group 
\[  \PGL_2(F) \times \Z/2\Z  \cong  \{\pm 1\} \times \mathrm {SO}_3= \mathrm O_3. \]   
On the space $C_c(M_2(F))$, we have a Weil representation of $\mathrm O_3 \times {\SL}_2$, which we may regard as a representation of $\GL_2(F) \rtimes \Z/2\Z$. Then 
the following lemma follows by a standard computation:
\vskip 5pt

\begin{lemma}
We have an isomorphism of $(\PGL_3 \rtimes \Z/2\Z) \times \SL_2$-modules:
\[ 
\Pi_{[U,U], \psi}\cong \mathrm{ind}_{\GL_2\rtimes \Z/2\Z}^{\PGL_3\rtimes \Z/2\Z} (C_c(M_2(F))
\] 
where $\GL_2$ is embedded in $\PGL_3$ via
\[  \left( \begin{array}{cc}
 a & b \\
 c & d \end{array} \right)  \mapsto \left( \begin{array}{ccc}
a & & b \\
&1 &  \\
c & & d \end{array} \right). \]
\end{lemma}

\vskip 5pt
Using the lemma, we can now prove:
\vskip 5pt

\begin{prop} 
Let $\tau \in {\rm Irr}(\PGL_3)$ be a self-dual supercuspidal  representation. 
Then  $\Theta(\tau^-) \neq 0$.
\end{prop}  
\begin{proof}  It suffices to show that $\tau^-$ is a quotient of $\Pi_{[U,U], \psi}$, in fact we shall show that $\tau^-$  is a 
quotient of $\SL_2$-coinvaraints of $\Pi_{[U,U], \psi}$. 
Decompose  $M_2(F) = M^{\circ}(F) \oplus F$, where $M_2(F)^{\circ}$ is the subspace of trace 0 elements and $F$ is the center. 
Accordingly, the Weil representation of $\mathrm{O}_3  \times \widetilde{\SL}_2$ on $C_c(M_2(F))$ decomposes as a tensor product 
\[ 
C_c(M_2(F))=C_c(M^{\circ}_2(F)) \otimes C_c(F)
\] 
where $\mathrm O_3$ acts trivially on $C_c(F)$ and $\tilde{\SL}_2$ acts by the Weil representations $\rho_{\psi}$.
Recall that   as an $\widetilde{\SL}_2$-module, $\rho_{\psi}$ decomposes as a sum 
\[
\rho_{\psi} = \rho_{\psi}^+ \oplus \rho_{\psi}^- 
\] 
of even and odd Weil representations.  Let $\Theta(\rho_{\bar\psi}^+)$ and $\Theta(\rho_{\bar\psi}^-)$ be the theta lifts of their contragredients to $\mathrm O_3$, via  the Weil representation on $C_c(M^{\circ}_2(F))$ with respect to $\psi$. 
Thus the  $\SL_2$-coinvariant  of $\Pi_{[U,U], \psi}$  is given by 
\[ 
\mathrm{ind}_{\GL_2\rtimes \Z/2\Z}^{\PGL_3\rtimes \Z/2\Z} (\Theta(\rho_{\bar\psi}^+)) 
\oplus 
\mathrm{ind}_{\GL_2\rtimes \Z/2\Z}^{\PGL_3\rtimes \Z/2\Z} (\Theta(\rho_{\bar\psi}^-)). 
\] 
Let $\mathrm{st}$ be the Steinberg representation of $\SO(3)\cong \PGL_2$. We extend $\mathrm{st}$ to two representations $\mathrm{st}^+$ and 
$\mathrm{st}^-$ of $\mathrm O(3)$ by letting $-1\in \mathrm O(3)$ act by $1$ and $-1$ respectively. Then $\Theta(\rho_{\psi}^-)\cong \mathrm{st}^-$ 
while $\Theta(\rho_{\psi}^+)$ is the principal series representation with the trivial representation as a quotient and $\mathrm{st}^+$ as  a submodule. 
Since $\tau^-$ is a supercuspidal representation, it suffices to show that it is a quotient of 
\[ 
\mathrm{ind}_{\GL_2\rtimes \Z/2\Z}^{\PGL_3\rtimes \Z/2\Z} (\mathrm{st}^+) 
\oplus 
\mathrm{ind}_{\GL_2\rtimes \Z/2\Z}^{\PGL_3\rtimes \Z/2\Z} (\mathrm{st}^-). 
\] 
It is known that any generic representation of $\GL_2$, in particular $\mathrm{st}$, is a quotient of $\tau$. Hence either $\mathrm{st}^+$ or 
$\mathrm{st}^-$ is a quotient of $\tau^{-}$. Now the proposition follows from Frobenius reciprocity.

\end{proof}

\vskip 5pt 

\subsection{\bf Main result.}
 
We shall now strengthen Theorem \ref{T:fromGS}. 

\vskip 5pt

\begin{thm} \label{T:PGL3}
For any $\tau \in {\rm Irr}(\PGL_3)$, let $\phi_{\tau}: WD_F \longrightarrow \SL_3(\C)$ denote the L-parameter of $\tau$.    
\vskip 5pt

(i) The representation $\Theta(\tau^{\epsilon})$ is zero if and only if $\phi_{\tau}$ contains the trivial representation (so $\tau \cong \tau^{\vee}$) and $\epsilon = -$.

\vskip 5pt
(ii) For any $\epsilon = \pm$,  $\Theta(\tau^{\epsilon})$ has finite length with unique irreducible quotient $\theta(\tau^{\epsilon})$ (if it is nonzero).
\vskip 5pt

\noindent (iii)  $\theta(\tau^{\epsilon})$  is generic if and only if $\tau$ is generic and $\epsilon  = +$.  \vskip 5pt

\vskip 5pt

\noindent (iv)  Suppose that $\theta(\tau^{\epsilon}) \ne 0$. If $\tau$ is a discrete series (resp. tempered) representation, so is $\theta(\tau^{\epsilon})$. 
Moreover, $\theta(\tau^{\epsilon})$  is supercuspidal if and only if $\tau$ is supercuspidal  or $\tau^{\epsilon} = {\rm St}^-$. 
\vskip 5pt

\noindent (v) If  $\theta(\tau_1^{\epsilon_1}) \cong \theta(\tau_2^{\epsilon_2}) \ne 0$, then $\tau_1^{\epsilon_1} = \tau_2^{\epsilon_2}$. 
\vskip 5pt

In particular, the Howe duality theorem holds for the dual pair $(\PGL_3 \rtimes \Z/2\Z) \times G_2$:  
\[  \dim \Hom_{G_2}(\theta(\tau_1^{\epsilon_1}), \theta(\tau_2^{\epsilon_2})) \leq \dim \Hom_{\PGL_3 \rtimes \Z/2\Z} (\tau_1^{\epsilon_1}, \tau_2^{\epsilon_2}) \]
for any $\tau_1^{\epsilon_1}, \tau_2^{\epsilon_2} \in {\rm Irr}(\PGL_3 \rtimes \Z/2\Z)$. Moreover, for $\pi \in {\rm Irr}(G_2)$, $\Theta(\pi)$ is a finite length representation of $\PGL_3$ with a unique irreducible quotient (if nonzero).
\end{thm}

\vskip 5pt

\begin{proof} 

\noindent (i) From Theorem \ref{T:fromGS}, it remains to show that $\Theta(\tau^{\epsilon})$ is nonzero for those representations $\tau$ as in Theorem \ref{T:fromGS}(iii) and any $\epsilon = \pm$. Consider first the Steinberg representation.
 Recall that $\pi_{\rm gen}[1]$ is generic while $\pi_{\rm sc}[1]$ is not. 
It follows, from Lemma \ref{L:twisted_W} part (ii),  that  $\pi_{\rm gen}[1]$ is a summand of $\theta(\rm St^+)$.  Furthermore, by Proposition \ref{P:one-to-one I}, 
$\pi_{\rm sc}[1]$ cannot be a summand of $\theta(\rm St^+)$. 
Hence 
\[ \theta(\rm St^+)= \pi_{\rm gen}[1] \text{ and } \theta(\rm St^-)= \pi_{\rm sc}[1].
\] 
 The same argument works in the other three cases  to show that $\theta(\tau^+)$ is the generic $G_2$ summand and $\theta(\tau^-)$ is the degenerate summand. Moreover, in 
 the last case of Theorem \ref{T:fromGS}, where $\tau$ is a self-dual supercuspidal representation (so $p=2$), we deduce by Proposition \ref{P:one-to-one I} again that  $\pi_{\rm deg}$ is nonzero irreducible. 

\vskip 5pt

\noindent (ii) This follows from Theorem \ref{T:fromGS} and the irreducibility of $\pi_{\rm deg}$ in the proof of (i) above.
\vskip 5pt

\noindent (iii) and (iv): These summarize what we already know from Theorem \ref{T:fromGS}.
\vskip 5pt

\noindent (v) Suppose that 
\[ \pi:= \theta(\tau_1^{\epsilon_1}) \cong \theta(\tau_2^{\epsilon_2}) \ne 0. \]
If $\pi$ is non-supercuspidal, then $\tau_1$ and $\tau_2$ are both non-supercuspidal. The desired equality $\tau_1^{\epsilon_1} \cong \tau_2^{\epsilon_2}$ follows from the results of \cite{GS04} and our new understanding in (i) (which determines $\theta(\tau^{\epsilon})$ for those $\tau$ in Theorem \ref{T:fromGS}(iii))
\vskip 5pt

Suppose  that $\pi$ is supercuspidal. Then $\tau_i^{\epsilon_i}$ is either supercuspidal or $St^-$, in which case both $\tau_1$ and $\tau_2$ are generic discrete series representations. By (iii), we deduce that $\epsilon_1 = \epsilon_2$. Hence, it remains to show that $\tau_1 = \tau_2$ or $\tau_2^{\vee}$.
We now consider the following two cases:
\vskip 5pt

\begin{itemize}
\item[(a)] Suppose $\tau_1 \ncong \tau_1^{\vee}$ and $\tau_2 \ncong \tau_2^{\vee}$. Then $\epsilon_1 = \epsilon_2 = +$ and $\pi$ is an irreducible generic supercuspidal representation.
\vskip 5pt

Consider, for $i =1$ or $2$,  the induced representation ${\rm Ind}_{P_3}^{\PGSp_6} (\tau_i)$, where $P_3$ is the Siegel parabolic subgroup. 
Its normalized Jacquet functor with respect to $P_3$ is $\tau_i \oplus \tau_i^{\vee}$. 
Since  $\tau_i \ne \tau_i^{\vee}$, it follows that ${\rm Ind}_{P_3}^{\PGSp_6} (\tau_i)$ is an irreducible generic tempered representation. 
\vskip 5pt

By the computation of the Jacquet module of the minimal representation $\Pi'$ of $G_2 \times \PGSp_6$ along $P_3$ given in \cite{MS}, we deduce that
 $\pi\otimes {\rm Ind}_{P_3}^{\PGSp_6} (\tau_i)$  is an irreducible quotient of $\Pi'$.  By \cite{SWe},  a generic representation cannot lift to two different generic representaitons of $\PGSp_6$. Hence, we must have
 \[  {\rm Ind}_{P_3}^{\PGSp_6} (\tau_1) \cong {\rm Ind}_{P_3}^{\PGSp_6} (\tau_2).\]
By consideration of the Jacquet modules with respect to $P_3$,  we see  that $\tau_2 \cong  \tau_1$ or $\tau_1^{\vee}$, as desired.
   
 \vskip 5pt

\item[(b)] Assume now that $\tau_1 = \tau_1^{\vee}$. In this case,  we know that 
\[  \theta(\tau_1^+) = \pi_{gen} \quad \text{and} \quad  \theta(\tau_1^-) = \pi_{deg}. \]
Moreover,  the tempered representation
  ${\rm Ind}_{P_3}^{\PGSp_6} (\tau_1)$ is the sum of two representations, one of which is generic and the other degenerate (see Proposition \ref{P:P3}(i)). 
 By the Jacquet module of $\Pi'$ again,  we see that both $\pi_{gen}$ and $\pi_{deg}$ lifts to irreducible summands of ${\rm Ind}_{P_3}^{\PGSp_6} (\tau_1)$.  
  Moreover, $\pi_{deg}$ cannot lift to a generic representation of $\PGSp_6$ and hence must lift to the nondegenerate summand \cite{SWe}.
  By Proposition \ref{P:one-to-one I},  it follows that $\pi_{gen}$ cannot lift to the degenerate summand and thus must lift to the generic summand. 
  \vskip 5pt
  
  Now suppose that $\epsilon_1= \epsilon_2 = +$, so that $\pi = \theta(\tau_1^+) = \theta(\tau_2^+)$ is generic. 
  Then as before, we see that $\pi$  lifts to the generic summand of ${\rm Ind}_{P_3}^{\PGSp_6} (\tau_i)$ (regardless of whether $\tau_2$ is self-dual or not). By Jacquet module consideration, we see that $\tau_1 \cong \tau_2$. 
  On the other hand, if $\epsilon_1 = \epsilon_2 = -$, so that $\pi = \theta(\tau_1^-) = \theta(\tau_2^-)$ is nongeneric, then Proposition \ref{P:one-to-one II} implies that the nongeneric summand of  ${\rm Ind}_{P_3}^{\PGSp_6} (\tau_1)$ is contained in ${\rm Ind}_{P_3}^{\PGSp_6} (\tau_2)$. Again, Jacquet module considerations shows that $\tau_1 \cong \tau_2$.
    \end{itemize} 
 \vskip 5pt
 
 The inequality at the end of the theorem is simply a restatement of (v). Finally, given $\pi \in {\rm Irr}(G_2)$, we write
 \[   \Theta(\pi) = \Theta(\pi)_c \oplus \Theta(\pi)_{nc} \]
 as a sum of its cuspidal and noncuspidal component. As we noted in Lemma \ref{L:basic}, the results of \cite{GS04} imply that $\Theta(\pi)_{nc}$ has finite length. The result in (v) shows that $\Theta(\pi)$ has a unique irreducible quotient if it is nonzero, implying in particular that $\Theta(\pi)_c$ is either $0$ or irreducible, and hence $\Theta(\pi)$ has finite length. 
 \end{proof} 

\vskip 10pt

\section{\bf The group $\PGSp_6$}
Before discussing the last dual pair $G_2 \times \PGSp_6$, we need to devote the next few sections to a discussion of the structure and representations of $\PGSp_6$, as well as certain particular periods on $G_2$ and $\PGSp_6$. 
\vskip 5pt

Let $e_1, \ldots, e_6$ be the standard basis of $F^6$, where we have a symplectic form defined by  
\[ 
\omega(e_1,e_6)=\omega(e_2, e_5)= \omega(e_3,e_4)=1 
\] 
and  all other $\omega(e_i, e_j) =0$ with $i < j$ .  Let $\GSp_6$ be the group of linear transformations $g$ of $F^6$, such that for 
some $\nu(g)\in F^{\times}$ 
\[ 
\omega(gv,gw)= \nu(g)\cdot  \omega (v,w) 
\] 
for all $v,w\in F^6$. Then $\nu: \GSp_6\rightarrow F^{\times}$ is the similitude character.  
\vskip 5pt

Let $\tilde P_1$, $\tilde P_2$ and $\tilde P_3$ be maximal 
parabolic subgroups of $\GSp_6$ defined as the stabilizers of subspaces 
\[ 
\langle e_1\rangle \subset \langle e_1, e_2 \rangle \subset \langle e_1, e_2, e_3 \rangle 
\] 
respectively. For $i=1,2,3$, let $P_i \subseteq \PGSp_6$ be the quotient of $\tilde P_i$ by the center of $\GSp_6$.   
The group $\PGSp_6$ acts faithfully on $J= \wedge^2 F^6 \otimes i^{-1}$, and we shall (partially) describe how the parabolic subgroups act on this module.  
\vskip 5pt

The group $\PGSp_6$ can be explicitly described in terms of its action on $J$ as follows. Let $x_{ij} = e_i \wedge e_j  \in J$ for $i\neq j$.  
On $J$,  we have a natural trilinear form $(x,y,z)$
\[ 
\wedge^2 F^6  \times \wedge^2 F^6 \times \wedge^2 F^6  \rightarrow \wedge^6 F^6 \cong F. 
\] 
The group of linear transformations of $J$ preserving this form is $\SL_6/\mu_2$ and $\PGSp_6 =\Sp_6/\mu_2$ is the subgroup fixing 
\[ 
e=x_{16} + x_{25} + x_{34}. 
\] 
The Levi factor $M_3$ of $P_3$, as an algebraic group, is isomorphic to $\GL_3/\mu_2$.  Observe that group acts faithfully on $\wedge^2 F^3$, and since the latter 
is a three dimensional vector space, this action gives an isomorphism $\GL_3/\mu_2\cong \GL_3$. Thus we have an identification  
\[ 
M_3 = \GL(\langle x_{12}, x_{13}, x_{23}  \rangle). 
\] 
Under this identification, the maximal torus is given by diagonal matrices $(t_1, t_2, t_3)$. The three simple co-roots of $\PGSp_6$ are, respectively, 
\[ 
\alpha_1^{\vee}(t)=(1,t,t^{-1}), \hskip 5pt  \alpha_2^{\vee}(t)=(t,t^{-1},1), \hskip 5pt \alpha_3^{\vee}(t)=(1,t,t). 
\] 
An unramified character $\chi$ of the maximal torus is given by a triple of complex numbers $(s_1, s_2, s_3)$  
\[ 
\chi(t_1,t_2,t_3)= |t_1|^{s_1} |t_2|^{s_2} |t_3|^{s_3}. 
\] 
The Weyl group action on the characters is somewhat different in this picture. The simple reflections corresponding to the first two roots 
$\alpha_1$ and $\alpha_2$ are 
the usual permutations of entries of $(s_1, s_2, s_3)$, however, the simple reflection corresponding to the third simple root $\alpha_3$ is given by 
\[ 
(s_1, s_2, s_3) \mapsto (s_1+s_2+s_3, -s_3, -s_2). 
\] 
Thus the root hyperplanes are $s_i-s_j=0$  and $s_i+s_j=0$ for short and long roots, respectively. This looks like a $D_3$ root system; however,  
the Weyl-invariant quadratic form in this case is 
\[ 
q(s_1, s_2, s_3) = s_1^2 + s_2^2 + s_3^3 -\frac{1}{4}(s_1 + s_2 +s_3)^2 
\] 
rather than the usual dot product, and with this form, we have a realization of the $C_3$ root system with simple roots 
\[ 
\alpha_1=(0,1,-1), \hskip 5pt \alpha_2=(1,-1,0), \hskip 5pt \alpha_3=(0,2,2). 
\] 
This somewhat unconventional description of the $C_3$ root system is a source of potential confusion, as one has the tendency to confound it with the more familiar description of the root system of $\Sp_6$, but what we have done here is definitely the natural way to set things up for $\PGSp_6$. 

\vskip 5pt
 
The character $\chi$ is in the positive chamber if for every positive root $\alpha$, $\chi(\alpha^{\vee} (t))=|t|^{s}$ for some $s\in \mathbb C$ such that  
$\Re(s)>0$ (the real part).  One checks that $\chi$ is positive if 
\[ 
\Re(s_1) > \Re(s_2) > |\Re(s_3)|. 
\] 

The modulus character of $M_3\cong \GL_3$ is 
\[ 
\delta_{P_3}(m) = |\det(m)|^2 .   
\] 
It follows that the Levi factor $M_{13}$ of $P_{13} = P_1 \cap P_3$ is 
\[ 
M_{13}=  \GL(\langle x_{12}, x_{13}\rangle) \times \GL(\langle x_{23}\rangle). 
 \] 
The group $P_{13}$ is the stabilizer of the amber space $V_2=\langle x_{12}, x_{13}\rangle$.  
\vskip 5pt

Consider now the group $P_2$ and its Levi factor $M_2$.  The standard Levi factor of $\tilde P_2$ is $\GL_2 \times \GL_2$ where the first 
$\GL_2$ acts on $\langle e_1, e_2\rangle $ in the standard way, fixes $\langle e_3, e_4\rangle $ and 
acts by  transpose-inverse on $\langle e_5, e_6\rangle $. The second $\GL_2$ acts on $\langle e_3, e_4\rangle $ in the standard way, by $\det$ on $\langle e_1, e_2\rangle $ and 
fixes  $\langle e_5, e_6\rangle $.  
The group $P_2$ is the stabilizer of the singular line $V_1= \langle x_{12} \rangle$, and the Levi factor $M_2$ acts faithfully on the 4-dimensional subspace 
\[ 
V_4= \langle x_{13}, x_{23}, x_{14}, x_{24}  \rangle
\] 
preserving the quadratic form $x\mapsto (x,x, x_{56})$.  If we identify $x=ax_{14} + b x_{13} + cx_{24} + d_{23} $  with the matrix 
\[ 
\left( \begin{array}{cc} 
a & b \\ 
c & d \end{array} 
\right) 
\] 
then $(x,x,x_{56})=2\det(x)$. Thus, with $V_4$ identified with the set of $2\times 2$ matrices, we have 
\[  M_2  \cong \GL_2 \times  \GL_2  / \GL_1^{\nabla} \quad \text{where $\GL_1^{\nabla} = \{ (t, t^{-1}): t \in \GL_1 \}$}, \]
 so that $(\alpha,\beta) \in M_2$ acts on $x\in V_4$ by $x\mapsto \alpha x\bar \beta$ where  
$\bar \beta$ is the transpose of $\beta$.  
The element  $(\alpha,\beta)$  acts on the line $\langle x_{12} \rangle$ by $ \det (\alpha\beta)$.  The modulus character is 
\[ 
\delta_{P_2}((\alpha,\beta)) = |\det(\alpha\beta)|^5. 
\] 
This sets up the necessary notation to discuss the representations of $\PGSp_6$.
 \vskip 15pt

\section{\bf Representations of $\rm PGSp_6$}
In this section we list some irreducible non-supercuspidal representations of $\PGSp_6(F)$, relevant to this work. 
Observe that Langlands parameterization of irreducible representations is known for all 
of Levi factors of parabolic subgroups of $\PGSp_6(F)$  (by Gan and Takeda \cite{GT} for $M_1\cong \GSp_4$). Thus, following Shahidi \cite{Sh}, reducibility points of generalized principal series 
can be computed using $L$-functions of Langlands parameters. 

\vskip 10pt
\subsection{\bf Principal series representations for $P_2$. }
We first consider certain principal series representations for the parabolic subgroup $P_2 = M_2N_2$, where $M_2 \cong \GL_2\times \GL_2 / \GL_1^{\nabla}$. 
Let $\tau$ be an irreducible representation of $\GL_2$ with L-parameter $\phi_{\tau}$ and the central character $\omega_{\tau}$. Set 
\[  I_2(\tau\otimes \tau)= {\rm Ind}_{P_2}^{{\rm PGSp}_6} \tau \otimes \tau   \text{ and } 
 I_2(s, \tau\otimes \tau)  = {\rm Ind}_{P_2}^{G_2}( |\det|^s  \tau) \otimes (|\det|^s \tau) \]
if we need to consider a family of induced representations. Then we have:
\vskip 5pt

\begin{prop}  \label{P:P2}
(i)  If $\tau$ is unitary supercuspidal, then $I_2(s, \tau\otimes \tau)$ is reducible if and only if  $\tau^{\vee} \cong \tau$ (so $\omega_{\tau}^2=1$) and one of the following holds:
\vskip 5pt
\begin{itemize}
\item $\omega_{\tau} =1$ and $s = \pm 1/2$, in which case one has: 
\[  \begin{CD}
0 @>>>  \delta_2(\tau) @>>> I_2(1/2, \tau\otimes \tau) @>>>  J_2(1/2, \tau\otimes \tau) @>>> 0, \end{CD} \]
where $\delta_2(\tau)$ is a generic discrete series representation. 
\vskip 5pt 

 \item $\omega_{\tau}  \ne 1$ (so $\tau$ is dihedral),  ${\rm Im}(\phi_{\tau}) = S_3$  and $s=\pm 1$,in which case one has: 
   \[  \begin{CD}
0 @>>>  \sigma_{\mathrm {gen}}[\tau] @>>> I_2(1, \tau\otimes \tau) @>>>  J_2(1, \tau\otimes \tau) @>>> 0, \end{CD} \]
where $\sigma_{\mathrm {gen}}[\tau] $ is a generic discrete series representation.

\vskip 5pt

\item $\omega_{\tau} \ne 1$, ${\rm Im}(\phi_{\tau})  \ne S_3$  and $s= 0$, in which case one has:

\[   I_2(0,\tau\otimes \tau)  = I_2(\tau\otimes \tau)_{\mathrm {gen}}  \oplus I_{2}(\tau\otimes \tau)_{\mathrm {deg}} \]
where $I_2(\tau\otimes \tau)_{\mathrm {gen}}$ is generic. 

\end{itemize}

\vskip 5pt

\noindent (ii) If $\tau = {\rm st}_{\chi}$ is a twisted Steinberg representation, then $I_2(s,\tau\otimes \tau)$ is irreducible except for the following cases:
\vskip 5pt
\begin{itemize}
\item $\chi=1$ and $s = \pm 5/2$ or $\pm 1/2$, in which case one has
\[  \begin{CD}
0 @>>> {\rm St}_{{\rm PGSp}_6} @>>>  I_2(5/2, {\rm st}\otimes {\rm st}) @>>>  J_2(5/2,  {\rm st}\otimes {\rm st}) @>>> 0, \end{CD} \]
and
\[  \begin{CD}
0 @>>>   {\rm Ind}_{P_3}^{{\rm PGSp}_6}({\rm St})_{\rm gen}  @>>> I_2(1/2, {\rm st}\otimes {\rm st}) @>>>  J_2(1/2, {\rm st}\otimes {\rm st}) @>>> 0. \end{CD} \]
\vskip 5pt

\item $\chi^2=1$ but $\chi \ne 1$ and $s = \pm 1/2$, in which case one has:
 \[  \begin{CD} 
0@>>> \sigma_{\mathrm {gen}}[\chi] @>>>  I_2(1/2, {\rm st}_{\chi}\otimes {\rm st}_{\chi}) @>>>  J_2(1/2, {\rm st}_{\chi}\otimes {\rm st}_{\chi}) @>>> 0,  \end{CD} \]
where $\sigma_{\mathrm {gen}}[\chi] $ is a generic discrete series representation. 
\end{itemize}
\vskip 5pt

\end{prop}

\vskip 10pt
\subsection{\bf Principal series representations for $P_{13}$. }
Now we consider certain principal series representations for the  parabolic subgroup $P_{13} = M_{13}N_{13}$, where $M_{13} \cong \GL_2\times {\GL_1}$. 
Let $\tau$ be an irreducible  representation of $\GL_2$ with the central character $\omega_{\tau}$. Set 
\[  I_{13}(\tau\otimes 1)= {\rm Ind}_{P_{13}}^{{\rm PGSp}_6} \tau \otimes 1   \text{ and } 
 I_{13}(s, \tau\otimes 1)  = {\rm Ind}_{P_{13}}^{G_2}  (|\det|^s \tau) \otimes 1 \]
if we need to consider a family of induced representations.  In the more familiar language of representations of $\Sp_6$, the  restriction of $I_{13}(s,\tau\otimes 1)$ 
to $\Sp_6$ is a principal series induced from $|\cdot |^{2s} \omega_{\tau} \otimes |\det|^s \tau$. In particular, if $\tau$ is unitary tempered and $s>0$, this is a standard module. 
We have:
\vskip 5pt

\begin{prop}  \label{P:P13}
  If $\tau$ is unitary supercuspidal, then $I_{13}(s, \tau\otimes 1)$ is reducible if and only if  $\tau^{\vee} \cong \tau$ (so $\omega_{\tau}^2=1$) and one of the following holds:
\vskip 5pt
\begin{itemize}
\item $\omega_{\tau} =1$ and $s = \pm 1/2$, in which case  $I_{13}(1/2,\tau\otimes 1)$ has length 4 and has a unique irreducible submodule $\delta_{13}(\tau)$, which is 
 a generic discrete series representation. 

\vskip 5pt

\item $\omega_{\tau} \ne 1$, ${\rm Im}(\phi_{\tau})  \ne S_3$  and $s= 0$, in which case one has:

\[   I_{13}(0,\tau\otimes 1)  = I_{13}(\tau\otimes 1)_{\mathrm {gen}}  \oplus I_{13}(\tau\otimes 1)_{\mathrm {deg}} \]
where $I_{13}(\tau\otimes 1)_{\mathrm {gen}}$ is generic. 

\end{itemize}

\end{prop} 

\vskip 10pt
\subsection{\bf Principal series representations for $P_{3}$. }
Now we consider certain principal series representations for the  parabolic subgroup $P_{3} = M_{3}N_{3}$, where $M_{3} \cong \GL_3$. 
Let $\tau$ be an irreducible  representation of $\GL_3$. We set 
\[  I_{3}(\tau)= {\rm Ind}_{P_{3}}^{{\rm PGSp}_6} \tau.  
\] 
Then we have \cite[Example 7.7 and Theorem 7.9:]{Ta}:

\begin{prop}  \label{P:P3}
 (i) Assume that $\tau$ is discrete series representation with trivial central character. Then we have two cases: 
\begin{itemize} 
\item If $\tau \neq \tau^{\vee}$ then 
\[ 
I_3(\tau) \cong I_3(\tau^{\vee}) 
\] 
is irreducible. 
\vskip 5pt 
\item If $\tau\cong \tau^{\vee}$ then 
\[ 
I_3(\tau)  = I_3(\tau)_{\rm gen} \oplus I_3(\tau)_{\rm deg}
\] 
where $I_3(\tau)_{\rm gen}$  is generic.  

\end{itemize} 

\vskip 5pt 
(ii) Let $\chi_1, \chi_2, \chi_3$ be three characters of $F^{\times}$ such that $\chi_1 \cdot \chi_2 \cdot \chi_3=1$, and let $\tau=\tau(\chi_1,\chi_2, \chi_3)$ be the 
associated principal series representation of $\GL_3(F)$ (which is possibly reducible). 
Then the induced representation $I_3(\tau)$ is irreducible unless one of the following two conditions hold: 
\begin{itemize} 
\item $\chi_i= |\cdot| ^{\pm 1}$ for some $i$  or $\chi_i/\chi_j=|\cdot |^{\pm 1}$ for a pair $i\neq j$. 
\vskip 5pt 
\item The three characters $\chi_i$ are quadratic, non-trivial and pairwise different. Then 
\[ 
I_3(\tau)  = I_3(\tau)_{\rm gen} \oplus I_3(\tau)_{\rm deg}
\] 
where $I_3(\tau)_{\rm gen}$  is generic.  
\end{itemize} 
\end{prop}

\vskip 10pt
\subsection{\bf Principal series representations for $P_{1}$. }
Now we consider certain principal series representations for the parabolic subgroup $P_{1} = M_{1}N_{1}$, where $M_{1} \cong \GSp_4$. 
Let $\tau$ be an irreducible  representation of $\GSp_4$. We set 
\[ 
 I_{1}(\tau)= {\rm Ind}_{P_{1}}^{{\rm PGSp}_6} \tau \text { and }  I_{1}(s,\tau)= {\rm Ind}_{P_{1}}^{{\rm PGSp}_6} |\nu|^s \tau 
\] 
where $\nu$ is the similitude character of $\GSp_4$.  Let $\tau$ be an irreducible supercuspidal representation of $\GSp_4(F)$ with  trivial central character.  
Let $\varphi_{\tau} : WD_F \rightarrow \Spin_5 \cong \Sp(4)$ be its Langlands parameter \cite{GT}.  

\begin{prop}  \label{P:P1}
 Assume that $\tau$ is a supercuspidal representation of $\GSp_4(F)$ with  trivial central character such that the parameter ${\rm std} \circ \varphi_{\tau}$ contains the trivial representation, where ${\rm std}$ denotes the 5-dimensional standard representation  of $\Spin_5$. 
 Then $I_1(s, \tau)$ is reducible if and only if 
 $s = \pm 1/2$, in which case one has: 
\[  \begin{CD}
0 @>>>  \delta_1(\tau) @>>> I_1(1/2, \tau) @>>>  J_1(1/2, \tau) @>>> 0, \end{CD} \]
where $\delta_1(\tau)$ is a discrete series representation. 
\end{prop} 
\vskip 5pt

\section{\bf Fourier-Jacobi and Shalika periods}  
In this section, we introduce and study a Fourier-Jacobi-type model for the group $G_2$ and a Shalika period for $\PGSp_6$. These are some of the periods that will appear when we consider a game of ping-pong with periods for the dual pair $G_2 \times \PGSp_6$, as discussed at the end of the introduction. 
\vskip 5pt

\subsection{\bf Whittaker periods}
We begin by recalling the following results about Whittaker periods from \cite[Prop. 19 and Cor. 20]{GS04}, see also the appendix of \cite{HKT}.  
\vskip 5pt

\begin{prop} \label{P:GS-Whit}
Let $\Pi$ be the minimal representation of $E_7$ and let $(V', \psi_{V'})$ be a Whittaker datum for $\PGSp_6$ (so $V'$ is a maximal unipotent subgroup and $\psi_{V'}$ a generic character of $V'$). Then we have an isomorphism of $G_2$-modules:
\[  \Pi_{V', \psi'} \cong {\rm ind}_V^{G_2} \psi_V \]
where $(V, \psi_V)$ is a Whittaker datum for $G_2$. 
\end{prop}
\vskip 5pt

\begin{cor}  \label{C:GS-Whit}
(i) If $\pi \in {\rm Irr}(G_2)$ is generic and $\Theta(\pi)$ is its big theta lift to $\PGSp_6$, then 
\[  \dim \Hom_{V'}(\Theta(\pi),\psi) = 1 \]
so that $\Theta(\pi)$ contains a unique irreducible generic subquotient and  thus is nonzero.
\vskip 5pt

(ii) If $\pi \in {\rm Irr}(G_2)$ is non-generic and $\tau \in {\rm Irr}(\PGSp_6)$ is generic, then 
\[  \Hom_{G_2 \times \PGSp_6}(\Pi, \pi \otimes \tau) = 0. \]
\end{cor}
\vskip 5pt

\subsection{\bf Fourier-Jacobi period of $G_2$}  \label{SS:FJ} 
 
Let $Q=LU$ be the 3-step maximal parabolic subgroup of $G_2$. Recall that $[L,L]\cong \SL_2$ corresponds to the long simple root $\beta$. 
 Thus $V=U_{\beta} U$ is the unipotent radical of the standard Borel subgroup of $G_2$.  If we set
$J=[L,L]U$, then the quotient of $J$ by the center of $U$ is the Jacobi group. Let $\rho_{\psi}$ be the unique irreducible representation 
of the 2-fold cover  $\tilde J$, such that the center of $U/[U,U]\cong U_{2\alpha+\beta}$ acts by $\psi$. 
If $\sigma$ is a genuine representation of $ \widetilde{\SL}_2$, we have a representation of $J$ on $\sigma\otimes \rho_{\psi}$.  
For $\pi \in {\rm Irr}(G_2)$ and genuine  $\sigma \in {\rm Irr}(\tilde{\SL_2})$, the Fourier-Jacobi period of $\pi$ with respect to $\sigma$ is the space
\[  \Hom_J(\pi, \sigma \otimes \rho_{\psi}) \cong \Hom_{G_2}(\pi, {\rm Ind}_J^{G_2} \sigma \otimes \rho_{\psi}). \] 
\vskip 5pt

Recall that $\widetilde{\SL}_2$ has a family of principal series representations $I_{\psi}(s)$  such that when $s = 1/2$, one has a short exact sequence:
\[ 
0 \rightarrow \mathrm {St}_{\psi} \rightarrow I_{\psi}(1/2)\rightarrow \rho^{+}_{\psi} \rightarrow 0. 
\] 
 Moreover, the contragredient of $I_{\psi}(1/2)$ is $I_{\bar \psi}(-1/2)$. 
 
 \vskip 5pt 

We shall use the following enhanced uniqueness of the Fourier-Jacobi model (we say enhanced since the principal series is not assumed irreducible):

\begin{lemma} \label{L:FJ} 
 If  $\pi \in {\rm Irr}(G_2)$ is generic and tempered (or has cuspidal support on $Q$),  then 
\[ 
{\rm Hom}_{G_2} ( {\rm ind}^{G_2}_{J} (I_{ \psi}(-s) \otimes \rho_{\bar\psi}), \pi^{\vee}) \cong 
{\rm Hom}_{G_2} (\pi,  {\rm Ind}^{G_2}_{J} (I_{\bar \psi}(s) \otimes \rho_{ \psi}))
 \cong \mathbb C. 
\] 
for all $s<1/2$. 
\end{lemma} 
\begin{proof} 
Let $B=T_{\beta} U_{\beta} \subset [L,L]\cong \SL_2$ be the Borel subgroup, where $T_{\beta}$ is the one-dimensional torus
generated by the image of  the simple coroot $\beta^{\vee}: \GL_1\longrightarrow T$.  
 Observe that  for any genuine character $\chi$ of $\widetilde B$, we have an isomorphism of $J$-modules 
\[ 
 {\rm Ind}_{\widetilde B}^{\widetilde \SL_2} (\chi) \otimes \rho_{\psi} \cong {\rm Ind}_{ B U}^{J} (\chi \cdot \rho_{\psi}) 
\] 
where $f \otimes v $ is mapped to a function on $\SL_2$ given by $g\mapsto f(g) \cdot \rho_{\psi}(g)(v)$. Since $\rho_{\psi}$ restricted to $\widetilde BU$ is induced, 
using the transitivity of induction, it is easy to check that 
\[ 
I_{\bar \psi}(s) \otimes \rho_{ \psi} \cong {\rm ind}_{T_{\beta} N}^J( |\cdot|^{s+3/2} \cdot \psi_{2\alpha+\beta}) , 
\] 
where $N$ is the unipotent radical of the maximal parabolic $P$ (so $V=U_{\alpha}N$) and $\psi_{2\alpha+\beta}$ is a character of $N$ nontrivial only on the 
root space $U_{2\alpha+\beta}\subset N$ as the subscript indicates. Here the induction is not normalized. 
\vskip 5pt

Thus, it follows by Frobenius reciprocity that
\[ {\rm Hom}_{G_2} (\pi,  {\rm Ind}^{G_2}_{J} (I_{\bar \psi}(s) \otimes \rho_{ \psi})) \cong 
  \Hom_{T_{\beta} N}(\pi, |\cdot|^{s+3/2} \cdot \psi_{2\alpha+\beta}) \]
and to prove the lemma, it suffices to show that
\[  \dim \Hom_{T_{\beta} N}(\pi, |\cdot|^{s+3/2} \cdot \psi_{2\alpha+\beta})  =1 \quad \text{  if $s <1/2$.} \]  
We shall prove this in a somewhat roundabout fashion. 

\vskip 5pt

Let $\psi_{\alpha}$ be the character of 
$U$ obtained by restricting a Whittaker character of $V$. As the notation indicates, $\psi_{\alpha}$ is only supported on the root space $U_{\alpha} \subset U$.  
Let $\gamma=3\alpha+2\beta$ be the highest root. Observe that $\gamma$ is perpendicular to $\alpha$ and therefore, $\pi_{U,\psi_{\alpha}}$  is naturally a 
module for $T_{\gamma} U_{\beta}$, a mirabolic subgroup of $L\cong \GL_2$. Recall that the mirabolic subgroup has a unique irreducible infinite-dimensional (generic) representation, and  all  others are one-dimensional with the trivial action of $U_{\beta}$. The infinite-dimensional representation is realized on $C_c(T_{\gamma})$, so it is a regular representation 
of $T_{\gamma}\cong \GL_1$.  Since $\pi$ is Whittaker generic, $C_c(T_{\gamma})$ has to appear with multiplicity one in $\pi_{U,\psi_{\alpha}}$. Thus we have an exact 
sequence  of $T_{\gamma} U_{\beta}$-modules
\[ 
0\rightarrow C_c(T_{\gamma})\rightarrow \pi_{U,\psi_{\alpha}} \rightarrow \pi_{V,\psi_{\alpha}} \rightarrow 0 
\] 
where the third term is the $U_{\beta}$-coinvariants.
Next, observe that $ \pi_{V,\psi_{\alpha}}$ is a quotient of $\pi_{N}$. Since we have assumed that $\pi$ is tempered, the group $T_{\gamma}$ acts on $\pi_{V,\psi_{\alpha}}$  by 
characters which are, after taking the absolute value, equal to $|\cdot |^{t+3}$ where $t\geq 0$. It follows that 
\[  \dim \Hom_{T_{\gamma} U}(\pi, |\cdot |^{s+3} \cdot \psi_{\alpha}) =1 \quad \text{  if $s<0$.} \]
 
\vskip 5pt

However, we are still not done since the group $T_{\gamma} U$ is not even conjugate to $T_{\beta}N$. The last step requires the technique of root exchange. 

\vskip 5pt

Let  $N'$ be obtained by adding $U_{-\beta}$ to $U$ and removing $U_{\alpha+\beta}$, so that $N'$ is conjugate to $N$ (by the simple Weyl reflection $w_{\beta}$).  Now we claim that
there is an isomorphism
\[  \Hom_{T_{\gamma} U}( \pi, |-|^{s+3} \cdot \psi_{\alpha})  \cong \Hom_{T_{\gamma} N'} (\pi, |-|^{s+2} \cdot \psi_{\alpha}). \]
which sends $\ell$ on the LHS to an element $\ell'$ on the RHS defined by  the convergent integral
\[ 
\ell'(v)= \int_{U_{-\beta}}\ell( \pi(u) \cdot v ) ~du
\] 
Conversely, we can recover  $\ell$ from $\ell'$ by integrating over $U_{\alpha+\beta}$.   Assuming the claim, 
we complete the proof of the lemma by observing that  the pair $(T_{\beta}N, |-|^{s+2} \cdot\psi_{2\alpha+ \beta})$ is conjugate to the pair $(T_{\gamma} N', |-|^{s+2} \cdot \psi_{\alpha})$ (by the Weyl element $w_{\alpha} \cdot w_{\beta}$).  
\vskip 5pt

To justify the root exchange argument  in the claim, we observe that $U_{-\beta}$ and $U_{\alpha+\beta}$  generate a Heisenberg group with 
center $U_{\alpha}$, modulo higher order commutators. More precisely, consider the group
\[  V' = U \cdot U_{\beta} =  N' \cdot U_{\alpha+\beta}, \]
which is a maximal unipotent subgroup of $G_2$ and hence conjugate to $V$ (by the simple reflection $w_{\beta}$). If we consider the lower central series of the unipotent group $V'$:
\[  V' \supset [V',V'] = V'_1 \supset V'_2 = [V', V'_1] \supset V'_3 \supset \{1 \}, \] 
then $V / V'_2$ is   the Heisenberg group in question with center $V'_1/ V'_2 = U_{\alpha}$. Note moreover that the elements $\ell$ and $\ell'$ in the two Hom spaces in the claim both factors through $\pi_{V'_2}$ (which is a module for the Heisenberg group $V'/V'_2$). With this observation, the justification of the claim  is given by the following lemma, included as a convenience to the reader.   
\end{proof} 
\vskip 5pt

\begin{lemma} Let $H$ be a Heisenberg group.  Let $\pi$ be a smooth $H$-module.  Let $X$ and $Y$ be two maximal abelian subgroups of $H$. Let 
$\psi_X$ and $\psi_Y$ be characters of $X$ and $Y$, agreeing on the intersection $X\cap Y$, and non-trivial on the center 
of $H$. Then we have an isomorphism $\Hom_X(\pi, \psi_X)\cong  \Hom_Y(\pi , \psi_Y)$,  $\ell\mapsto \ell'$, 
defined by  
\[ 
\ell'(v) = \int_{Y/X\cap Y} \ell(\pi(y) v) ~ dy. 
\]  
\end{lemma} 
\begin{proof}  By the Frobenius reciprocity, we have 
\[ 
\Hom_X(\pi, \psi_X)\cong \Hom_H(\pi, \Ind_X^H \psi_X) \text { and } \Hom_Y(\pi, \psi_Y)\cong \Hom_H(\pi, \Ind_Y^H \psi_Y). 
\] 
We also have an isomorphism $\Ind_X^H \psi_X\cong \Ind_Y^H \psi_Y$, where every $f\in \Ind_X^H \psi_X$ goes to 
 $f'\in \Ind_Y^H \psi_Y$ defined by 
 \[ 
 f'(h)= \int_{Y/X\cap Y} f(yh) \bar \psi_Y(y)  ~ dy. 
 \] 
 The lemma follows by combining this isomorphism with the two Frobenius reciprocity isomorphisms. 
\end{proof} 
\vskip 5pt

\subsection{\bf Shalika period on $\PGSp_6$.}   
We shall now discuss a Shalika period on $\PGSp_6$.
\vskip 5pt

Recall the maximal parabolic subgroup $P_2 = M_2N_2$ of $\PGSp_6$ with Levi factor $M_2 \cong (\GL_2 \times \GL_2) / \GL_1^{\nabla}$ and unipotent radical $N_2$.
 Let $\psi_2$ be a generic character of $N_2(F)$.   The stabilizer of $\psi_2$ in the Levi group $M_2$ 
is the diagonally embedded $\PGL_2^{\Delta}$. The Shalika subgroup of $\PGSp_6$ is the semi-direct product
\[   S = PGL_2^{\Delta} \ltimes N_2 \]
and the Shalika character $\psi_S$ is the character $\psi_2$ extended to $S(F)$ (trivially on $\PGL_2(F)$).
For any smooth representation $\pi$ of $\PGSp_6(F)$, the Shalika period of $\pi$ is the coinvariant space $\pi_{S, \psi_S}$.
\vskip 5pt

This Shalikla period has already been exploited in \cite{SWe}. Indeed, the following was shown in \cite[Lemma 4.5]{SWe}:
\vskip 5pt

\begin{prop} \label{P:SWe-Shalika}
Let $\Pi$ be the minimal representation of $E_7$ and $(V, \psi_V)$ a Whittaker datum for $G_2$. Then
\[  \Pi_{V, \psi_V} \cong {\rm ind}_S^{\PGSp_6} \psi_S \]
as $\PGSp_6$-modules.
\end{prop}

\vskip 5pt

\subsection{\bf Shalika period of $\Pi$.}

We now consider the minimal representation $\Pi$ of the dual pair $G_2 \times \PGSp_6$ and determine its Shalika period $\Pi_{S, \psi_S}$ as a representation of $G_2(F)$. 
To describe the answer, we need to introduce some more notations.
\vskip 5pt

The group $\PGL_2$ acts by conjugation on the space $M_2(F)$ of $2 \times 2$ matrices, preserving the determinant (quadratic) form. As we saw in \S \ref{SS:ober},  there is a Weil representation of  $\PGL_2 \times \SL_2$ on $C_c(M_2(F))$ which decomposes as a tensor product 
\[ 
C_c(M_2(F))=C_c(M^{\circ}_2(F)) \otimes \rho_{\psi} 
\] 
where $M_2^{\circ}(F)$ is the space of trace zero matrices. We view $C_c(M^{\circ}_2(F)) \otimes \rho_{\psi} $ as a representation of the group $J = [L,L]U \subset G_2$ introduced in  \S \ref{SS:FJ}, where the first factor is a  representation of  $\widetilde \SL_2$ and 
$\rho_{\psi}$ is the irreducible representation of $\tilde J$ introduced in \S \ref{SS:FJ}. With the group $\PGL_2$  acting trivially on $\rho_{\psi}$, we see that $C_c(M_2(F))$ becomes a representation of $\PGL_2 \times J$.  

\vskip 5pt 

 We are now ready to compute $\Pi_{S,\psi_S}$.   Firstly, we have 
\[
\Pi_{N_2, \psi_2} \cong  {\rm ind}_{J} ^{G_2}  (C_c(M^{\circ}_2(F)) \otimes \rho_{\psi})
\] 
as $G_2 \times \PGL_2$-modules. It remains to compute the $\PGL_2$-coinvariant of the RHS. We need the following: 
\vskip 5pt 
\begin{lemma} \label{L:ind}
Let $H\subset G$ and $L$ be three $p$-adic groups. Let $W$ be a smooth $H\times L$-module, and $\tau$ an irreducible representation of $L$. 
Let $\Theta(\tau) \otimes \tau$ be the maximal $\tau$-isotypic quotient of $W$. 
If ${\rm Ext}^1_L(\tau,\tau)=0$ then 
\[ 
{\rm ind}_H^G \Theta (\tau) \otimes \tau  
\] 
is the maximal $\tau$-isotypic quotient of ${\rm ind}_H^GW$. Here ${\rm ind}$ stands for  induction with  compact support. 
\end{lemma} 

\begin{proof} Since ${\rm Ext}^1_L(\tau,\tau)=0$, the kernel of the projection of $W$ on $\Theta(\tau) \otimes \tau$ does not have $\tau$ as a quotient. Thus, it 
suffices to prove  that if $\Hom_L(W,\tau)=0$, then $\Hom_L( {\rm ind}_H^G W, \tau)=0$.  We shall prove that 
\[ 
\Hom_L( ({\rm ind}_H^G W)^K, \tau)=0
\] 
for any open compact subgroup $K$ of $G$. Write $G=\cup_{i\in I} H g_i K$ where $I$ is an index set, and set $K_i = H\cap g_iKg_i^{-1}$ for every $i\in I$. 
Then, as an $L$-module,  $({\rm ind}_H^G W)^K$ is a direct sum of $W^{K_i}$. Since $W^{K_i}$ is a direct summand of the $L$-module $W$, it follows that $\Hom_L( W^{K_i}, \tau)=0$, and this proves the lemma. 
\end{proof}

\vskip 5pt 
We apply Lemma \ref{L:ind} taking  $H \subset G$ to be $J \subset G_2$ and $L = \PGL_2$.  
Since ${\rm Ext}^1_{\PGL_2} (1,1)=0$, the lemma implies that 
computing $\PGL_2$-coinvariants of $\Pi_{N_2, \psi_2}$ boils down to computing the $\PGL_2$-coinvariant of $C_c(M^{\circ}_2(F))$, where 
it is the full degenerate principal series $I_{\bar \psi}(1/2)$.   We have shown:
\vskip 5pt

\begin{prop}  \label{P:shalika}
As a representation of $G_2(F)$, one has:
\[ 
\Pi_{S,\psi_S}\cong (\Pi_{N_2, \psi_2})_{\PGL_2} 
\cong {\rm ind}_{J} ^{G_2}(I_{\bar\psi}(1/2)\otimes  \rho_{\psi}). 
\] 
\end{prop}
\vskip 15pt

\section{\bf Howe Duality for $G_2 \times \PGSp_6$: Tempered Case} \label{S:howetemp}
After the preparation of the previous 3 sections, we are now in a position to begin our study of the theta correspondence for the dual pair $G_2 \times \PGSp_6$. 
In this section, we shall show the Howe duality theorem for tempered representations of $G_2$. The key is to show the analog of Propositions \ref{P:one-to-one I} and \ref{P:one-to-one II} for generic representations of $G_2$. This will rely on another curious chain of containment given in the following lemma, which comes from a consideration of a game of ping-pong with periods. 
\vskip 5pt

\begin{lemma}  \label{L:curious-2} Let $\Pi$ be the minimal representation of $E_7$.  
Let $\pi \in {\rm Irr}(G_2)$ be tempered and let $\psi_V: V \rightarrow \mathbb C^{\times}$ be  a Whittaker character  for $G_2$. 
% If $E$ is a field,  assume that $\pi$ is  different from the 3 exceptional representations in Proposition \ref{P:temp-sum}.
 Let   $H={\rm PGSp}_6$ and   $\tau \in {\rm Irr}(H)$ be  tempered such that 
\[  \Hom_{G_2 \times H }(\Pi,   \pi \boxtimes \tau) \ne 0. \]
Then we have the following natural inclusions 
 \begin{align} \label{E:chain}
 & \Hom_V(\pi, \psi_V)  \subseteq \Hom_V(\Theta(\tau), \psi_V)  \cong \Hom_{S}(\tau^{\vee},\bar\psi_S)   \subseteq
 \Hom_{S} (\Theta(\pi^{\vee}), \bar\psi_S) \cong   \Hom_{G_2}( {\rm ind}^{G_2}_{J}  I_{\psi}(1/2) \otimes \rho_{\bar\psi}, \pi^{\vee} ). \notag 
 \end{align}
 If $\pi$ is generic, then all of these spaces are one-dimensional. 
 \end{lemma}
 
 \begin{proof} 
 We examine each containment in turn:
 \vskip 5pt
 
 \begin{itemize}
 \item  The first inclusion arises from the surjection $\Theta(\tau) \twoheadrightarrow \pi$. 
 \vskip 5pt
 \item  The second follows from the identity 
 \[ 
  \Hom_V(\Theta(\tau), \psi_V)  \cong \Hom_{V \times H}(\Pi, \psi_V \boxtimes \tau) \cong  \Hom_{H}(\Pi_{V,\psi_V},  \tau) 
 \] 
 combined with  Proposition \ref{P:SWe-Shalika} (i.e. \cite[Lemma 4,.5]{SWe}):
 \[ 
 \Pi_{V,\psi_V} \cong \mathrm{ind}_{S}^{H} \psi_S 
 \] 
  and the Frobenius reciprocity. 
  \vskip 5pt
  
  \item  For the third, observe that $ \Theta(\bar \pi)$ is the complex conjugate of $\Theta(\pi)$. 
 Since $\bar\pi \cong \pi^{\vee}$ and $\bar \tau\cong \tau^{\vee}$ and we have 
$\Theta(\pi^{\vee}) \twoheadrightarrow \tau^{\vee}$. 
\vskip 5pt

\item The fourth follows from the  identity, 
\[  \Hom_{S}( \Theta(\pi^{\vee}) , \bar \psi_S) \cong \Hom_{S \times G_2}(\Pi, \bar \psi_S \boxtimes \pi^{\vee}) \cong \Hom_{G_2}(\Pi_{S, \bar \psi_S}, \pi^{\vee}) \]
combined  with Proposition \ref{P:shalika}:
\[  \Pi_{S,\bar\psi_S} \cong {\rm  ind}_J^{G_2} I_{\psi}(1/2) \otimes \rho_{\bar \psi} \]
and Frobenius reciprocity.
\end{itemize}
\vskip 5pt

 If $\pi$ is generic, then the first and the last spaces are one-dimensional, with the latter by Lemma \ref{L:FJ} applied to $s=-1/2 < 0$. Hence, all spaces in the chain
are one-dimensional. 
  \end{proof} 
  \vskip 10pt
  
  We can now obtain the following two propositions as consequences of Lemma \ref{L:curious-2}.
  \vskip 5pt
  
  \begin{prop} \label{P:one-to-one III} 
Let $\tau \in {\rm Irr}(\PGSp_6)$ be tempered. Then $\Theta(\tau)$ cannot have two irreducible tempered and generic quotients.  
\end{prop}  
\begin{proof} Let $\pi_1, \pi_2  \in {\rm Irr}(G_2)$ be tempered and generic such that $\Theta(\tau) \twoheadrightarrow \pi_1 \oplus \pi_2$. 
Then 
\[ 
\dim \Theta(\tau)_{V,\psi_V} \geq 2.
\] 
 On the other hand, $\dim \Theta(\tau)_{V,\psi_V}=1$ by Lemma \ref{L:curious-2}, which is  a contradiction. 
\end{proof} 

\vskip 5pt 
\noindent 
{\bf Remark:} This proposition is proved in \cite{SWe} using the uniqueness of Shalika functional, however, the proof of this uniqueness is difficult.
 The proof here is based on  Lemma \ref{L:FJ}
 whose proof is very elementary,  based on the Frobenius reciprocity and the root exchange technique. 

\vskip 5pt

 \begin{prop} \label{P:one-to-one IV} 
Let $\pi \in {\rm Irr}(G_2)$ be tempered and generic. 
Then $\Theta(\pi)$  cannot have two tempered irreducible quotients. In particular, its cuspidal component $\Theta(\pi)_c$ is irreducible or $0$.
\end{prop}  
\begin{proof}  Let $\tau_1, \tau_2  \in {\rm Irr}(\PGSp_6)$ be irreducible tempered and  such that $\Theta(\pi) \twoheadrightarrow \tau_1 \oplus \tau_2$. 
 By Lemma \ref{L:curious-2}, applied to 
$\pi^{\vee}$, $\tau_1^{\vee}$ and then to $\pi^{\vee}$, $\tau_2^{\vee}$, one has:
\[ 
1= \dim \Hom_{S}(\tau_1, \psi_S) =\dim \Hom_S(\Theta(\pi), \psi_S) =\dim \Hom_{S}(\tau_2, \psi_S). 
\] 
 Since $\tau_1\oplus \tau_2$ is a quotient of $\Theta(\pi)$, 
\[ 
1= \dim \Hom_{S}(\Theta(\pi), \psi_S) \geq \dim \Hom_{S}(\tau_1, \psi_S) + \dim \Hom_{S}(\tau_2, \psi_S) =2, 
\] 
which is a contradiction. 
\end{proof} 

Combining Propositions \ref{P:one-to-one III} and \ref{P:one-to-one IV}  with  the results of \S \ref{S:dichotomy}, we can now show the Howe duality theorem for tempered representations of $G_2$: 
\vskip 5pt

\begin{theorem} \label{T:main}
Let $\pi \in {\rm Irr}(G_2)$ be tempered and consider its big theta lift $\Theta(\pi)$ on $\PGSp_6$. Then 
\vskip 5pt

(i)  $\Theta(\pi)$ has finite length and a unique irreducible quotient $\theta(\pi)$ (if nonzero), which is tempered. 
\vskip 5pt

(ii) Moereover, for tempered $\pi_1, \pi_2 \in {\rm Irr}(G_2)$,
\[ 0 \ne  \theta(\pi_1) \cong \theta(\pi_2) \Longrightarrow \pi_1 \cong \pi_2. \]
\end{theorem} 

\begin{proof}
\noindent (i) We have seen (i) for non-generic $\pi$ in Corollary \ref{C:howe}. The proof for generic $\pi$ is the same, using Lemmas \ref{L:basic}(i) and \ref{L:basic2}(ii), as well as Proposition \ref{P:one-to-one IV}. 
\vskip 5pt

\noindent  (ii) If one of $\pi_1$ or $\pi_2$ is nongeneric, then the desired result follows by Proposition \ref{P:one-to-one I}. If $\pi_1$ and $\pi_2$ are both generic, then the desired result follows by Proposition \ref{P:one-to-one III}.

\end{proof}

We also point out the following corollary: 

\begin{cor} Let $\pi \in {\rm Irr}(G_2)$ be generic, supercuspidal and not a theta lift from $\PGL_3$. Then 
$\Theta(\pi)$ is generic, supercuspidal and irreducible.  
\end{cor} 
\begin{proof} By \cite{SWe}, we have known that $\Theta(\pi)$ is generic and supercuspidal (hence tempered and semisimple), but now we know by Proposition  \ref{P:one-to-one IV} that  it is also irreducible. 
\end{proof} 

\vskip 15pt

\section{\bf Jacquet Modules} \label{S:jacquet} 
 The purpose of this section is to  compute the various Jacquet modules of the minimal representation of $E_7$ with respect to the maximal parabolic subgroups of $G_2$ and $\PGSp_6$. We note that the results of this section are entirely self-contained, and do not depend on any prior results in this paper. As consequences of the results here, we deduce Lemmas \ref{L:basic} and \ref{L:basic2} for the dual pair $G_2 \times \PGSp_6$. Indeed, we shall determine in Theorem \ref{T:nont}   the theta lifts of all non-tempered representations of $G_2$ and $\PGSp_6$ precisely.   
   \vskip 5pt
 
\subsection{\bf Jacquet functors for $G_2$}  \label{SS:jac}
Recall that  $P=MN$ and $Q=LU$ are the two maximal parabolic subgroups of $G_2$ as before,   in  standard position relative to a maximal split 
torus $T$ in $G_2$ and a choice of positive roots, so that $P \cap Q$ is a Borel subgroup. In particular, their Lie algebras arise from 
$\mathbb Z$-gradings given by two fundamental co-characters. 
Since $G_2$ is contained in $E_7$, as a memeber of the dual pair, the two co-characters give $\mathbb Z$-gradings of the Lie algebra of $E_7$, 
defining parabolic subgroups $\mathcal P=\mathcal M \mathcal N$ and $\mathcal Q=\mathcal L \mathcal U$ of $E_7$, whose intersections with $G_2$ are $P$ and $Q$, respectively. The Lie types of the Levi factors 
 $\mathcal M $ and $\mathcal L$ are $D_6$ and $A_1 \times A_5$, as explained in  \cite{GS99}.   In the rest of the paper, we shall fix the following: 
 
 \begin{itemize} 
\item The group $P$ acts on  $C_c(\GL_2)$  and $C_c(\GL_1)$ by left translation by $g$ and $\det(g)$, respectively,  via the identification $M\cong \GL_2$ .
\item The group $Q$ acts on  $C_c(\GL_2)$ and $C_c(\GL_1)$ by left translation by  $g$ and $\det(g)$, respectively,  via the  identification $L\cong \GL_2$.
\item Let $\bar B$ be the group of lower-triangular matrices in $\GL_2$. Then $\bar B$   acts on  $C_c(\GL_1)$ by right translation by the $(1,1)$ matrix entry of $g\in \bar B$.   
\end{itemize}

 \vskip 10 pt

We identify $M\cong \GL_2$ such that the action of $M$ on $N/[N,N]$ is the symmetric cube of the standard representation of $\GL_2$ twisted 
by determinant inverse. In particular, a scalar matrix $(z,z)$ in $\GL_2$ acts by $z$. We have \cite[Theorem 6.1]{MS},  
 
 \begin{prop}\label{P:jf_P}  Let $H=\PGSp_6$. As a $\GL_2\times H$-module, $r_{P}(\Pi)$ (the normalized Jacquet functor) has a filtration with three successive 
sub quotients (top to bottom): 
\begin{enumerate} 
\item $\delta_P^{-1/2}\cdot \Pi_{\mathcal N}=\Pi_{D_6} \cdot |\det|^{1/2}\oplus \Pi_{\emptyset}\cdot |\det|^{3/2}$. 
\item  $\Ind_{\bar B\times P_2}^{\GL_2\times H}(\delta \cdot C_c(\GL_1))$. 
\item  $\Ind_{P_{13}}^{H} C_c(\GL_2)$. 
\end{enumerate}  
Here, note that:
\vskip 5pt

\begin{itemize}
\item[-] In (1), the center of $M\cong \GL_2$ acts trivially on both $\Pi_{D_6}$ and $\Pi_{\emptyset}$, the minimal and the trivial representation of the Levi $\mathcal M$. 
\item[-] In (2), $\delta= |\cdot|^{-1/2} \times |\cdot |$ is a character of the group $\bar B$ of  lower triangular matrices in $\GL_2$.
\end{itemize} 
\end{prop}

\vskip 5pt

 Let $W$ be the Weil representation for the similitude dual pair $\GL_2 \times \mathrm{GSO}_4$; see  \cite{R96} where theta correspondences for similitude groups are 
treated in detail.   Observe that $ \mathrm{GSO}_4\cong (\GL_2 \times \GL_2) / \GL_1^{\nabla}$, with the isomorphism realized by latter 
acting on the space $\mathbb M_2(F)$ of $2\times 2$ matrices by left and right multiplication and  the quadratic form given by the determinant.  We identify the first factor $\GL_2$ with $L$  so that the action of $L$ on $U/[U,U]$ is the standard representation of $\GL_2$. 
The irreducible quotients of $W$ are $\pi^{\vee}\otimes \pi \otimes \pi$, where $\pi$ is an irreducible representation of $\GL_2$.  We need a slight refinement of this 
to the big theta lifts. 

\begin{lemma}  \label{L:sim-big}
Consider the similitude theta correspondence for the dual pair $\GL_2 \times \mathrm{GSO}_4$ on $W$. Let $\pi$ be an irreducible generic representation of $\GL_2$. 
Then $\Theta(\pi^{\vee}) = \pi \otimes \pi$ and $ \Theta(\pi \otimes \pi)= \pi^{\vee}$.  
\end{lemma} 
\begin{proof} Let $V\cong F$ be a maximal unipotent subgroup of $\GL_2$ and $\psi$ a non-trivial character of $V$. We shall use that 
$W_{V,\psi} =C_c(\GL_2)$, the regular representation of $\GL_2$, where $V$ is in any of the three $\GL_2$.  

We have  $\Theta(\pi\otimes \pi)\otimes (\pi\otimes \pi)$  as a quotient of $W$. Apply the functor of $(V,\psi)$-coinvariants, with $V$ sitting in one of $\GL_2$ 
factors of $\mathrm{GSO}_4$, to conclude that $\Theta(\pi\otimes \pi)\otimes \pi$ is a quotient of the regular representation of $\GL_2$. This implies that 
$\Theta(\pi\otimes \pi) \cong \pi^{\vee}$, as desired.  In the other direction,  similar arguing shows that $\Theta(\pi^{\vee})$ cannot be an extension of $\pi\otimes \pi$ by 
$\pi\otimes \pi$. Thus the lemma holds except perhaps when $\pi$ is a character twist of the Steinberg representation ${\rm st}$. 
For example,  $\Theta({\rm st})$ could 
be a non trivial extension of ${\rm st}\otimes {\rm st}$ by $1 \otimes 1$.  But 
\[ 
{\rm Ext}^1_{\SO_4}( 1\otimes 1, {\rm st}\otimes {\rm st})=0 
\] 
thus one can have a non-trivial extension of these two representations of $\mathrm{GSO}_4$ only if the center of $\mathrm{GSO}_4$ does not act by scalars 
on $\Theta({\rm st})$. But it does, since the centers of the three $\GL_2$ are identified, the action on  $\Theta({\rm st})$  is equal to the action on 
${\rm st}$ where it is scalar action.

\end{proof} 
 
 \begin{prop}\label{P:jf_Q}  
 Let $H=\PGSp_6$. As a $\GL_2\times H$-module, $r_{Q}(\Pi)$ (the normalized Jacquet functor) has a  filtration with three successive 
sub quotients (top to bottom): 
\begin{enumerate} 
\item $\delta_Q^{-1/2}\cdot \Pi_{\mathcal U}=\Pi_{A_5} \cdot |\det|^{3/2}\oplus \Pi_{A_1}\cdot |\det|^2$. 
\item  $\Ind_{\bar B\times P_2}^{\GL_2\times H}(\delta \cdot C_c(\GL_1))$. 
\item  $\Ind_{P_2}^{H} W$. 
\end{enumerate}  
\vskip 5pt

\begin{itemize}
\item[-] In (1), the center of $L\cong \GL_2$ acts trivially on both $\Pi_{A_5}$ and $\Pi_{A_1}$, the minimal and a principal series 
 representation of the two factors of $\mathcal L$. 
\item[-] In (2) $\delta= |\cdot|^{1/2} \times |\cdot |$ is a character of the group $\bar B$ of  lower triangular matrices in $\GL_2$. 
\end{itemize} 
\end{prop}
\begin{proof} 
This proposition is entirely similar to Proposition 6.8 in  \cite{GS99}, which treated the case of non-split form of $H$,  
except the character $\delta$ was not determined there. This is done as follows. For a generic  character 
$\chi$ of $\GL_2$, representations $I_Q(\chi)$ and $I_2(\chi\otimes \chi)$ are both irreducible and $I_Q(\chi) \otimes I_2(\chi\otimes \chi)$ is a quotient of 
$\Pi$, this follows from the bottom factor (3) of the filtration. Hence $r_Q(I_Q(\chi))\otimes I_2(\chi\otimes \chi)$ is a quotient of 
$r_{P_2}(\Pi)$. Now determining $\delta$ is an easy exercise using $r_{Q}(I_Q(\chi))$.  
\end{proof}

\vskip 10pt

\subsection{\bf Non-tempered representations}  \label{SS:nt}
We enumerate the nontempered irreducible representations of $G_2$ using the discussion from Section 3. 
Let $P=MN$ and $Q=LU$ be the two maximal parabolic subgroups in $G_2$ as before. Their Levi groups are isomorphic to $\GL_2$. 
Let $\tau$ be a representation of $\GL_2$, and let 
$I_P(\tau)$ and $I_Q(\tau)$ be the corresponding normalized induced representations of $G_2$. Irreducible, non-tempered representations of $G_2$ 
are described as follows, where $\tau$ is irreducible, and $\omega_{\tau}$ is the central character of $\tau$.

\begin{itemize} 
\item[(a)]  Unique irreducible quotient of $I_Q(\tau)$ where $\tau$ is an unramified twist of a tempered representation such that $|\omega_{\tau}|=|\cdot |^s$ for some $s>0$. 
\item[(b)]  Unique irreducible quotient of $I_P(\tau)$ where $\tau$ is an unramified twist of a tempered representation such that $|\omega_{\tau}|=|\cdot |^s$ for some $s>0$. 
\item[(c)]  Unique quotient of $I_P(\tau)$ where $\tau$ is the unique quotient of a representation induced from an ordered pair of characters $\chi_1, \chi_2$ such that 
$|\chi_1|=|\cdot |^{s_1}$, $|\chi_2|=|\cdot |^{s_2}$ where $s_1> s_2 > 0$. 
\end{itemize}  
In (a) and (b),  $I_Q(\tau)$ and $I_P(\tau)$ are  standard modules, while in (c), $I_P(\tau)$ is a quotient of a standard module associated 
to the minimal parabolic $P\cap Q$.  In any case, each of these induced representations has a unique irreducible quotient which we denote by $J_Q(\tau)$ in (a) and by $J_P(\tau)$ in (b) and (c). These representations $J_Q(\tau)$ and $J_P(\tau)$ exhaust the irreducible nontempered representations of $G_2$. 

\vskip 5pt

 %The functoriality principle for the inclusion of dual groups $G_2(\mathbb C) \subset {\rm Spin}_7(\mathbb C)$ predicts where 
%non-tempered representations of $G_2$ should lift.  We summarize here, for more details see Proposition 1.1 and its proof in \cite{SWe}. 
We also enumerate some relevant nontempered representations of $\PGSp_6$.
 Let $P_i=M_i N_i$, $i=1,2,3$ be three maximal parabolic subgroups of $\PGSp_6$. 
 Let  $I_i(\sigma)$ denote the representation of $\PGSp_6$ obtained by 
 normalized parabolic induction from $P_i$, and let $I_{jk}(\sigma)$  denote the representation of $\PGSp_6$ obtained by 
 normalized parabolic induction from $P_j\cap P_k$.  We shall consider the following non tempered representations of $\PGSp_6$, corresponding to the cases (a), (b) and (c) above:
 \vskip 5pt
 
 \begin{itemize}
 \item[(a')]  If $\tau$ is an irreducible representation of $L=\GL_2$ satisfying the conditions of (a) above,   let $\sigma =\tau\otimes \tau$ be a representation of $M_2 \cong \GL_2 \times \GL_2 / \GL_1^{\nabla} \cong {\rm GSO}_4$. Then $I_2(\sigma)$ is a standard module, with unique irreducible quotient $J_2(\sigma) = J_2(\tau \otimes \tau)$. 
 \vskip 5pt
 
 \item[(b')]  If  $\tau$ is an irreducible representation of $M=\GL_2$, satisfying the conditions of (b) above,  let  $\sigma =\tau \otimes 1$ be a representation of $M_1\cap M_3 \cong \GL_2 \times \GL_1$.  Then $I_{13}(\sigma)$ is   a standard module with unique irreducible quotient $J_{13}(\sigma) = J_{13}(\tau \otimes 1)$.
 \vskip 5pt
 
 \item[(c')]  If  $\tau$ is an irreducible representation of $M=\GL_2$, satisfying the conditions of (c) above,  let  $\sigma =\tau \otimes 1$ be a representation of $M_1\cap M_3 \cong \GL_2 \times \GL_1$.  Then $I_{13}(\sigma)$ is  a quotient of a standard module associated to the 
 Borel subgroup, Hence, it has a unique irreducible quotient which we denote by $J_{13}(\sigma) = J_{13}(\tau \otimes 1)$.
 \end{itemize}
 
\vskip 5pt 
\subsection{\bf Theta lifts from $G_2$.}
Now the following lemma attempts to compute the theta lifts of the above non tempered representations of $G_2$ to $\PGSp_6$.
\vskip 5pt

\begin{lemma} \label{L:non_tempered} 
Let $\pi\in  {\rm Irr}(G_2)$ be non-tempered. 
\begin{itemize} 
\item If $\pi \subset I_Q( \tau^{\vee})$ where $\tau$ is as in (the first bullet (a) above, then $\Theta(\pi)$ is a quotient of $I_{2}(\tau\otimes\tau)$ and hence has finite length.
 Moreover, $\Theta(J_2(\tau\otimes\tau))\neq 0$ where $J_2(\tau\otimes\tau)$ is the unique irreducible quotient of $I_2(\tau\otimes\tau)$. 
 
\item If $\pi \subset I_P( \tau^{\vee})$ where $\tau$ is as in (the second bullet) (b) and (c)  above, then $\Theta(\pi)$ is a quotient of $I_{13}(\tau\otimes 1)$ and hence has finite length.
 Moreover, $\Theta(J_{13}(\tau\otimes 1))\neq 0$ where $J_{13}(\tau\otimes 1)$ is the unique irreducible quotient of $I_{13}(\tau\otimes 1)$. 
\end{itemize} 
\end{lemma} 
\begin{proof} 
Let $\Pi$ be the minimal representation, and $\pi\in {\rm Irr}(G_2)$.  
We shall use the fact  that 
\[ 
\Theta(\pi)^{\ast}\cong \Hom_{G_2}(\Pi, \pi) 
\] 
as non-smooth $H=\PGSp_6$-modules, where former is the linear dual of $\Theta(\pi)$.  Assume that $\pi \subset I_Q( \tau^{\vee})$. Then 
\[ 
\Theta(\pi)^* = \Hom_{G_2}(\Pi, \pi) \subset \Hom_{G_2}(\Pi, I_Q (\tau^{\vee}))\cong \Hom_{L}(r_Q(\Pi), \tau^{\vee}).
\] 
Now we shall use the filtration of $r_Q(\Pi)$ from Proposition \ref{P:jf_Q}.
\vskip 5pt

 Let $\Pi_1$, $\Pi_2$ and $\Pi_3$ denote the three sub quotients in the same order. 
 Observe that $\mathrm {Ext}^i_{L}(\Pi_1, \tau^{\vee})$ are trivial from the central character considerations, since the central character of 
$\tau^{\vee}$ is a negative power of $|z|$. Hence we have a long exact sequence 
\[ 
0 \rightarrow \Hom_{L}(\Pi_2, \tau^{\vee})\rightarrow \Hom_{L}(r_Q(\Pi), \tau^{\vee}) \rightarrow \Hom_{L}(\Pi_3, \tau^{\vee})\rightarrow \mathrm{Ext}^1_{L}(\Pi_2, \tau^{\vee}) 
\] 
Since $\Pi_2$ is induced from $\bar B$, by the second adjointness,  
\[
\mathrm{Ext}^i_{L}(\Pi_2, \tau^{\vee}) \cong \mathrm{Ext}^i_{T}( \Ind_{P_2}^{H}(\delta \cdot C_c(\GL_1)), r_{B}(\tau^{\vee})) 
\] 
where $T=\GL_1\times \GL_1$, the maximal torus in $B$. 
 Observe that the action of the second $\GL_1$ on $\Ind_{P_2}^{H}(\delta \cdot C_c(\GL_1))$ is $|\cdot |$, and this is 
different from the action on $r_{B}(\tau^{\vee})$ by our assumption on $\tau$. Hence $\mathrm{Ext}^i_{L}(\Pi_2, \tau^{\vee})=0$ for all $i$, and we can conclude that 
\[ 
\Hom_{L}(r_Q(\Pi), \tau^{\vee})\cong \Hom_{L}(\Pi_3, \tau^{\vee}) \cong \Hom_{L}(\Ind_{P_2}^{H} W,  \tau^{\vee}), 
\] 
where, for the second isomorphism, we have simply substituted the explicit expression for $\Pi_3$ given in Proposition \ref{P:jf_Q}.  By  \cite[Lemma 9.4]{GG06}, the maximal $\tau^{\vee}$ isotypic 
quotient of $ \Ind_{P_2}^{H} W$ is $(\Ind_{P_2}^{H} \Theta(\tau^{\vee}) ) \otimes \tau^{\vee}$ where $\Theta(\tau^{\vee})$ is the big theta lift for 
the similitude theta correspondence on $W$. Since $\tau$ is generic, Lemma \ref{L:sim-big} shows that
 $\Theta(\tau^{\vee}) = \tau\otimes\tau$ and  it follows that 
\[ 
\Hom_{L}(\Pi_3, \tau^{\vee})\cong I_{2}(\tau\otimes\tau)^{\ast}. 
\] 
Hence $\Theta(\pi)^{\ast}\subset I_{2}(\tau\otimes\tau)^{\ast}$,  and  $\Theta(\pi)^{\vee} \subset  I_{2}(\tau\otimes\tau)^{\vee}$ by taking smooth vectors. Thus 
$\Theta(\pi)$ is a quotient of $I_{2}(\tau\otimes\tau)$.  Observe that we have proved in the process that $I_{2}(\tau\otimes\tau)$ is a quotient of $\Pi$, so that $\Theta(J_2(\tau \otimes \tau) \ne 0$. This proves the first bullet. The proof of the second is completely analogous. 
\end{proof} 
\vskip 5pt

\subsection{\bf Jacquet functors for $\PGSp_6$} 

Recall that in $\PGSp_6$, we have fixed three standard maximal parabolic subgroups $P_1$, $P_2$ and $P_3$. They correspond to 
$\mathbb Z$-gradings of the Lie algebra of $\PGSp_6$ given by three fundamental co-characters. The action of  each of these three co-characters gives 
a $\mathbb Z$-grading of the Lie algebra of $E_7$, and these gradings define three parabolic subgroups $\mathcal P_1$, $\mathcal P_2$ and $\mathcal P_3$ 
of $E_7$. To recognize these parabolic subgroups, perhaps it is easiest to proceed as follows. 
Observe that the $E_7$ Dynkin diagram contains  a unique $D_4$ subdiagram. We embed $G_2$ into $D_4$. The centralizer of $G_2$ in the split, adjoint $E_7$ is $\PGSp_6$. 
Let $\mathcal P$ be the parabolic subgroup of $E_7$, whose Levi factor has the type $D_4$.  
This parabolic is contained in precisely three maximal parabolic subgroups denoted by $\mathcal P_1$, $\mathcal P_2$ and $\mathcal P_3$, whose Levi factor types are,  
$D_6$, $A_1 \times D_5$ and $E_6$, respectively.  The intersection of $\mathcal P_i$ and 
$\PGSp_6$ is $P_i$, for each $i$. We write $\mathcal P_i= \mathcal M_i \mathcal N_i$ and $P_i=M_i N_i$ the Levi decompositions for these parabolic subgroups.

\vskip 10pt 
\noindent
\underline{Case $P_3$:} 
This is treated in \cite[\S 5]{MS}, and we summarize the results as follows. 
The unipotent subgroups of $P_3$ and $\mathcal P_3$ are abelian, $M_3\cong \GL_3$ and the modular character is 
\[ 
\delta_{P_3}(m)= |\det(m)|^2. 
\] 
Let $\mathbb O_0$ denote the space of trace 0 elements in the octonion algebra $\mathbb{O}$.  On the space $\mathbb O_0^3$, we have the natural diagonal action 
$(x,y,z) \mapsto (gx,gy,gz)$ of $g\in G_2$ and the  row-vector  action $(x,y,z) \mapsto (gx,gy,gz)m^{-1}$ of $m \in \GL_3$.   
Let $\Omega\subset \mathbb O_0^3$ be the set of all  $(x,y,z)$ such that the linear subspace 
$\langle x,y,z \rangle \subset \mathbb O_0$ is a null-space for octonion multiplication i.e. the product of any two elements in the space is 0. 
Such non-zero null-spaces in $\mathbb O_0$ are of dimension 1 or 2. 
\vskip 5pt

We have an exact sequence of $G_2 \times \GL_3$-modules 
\[ 
0 \rightarrow C_c(\Omega) \rightarrow \Pi_{N_3} \rightarrow \Pi_{\mathcal N_3} \rightarrow 0 
\] 
where $(g, m) \in G_2 \times \GL_3$ acts on $f\in C_c(\Omega)$ by 
\[ 
((g,m) \cdot f) (x,y,z) = |\det(m)|^2\cdot  f((g^{-1}x,g^{-1}y,g^{-1}z)m). 
\] 
The group $G_2 \times \GL_3$ acts on $\Omega$ with two orbits $\Omega_1$ and $\Omega_2$ where $\Omega_i$ is the subset of triples 
$(x,y,z)$ such that $\langle x,y,z \rangle$ has dimension $i$. Thus $C_c(\Omega)$ has a filtration with $C_c(\Omega_2)$  a submodule and 
$C_c(\Omega_1)$  a quotient. Each of these can be explicitly described as $G_2 \times \GL_3$-modules.  
\vskip 5pt

In order to state the result, let $Q_1$ and $Q_2$ be the maximal 
parabolic subgroups of $\GL_3$ stabilizing subspaces consisting of row vectors $(\ast, 0,0)$ and $(\ast, \ast,0)$, respectively.  Observe that these are block
lower-triangular groups with Levi factors isomorphic to $\GL_1 \times \GL_2$ and $\GL_2 \times \GL_1$, respectively. Their modular characters are 
\[ 
\delta_{Q_1} (g_1, g_2)= |g_1|^{-2} \cdot |\det(g_2)|  \quad \text{ and } \quad \delta_{Q_2} (g_2, g_1)= |\det(g_2)|^{-1} \cdot |g_1|^2. 
\] 
Recall that $r_{P_3}(\Pi)= \delta_{P_3}^{-1/2} \cdot \Pi_{N_3}$ is  the normalized Jacquet module. Then: 

\begin{prop}\label{P:jf_p3} As a $G_2\times \GL_3$-module, $r_{P_3}(\Pi) $ has a filtration with three successive 
subquotients (from top to bottom): 
\begin{enumerate} 
\item $\delta_{P_3}^{-1/2} \cdot \Pi_{\mathcal N_3} = \Pi_{E_6} \oplus \Pi_{\emptyset} \cdot |\det|$. 
\item  $\Ind_{Q\times Q_1}^{G_2\times \GL_3}(\delta \cdot C_c(\GL_1))$. 
\item  $\Ind_{P\times Q_2}^{G_2\times \GL_3}(C_c(\GL_2))$. 
\end{enumerate} 
Here, note that:
\vskip 5pt

\begin{itemize}
\item[-] In (1), the center of $M_3\cong \GL_3$ acts trivially on both $\Pi_{E_6}$ and $\Pi_{\emptyset}$, the minimal and the trivial representation of the Levi $\mathcal M_3$. 
\item[-] In (2), $\delta(g_1,g_2)= |g_1|^{-1/2} \times |\det(g_2)|^{1/2} $ is a character of $Q_1$. 
\item[-] For $i=1,2$, $Q_i$ acts on $C_c(\GL_i)$ by right  translations via the factor $\GL_i$ as described above in \S \ref{SS:jac}.
\end{itemize} 
\end{prop}

\vskip 10pt 
\noindent
\underline{Case $P_1$:} 
 This case is not in the literature; however, it is similar 
to the computation of the Jacquet module of the minimal representation of $E_8$ with respect to a maximal parabolic subgroup of $F_4$ in \cite[\S 5]{SWo}. 
The unipotent radical subgroups of $P_1$ and $\mathcal P_1$ are Heisenberg groups with
$M_1 \cong \GSp_4$. Let $\nu$ be the similitude character of $\GSp_6$.  The modulus character of $M_1$ is 
\[ 
\delta_{P_1}(m)= |\nu(m)|^3. 
\] 
Recall that $\mathbb O_0$ is the space of trace 0 octonions. On $\mathbb O_0^4$, we have 
 the  row-vector  action 
 \[  (x,y,x',y') \mapsto (x,y,x',y')m^{-1} \quad \text{ of $m \in \GSp_4$} \]
 preserving the form  $\mathbb O_0^4 \rightarrow \wedge^2 \mathbb O_0$  given by
 \[ 
 (x,y,x',y') \mapsto x\wedge x' + y\wedge y'. 
 \] 
 Let $\Omega\subset \mathbb O_0^4$ be the set of all  nonzero $(x,y,x',y')$ such that the linear subspace 
$\langle x,y,x',y' \rangle \subset \mathbb O_0$ is a null-space for octonion multiplication \underline{and} 
$ x\wedge x' + y\wedge y'=0$. 
We have an exact sequence of $G_2 \times \GSp_4$-modules 
\[ 
0 \rightarrow C_c(\Omega) \rightarrow \Pi_{N_1} \rightarrow \Pi_{\mathcal N_1} \rightarrow 0 
\] 
where $(g, m) \in G_2 \times \GSp_4$ acts on $f\in C_c(\Omega)$ by 
\[ 
((g,m) \cdot f) (x,y,x',y')) = |\nu(m)|^3 \cdot f((g^{-1}x,g^{-1}y,g^{-1}x',g^{-1}y')m). 
\] 
\vskip 5pt

Now the group $G_2 \times \GSp_4$ acts on $\Omega$ with two orbits $\Omega_1$ and $\Omega_2$, where $\Omega_i$ is the subset of quadruples
$(x,y,x',y')$ such that $\langle x,y,x',y' \rangle$ has dimension $i$. Thus $C_c(\Omega)$ has a filtration with $C_c(\Omega_2)$  as a submodule and 
$C_c(\Omega_1)$ as  a quotient. Each of these can be explicitly described as $G_2 \times \GSp_4$-modules.  

\vskip 5pt

In order to state the result, let $Q_1$ and $Q_2$ be the maximal 
parabolic subgroups of $\GSp_4$ stabilizing subspaces consisting of row vectors $(\ast, 0,0,0)$ and $(\ast, \ast,0,0)$, respectively.  
Let  $L_1\cong \GL_1\times \GL_2$ be the Levi subgroup of $Q_1$ such that $(g_1,g_2)\in \GL_1\times \GL_2$ acts on the quadruples, after rearranging the order, by 
\[ 
(x,x', y,y') \mapsto (xg_1^{-1}, x' g_1 \det(g_2)^{-1}, (y,y') g_2^{-1}). 
\] 
Let $L_2\cong \GL_2\times \GL_1$ be the Levi subgroup of $Q_2$ such that $(g_2,g_1)\in \GL_2\times \GL_1$ acts on the quadruples by 
\[ 
(x,y, x',y') \mapsto ((x,y) g_2^{-1}, (x',y') g_1^{-1}g_2^{\top}). 
\]   
The similitude character $\nu$, restricted to $L_1$ and $L_2$, is given by  
\[ \nu(g_1,g_2)=\det g_2 \quad \text{and } \quad  \nu(g_2, g_1)= g_1  \]
 respectively, and the modulus characters are 
\[ 
\delta_{Q_1} (g_1, g_2)= |g_1|^{-4} \cdot |\det(g_2)|^2 \quad \text{ and } \quad \delta_{Q_2} (g_2, g_1)= |\det(g_2)|^{-3} \cdot |g_1|^3. 
\] 
Recalling that $r_{P_1}(\Pi)= \delta_{P_1}^{-1/2} \cdot \Pi_{N_1}$ is a  normalized Jacquet module, we have: 

\begin{prop}\label{P:jf_p1} As a $G_2\times \GSp_4$-module, $r_{P_1}(\Pi) $ has a filtration with three successive 
subquotients (from top to bottom): 
\begin{enumerate} 
\item $\delta_{P_1}^{-1/2} \cdot \Pi_{\mathcal N_1} = \Pi_{D_6}\cdot |\nu|^{1/2}  \oplus \Pi_{\emptyset} \cdot |\nu|^{3/2}$ . 
\item  $\Ind_{Q\times Q_1}^{G_2\times \GSp_4}(\delta \cdot C_c(\GL_1))$. 
\item  $\Ind_{P\times Q_2}^{G_2\times \GSp_4}(C_c(\GL_2))$. 
\end{enumerate} 
Here, note that
\vskip 5pt

\begin{itemize}
\item[-] In (1), the center of $M_1\cong \GSp_4$ acts trivially on both $\Pi_{D_6}$ and $\Pi_{\emptyset}$, the minimal and the trivial representation of the Levi $\mathcal M_1$.
\item[-] In (2) $\delta(g_1,g_2)= |g_1|^{-1/2} \times |\det(g_2)|^{1/2} $, a character of $Q_1$. 
\item[-] For $i=1,2$, $Q_i$ acts on $C_c(\GL_i)$ by right translations via the factor $\GL_i$  as described above in \S \ref{SS:jac}.
\end{itemize} 
\end{prop} 

\vskip 10pt 
\noindent
\underline{Case $P_2$:} 
 A variant of this case can be found in \cite{GS99} for the non-split form of $\PGSp_6$. However, for the split case considered in this paper, the Jacquet module filtration contains an additional  ``middle" term. 
 \vskip 5pt
 
The unipotent radical subgroups $N_2\subset P_2$ and $\mathcal N_2\subset  \mathcal P_2$ are two-step nilpotent subgroups. Let $\mathcal Z_2 \subset \mathcal N_2$ be the 
center of $\mathcal N_2$.  We now explain how the kernel of the natural projection $\Pi_{\mathcal Z_2} \rightarrow \Pi_{\mathcal N_2}$ contributes to  
$\Pi_{N_2}$.  We have 
\[ 
0\rightarrow C_c(\omega) \rightarrow \Pi_{\mathcal Z_2} \rightarrow \Pi_{\mathcal N_2} \rightarrow 0
\]
where $\omega$ is the $\mathcal M_2$-highest weight orbit in 
\[ 
\bar{\mathcal N}_2/\bar{\mathcal Z}_2  \cong \mathbb O\otimes \mathbb M_2(F) =\mathbb M_2(\mathbb O)
\] 
 where $\bar{\mathcal N}_2$ is the unipotent group opposite to $\mathcal N_2$, and 
  $\mathbb M_2(F)$ is the set of two-by-two matrices.  (In the non-split case $\mathbb M_2(F)$ is replaced by a division algebra, 
so $\Omega$ is empty; see the discussion on \cite[Pg. 137]{GS99}.)  Recall that the type of $\mathcal M_2$ is $D_5 \times A_1$ and  
$\bar{\mathcal N}_2/\bar{\mathcal Z}_2 \cong F^{16} \otimes F^2$ where $F^{16}$ is a spin-module of $D_5$. In the above isomorphism we assume that $A_1$ acts from 
the right on $\mathbb M_2(\mathbb O)$,  and columns are vectors in the spin-module.  Thus $\omega$ is the set of non-zero matrices 
 \[ 
\left(\begin{array}{cc} 
x & x' \\ 
y & y' \end{array}  
\right) 
\] 
where the two columns are linearly dependent over $F$ and each column (if non-zero) is a highest weight vector in the spin-module.  
Let $\Omega$ be the subset of $\omega$ such that $x, x', y, y'$ are traceless octonions. 
We have an exact sequence of $G_2 \times M_2$-modules 
\[ 
0 \rightarrow C_c(\Omega) \rightarrow (\Pi_{\mathcal Z_2})_{N_2}  \rightarrow \Pi_{\mathcal N_2} \rightarrow 0 
\] 
where $(g, (\alpha,\beta)) \in G_2 \times (\GL_2 \times \GL_2) /\GL_1^{\nabla}$ acts on $f\in C_c(\Omega)$ by 
\[ 
((g,(\alpha,\beta))  \cdot f) \left( 
\left(\begin{array}{cc} 
x & x' \\ 
y & y' \end{array}  
\right) 
\right) =
 \delta'(\alpha\beta) \cdot  f 
 \left( \bar\alpha 
\left(\begin{array}{cc} 
g^{-1} x & g^{-1}x' \\ 
g^{-1}y & g^{-1}y' \end{array}  
\right) \beta 
\right) 
\] 
for some (unknown) character $\delta'$.  The highest weight orbit in the 16-dimensional spin module is described in 
\cite{MS}. That result, applied to each column of $\mathbb M_2(\mathbb O)$, implies that  $x, x', y, y'$ (of an element in $\Omega$) generate a nil-subalgebra. 
The group $G_2 \times M_2$ acts on $\Omega$  with two orbits $\Omega_1$ and $\Omega_2$, where 
$\Omega_i$ consists of elements such that $\langle x,x',y,y'\rangle$ has dimension $i$. Thus $C_c(\Omega)$, as a $G_2\times M_2$-module, has 
 $C_c(\Omega_2)$ as a submodule and  $C_c(\Omega_1)$ as quotient. 

\begin{prop}\label{P:jf_p2} 
As a $G_2\times (\GL_2 \times \GL_2)/\GL_1^{\nabla}$-module, $r_{P_2}(\Pi) $ has a filtration with four successive 
sub quotients: 
\begin{enumerate} 
\item $\delta_{P_2}^{-1/2} \cdot \Pi_{\mathcal N_2} = \Pi_{D_5}\cdot |\det|^{1/2}  \oplus \Pi_{A_1} \cdot |\det |^{3/2}$ . 
\item  $\Ind_{Q\times (\bar B\times \bar B )/ \GL_1^{\nabla}}^{G_2\times (\GL_2 \times\GL_2)/\GL_1^{\nabla} }(\delta \cdot C_c(\GL_1))$. 
\item $\Ind_{P\times \bar B }^{G_2\times \GL_2 }(C_c(\GL_2))$.
\item $\Ind_{Q}^{G_2} W$. 
\end{enumerate} 
Here, note that:
\vskip 5pt

\begin{itemize}
\item[-] In (1), the second $\SL_2 \subset M_2$ acts trivially on the first summand, and the first $\SL_2 \subset M_2$ acts trivially on the second summand. 
The center of $M_2$ acts trivially on both $\Pi_{D_5}$ and $\Pi_{A_1}$, the minimal and a principal series representation of the two factors of $\mathcal M_3$. 
\item[-]  In (2) $\delta=|\cdot |^{1/2}\times |\cdot |$ on each $\bar B$.
\item[-]  In (3), $\bar B$ is the subgroup of the second factor $\GL_2$ of $M_2$. It acts on $C_c(\GL_2)$ by 
right translation by the scalar given by the $(1,1)$ matrix entry. The first factor $\GL_2$ of $M_2$ acts by right translations on $C_c(\GL_2)$. 
\item[-] In (4), $W$ is the Weil representation of $\GL_2 \times (\GL_2 \times \GL_2)/\GL_1^{\nabla} \cong \GL_2 \times {\rm GSO}_4$. 
\end{itemize}
\end{prop} 

This proposition is a combination of \cite[Proposition 8.1]{GS99}, which accounts for the bottom piece of the filtration (4),  and the above discussion. The pieces 
(2) and (3) are the spaces of functions $C_c(\Omega_1)$ and $C_c(\Omega_2)$, respectively. This also assumes that we have 
explicated the character $\delta'$ appearing in the action on $C_c(\Omega)$. 
To that end,  observe that (3) (or any unknown twist) gives a correspondence of generic principal series representations of $G_2$ and $\PGSp_6$ that has to be 
compatible with the one in Lemma  \ref{L:non_tempered}, and this determines $\delta'$ uniquely. 

\vskip 5pt

 \subsection{\bf Theta lifts from $\PGSp_6$.}
Using Propositions  \ref{P:jf_p3}, \ref{P:jf_p1} and \ref{P:jf_p2}, we can now prove the following analog of Lemma  \ref{L:non_tempered}.   
\vskip 5pt
 
\begin{lemma} \label{L:non_tempered2} 
Let $\sigma\in  {\rm Irr}({\rm PGSp}_6)$ be non-tempered. 
Then $\Theta(\sigma)=0$ unless $\sigma$ is as described in Lemma  \ref{L:non_tempered}. More precisely, 
\begin{itemize} 
\item If $\sigma \subset I_2( \tau^{\vee}\otimes \tau^{\vee})$,  then $\Theta(\sigma)$ is a quotient of $I_{Q}(\tau)$ and hence has finite length. Moreover, 
$\Theta(J_Q(\tau))\neq 0$ where $J_Q(\tau)$ is the unique irreducible quotient of $I_Q(\tau)$. 
\item If $\sigma \subset I_{13}( \tau^{\vee}\otimes 1)$,  then $\Theta(\sigma)$ is a quotient of $I_{P}(\tau)$,  and hence has finite length. Moreover,
 $\Theta(J_{P}(\tau))\neq 0$ where $J_{P}(\tau)$ is the unique irreducible quotient of $I_{P}(\tau)$. 
\end{itemize} 
\end{lemma} 
\begin{proof}  We set $H=\PGSp_6$. 
Assume that $\sigma$ is a Langlands quotient of a standard module for the maximal parabolic $P_2$. 
Then $\sigma \subseteq I_2(\tau_1^{\vee} \otimes \tau_2^{\vee})$ where $\tau_1$ and $\tau_2$ have the same central character and are both 
tempered representations of $\GL_2$ twisted by a positive power of $|\det|$.  Then 
\[ 
\Hom_H(\Pi, \sigma) \subseteq \Hom_H(\Pi, I_2(\tau_1^{\vee} \otimes \tau_2^{\vee})\cong \Hom_{M_2}( r_{P_2}(\Pi), \tau_1^{\vee} \otimes \tau_2^{\vee}). 
\] 
Let $\Pi_i$, $i=1,2,3,4$ be the sub quotients of $r_{P_2}(\Pi)$ as in Proposition \ref{P:jf_p2}, in the same order. We clam that 
$\Hom_{M_2}( r_{P_2}(\Pi), \tau_1^{\vee} \otimes \tau_2^{\vee})\cong \Hom_{M_2}( \Pi_4, \tau_1^{\vee} \otimes \tau_2^{\vee})$.  Assume this claim for a moment. 
Then 
\[ 
\Hom_{M_2}( r_{P_2}(\Pi), \tau_1^{\vee} \otimes \tau_2^{\vee})\cong \Hom_{M_2}( \Pi_4, \tau_1^{\vee} \otimes \tau_2^{\vee}) 
\cong \Hom_{M_2}( \Ind_Q^{G_2} W, \tau_1^{\vee} \otimes \tau_2^{\vee}) 
\] 
where $W$ is the Weil representation of $\GL_2^3$.  This implies that $\Theta(\sigma)=0$ unless $\tau_1\cong \tau_2$, and if we denote this 
representation as $\tau$, then $\Theta(\sigma)$ is a non-zero quotient of the standard module $I_Q(\tau)$. In order to prove the claim, we need to show that 
${\rm Ext}^n_{M_2}( \Pi_i, \tau_1^{\vee} \otimes \tau_2^{\vee})=0$ for all $n$ and $i<4$.  Consider $i=3$. Then, using the (second) Frobenius reciprocity for 
induction from $\bar B$ to $\GL_2$, the first factor of $M_2$, we have 
\[ 
 {\rm Ext}^i_{M_2}( \Ind_{P\times \bar B }^{G_2\times \GL_2 }(C_c(\GL_2)) , \tau_1^{\vee} \otimes \tau_2^{\vee}) \cong 
{\rm Ext}^i_{T \times \GL_2 }( \Ind_{P}^{G_2 }(C_c(\GL_2)),  r_B(\tau_1^{\vee}) \otimes \tau_2^{\vee}) 
\] 
where $T\cong \GL_1 \times \GL_1$ is the torus of diagonal matrices in $\GL_2$. Now recall that the second $\GL_1$ acts trivially on 
$\Ind_{P}^{G_2 }(C_c(\GL_2))$. On the other hand, since $\tau_1$ is a tempered with a positive twist of $|\det|$, the second 
$\GL_1$ acts on $r_B(\tau_1^{\vee})$ with characters $\chi$ such that $|\chi|$ is a negative power of absolute value. This prove the vanishing for $i=3$. The other two 
cases are just as easy or even easier: for $i=1$ vanishing follows from central character considerations, and for $i=2$  using Frobenius reciprocity where it 
suffices that either $\tau_1$ or $\tau_2$ is twist of a tempered representation by a positive power of $|\det|$.

Now assume that $\sigma$ is a Langlands quotient of a standard module for the parabolic $P_{23}=P_2\cap P_3$. 
Then, by induction in stages, we get that $\sigma \subseteq I_2(\tau_1^{\vee} \otimes \tau_2^{\vee})$ where $\tau_1$ is a twist of a tempered 
representation by a positive power of $|\det|$.  This is enough to show that $\Hom_{M_2} (\Pi_i, \tau_1^{\vee} \otimes \tau_2^{\vee}) =0$ for 
$i<4$. Thus, if $\Theta(\sigma)\neq 0$ then $\Hom_{M_2} (\Pi_4, \tau_1^{\vee} \otimes \tau_2^{\vee}) \neq 0$.  This implies that 
 $\tau_1\cong \tau_2$, contradicting that $\sigma$ is a Langlands quotient of a standard module for the parabolic $P_{23}$.  
 Hence $\Theta(\sigma)=0$.  
 
 If $\sigma$ is a Langlands quotient of a standard module for the parabolic $P_{12}=P_1\cap P_2$ 
then, by induction in stages, we get that $\sigma \subseteq I_2(\tau_1^{\vee} \otimes \tau_2^{\vee})$ where now $\tau_2$ is a twist of a tempered 
representation by a positive power of $|\det|$. In this case $\Hom_{M_2} ( \Pi_i, \tau_1^{\vee} \otimes \tau_2^{\vee})=0$ for $i\neq 3$ 
by repeating the above arguments. For $i=3$ we have 
\[ 
 \Hom_{M_2}( \Ind_{P\times \bar B }^{G_2\times \GL_2 }(C_c(\GL_2)) , \tau_1^{\vee} \otimes \tau_2^{\vee})\cong 
\Hom_{T \times \GL_2 }( \Ind_{P}^{G_2 }(C_c(\GL_2)),  r_B(\tau_1^{\vee}) \otimes \tau_2^{\vee}) 
\] 
and the last space is isomorphic to 
\[ 
 \Hom_{T \times \GL_2 }( \Ind_{P}^{G_2 }(\tau_2) \otimes \tau_2^{\vee} ,  r_B(\tau_1^{\vee}) \otimes \tau_2^{\vee}).
\] 
Recall that $T=\GL_1 \times \GL_1$ and the second $\GL_1$ acts trivially on $\Ind_{P}^{G_2 }(C_c(\GL_2))$ and hence 
on its quotient $\Ind_{P}^{G_2 }(\tau_2) \otimes \tau_2^{\vee}$. The first $\GL_1$ acts on this space by the central character of $\tau_2$, which 
is equal to the central character of $\tau_1$,  hence it is a nontrivial character, say $\chi$. Hence the above $\Hom$ space, if non-zero, is non-trivial 
if and only if $\chi\otimes 1$ is an exponent of $\tau_1$, and  then it is isomorphic to 
\[ 
\Hom_{ \GL_2 }( \Ind_{P}^{G_2 }(\tau_2) \otimes \tau_2^{\vee} ,  \tau_2^{\vee})\cong  \Hom( \Ind_{P}^{G_2 }(\tau_2), \mathbb C)=I_P(\tau_2)^*. 
\] 
Summarizing, $\Theta(\sigma)\neq 0$ implies that $\Theta(\sigma)$ is a quotient of $I_P(\tau_2)$. It follows that $J_P(\tau_1) \otimes \sigma$ is 
a quotient of $\Pi$, where $J_P(\tau_2)$ is the unique irreducible quotient of $I_P(\tau_2)$. 
  But, by Lemma  \ref{L:non_tempered},  $J_P(\tau_2)$ does not lift to $\sigma$. 
 This is a contradiction, hence $\Theta(\sigma)=0$. 

\vskip 5pt 
The remaining non-tempered representations of $H$  (associated to standard modules induced from $P_{123}$, $P_{13}$, $P_1$ or $P_3$) are easily dealt with using $r_{P_1}(\Pi)$ and $r_{P_3}(\Pi)$. We leave details to the reader.

\end{proof} 

\vskip 5pt 

\section{\bf Consequences of Jacquet Module Computations}  \label{S:conseq}

We can now draw some definitive consequences of  the Jacquet module computations of the previous section. In particular, we shall determine the theta lift of nontempered representations explicitly, and also complete the proofs of Lemmas \ref{L:basic} and  \ref{L:basic2} for the dual pair $G_2 \times \PGSp_6$.
\vskip 5pt

\subsection{\bf Lift of nontempered representations.}

Taken together, Lemmas \ref{L:non_tempered}  and \ref{L:non_tempered2}  allow us to determine the theta lift of nontempered representations explicitly:
\vskip 5pt

\begin{theorem}  \label{T:nont}
We have:
\begin{itemize}
\item[(a)] $\Theta(J_Q(\tau))$  is a nonzero quotient of $I_2(\tau \otimes \tau )$ and hence has finite length with unique irreducible quotient 
$J_2(\tau \otimes \tau)$. Likewise, $\Theta(J_2(\tau \otimes \tau))$ is a nonzero quotient of $I_Q(\tau)$ and hence has finite length with unique irreducible quotient $J_Q(\tau)$.
\vskip 5pt

\item[(b)] $\Theta(J_P(\tau))$ is a nonzero quotient of $I_{13}(\tau \otimes 1)$ and hence has finite length with unique irreducible quotient $J_{13}(\tau \otimes 1)$. Likewise, 
$\Theta(J_{13}(\tau \otimes 1))$ is a nonzero quotient of $I_P(\tau)$ and hence has finite length with unique irreducible quotient $J_P(\tau)$.
 \vskip 5pt
 
 \item[(c)] For all other nontempered $\sigma \in {\rm Irr}(\PGSp_6)$ different from those in (a) and (b), $\Theta(\sigma) = 0$.
\end{itemize} 
In particular, if $\pi \otimes \sigma \in {\rm Irr}(G_2 \times \PGSp_6)$ is such that $\pi \otimes \tau$ is a quotient of the minimal representation $\Pi$, then
\[  \text{$\pi$ nontempered} \Longleftrightarrow \text{$\sigma$ nontempered}. \]
Hence, we have shown  Lemma \ref{L:basic} for nontempered representations and also Lemma \ref{L:basic2}.
 \end{theorem}

 \vskip 15pt

\subsection{\bf Finiteness of $\Theta(\pi)_{nc}$}  \label{S:finite} 
To complete the proof of Lemma \ref{L:basic}, we need to show that for  tempered  $\pi \in {\rm Irr}(G_2)$ and $\sigma \in {\rm Irr}(\PGSp_6)$, the noncuspidal components $\Theta(\pi)_{nc}$ and $\Theta(\sigma)_{nc}$ are of finite length.     

\vskip 5pt

To show that $\Theta(\pi)_{nc}$ has finite length, it suffices to show that for each maximal parabolic subgroup $P_i = M_i N_i$ (with $1 \leq i \leq 3$) of $\PGSp_6$,   the Jacquet module $J_{P_i}(\Theta(\pi))$ has finite length as an $M_i$-module. In other words, we need to show that the multiplicity space of the maximal $\pi$-isotypic quotient of $r_{P_i}(\Pi)$ has 
finite length as an $M_i$-module.
\vskip 5pt

We have described in Propositions \ref{P:jf_p3}, \ref{P:jf_p1} and \ref{P:jf_p2} an equivariant filtration of $r_{P_i}(\Pi)$ as an $G_2 \times M_i$-module and described the successive quotients. It suffices to show that, for each of these successive quotients $\Sigma$, the multiplicity space of the $\pi$-isotypic quotient of $\Sigma$ has finite length. We shall explain how this can be shown, depending on whether $\Sigma$ is a top piece of the filtration or not. The difference lies in the fact that the top piece of the filtration involves a minimal representation of a smaller group $\mathcal{M}_i$ and hence one needs to consider theta correspondence in lower rank situations.  When $\Sigma$ is not the top piece of the filtration, the finite length of the multiplicity space of the maximal $\pi$-isotypic quotient of $\Sigma$ as an $M_i$-module follows readily from the explicit description of $\Sigma$. We give two examples as illustration:

\vskip 5pt

\begin{itemize}
\item Consider the case of $P_3 = \GL_3 \cdot N_3$. The bottom piece of the filtration in Proposition \ref{P:jf_p3} is
\[  \Sigma =  {\rm Ind}_{P \times Q_2}^{G_2 \times \GL_3} C_c(\GL_2). \]
Then for $\pi \in {\rm Irr}(G_2)$,  
\[  \Theta_{\Sigma}(\pi)^* := \Hom_{G_2}(\Sigma, \pi) \cong \Hom_M \left( {\rm Ind}^{M \times \GL_3}_{M \times Q_2} C_c(\GL_2), r_{\bar{P} }(\pi) \right) \]
where $M \cong \GL_2$. Now $r_{\bar{P}}(\pi)$ is a finite length $M$-module  and for any of its irreducible subquotient $\sigma$,
\[  
\Hom_M \left( {\rm Ind}^{M \times \GL_3}_{M \times Q_2} C_c(\GL_2), \sigma\right) \cong   \left( {\rm Ind}^{\GL_3}_{Q_2} \sigma^{\vee}
\right)^* \]
using the fact that the maximal $\sigma$-isotypic quotient of the regular representation $C_c(\GL_2)$ is of the form $\sigma^{\vee} \otimes \sigma$. 
On taking smooth vectors (which is a left exact functoir), we see that $\Theta_{\Sigma}(\pi)^{\vee}$ has a finite filtration whose successive quotients are submodules of ${\rm Ind}^{\GL_3}_{Q_2} \sigma^{\vee}$ for some irreducible $\sigma$. In particular, $\Theta_{\Sigma}(\pi)$ has finite length.
\vskip 5pt

\item Consider the case of $P_2 = M_2 \cdot N_2$ with $M_2 =  (\GL_2 \times \GL_2) / \GL_1^{\nabla} \cong {\rm GSO}_4$. The bottom piece of the filtration in Proposition \ref{P:jf_p2} is
\[  \Sigma =  {\rm Ind}_{Q \times M_2}^{G_2 \times M_2} W \]
where $W$ is the Weil representation for $\GL_2 \times {\rm GSO}_4$. Then for $\pi \in {\rm Irr}(G_2)$, 
\[   \Theta_{\Sigma}(\pi)^*  := \Hom_{G_2}(\Sigma, \pi) \cong \Hom_{L \times M_2}(W, r_{\overline{Q}}(\pi)) \]
Now $r_{\overline{Q}}(\pi)$ has finite length as $L$-module (where $L \cong \GL_2$) and if $\sigma$ is an irreducible subquotient, 
$\Hom_{L \times M_2}(W, \sigma) = \Theta_W(\sigma)^*$ where $\Theta_W(\sigma)$ is the big theta lift of $\sigma \in {\rm Irr}(\GL_2)$ to ${\rm GSO}_4$, which has finite length by the Howe duality theorem for classical (similitude) theta correspondence. From this, one deduces as above that $\Theta_{\Sigma}(\pi)$ has finite length as an $M_2$-module.
\end{itemize}

\vskip 5pt

Now let's consider the case when $\Sigma$ is the top piece of the filtration. From Propositions \ref{P:jf_p3}, \ref{P:jf_p1} and \ref{P:jf_p2}, we see that we need to consider the following theta correspondences in lower rank:
\vskip 5pt

\begin{itemize}
\item  $G_2 \times \PGL_3$ in $E_6$: for this case, the finite length of the big theta lift has been verified in Theorem \ref{T:PGL3}. 
\vskip 5pt

\item $G_2 \times \SO_3 \subset \SO_{10}$ or $G_2 \times \SO_5 \subset \SO_{12}$; we shall now treat these two cases together in the following proposition. 
\end{itemize}

\begin{prop}  \label{P:Hn}
Let $\Pi_n$ be the minimal representation of $\SO(2n)$ for  $n=5$ or $6$. Then for tempered $\pi \in {\rm Irr}(G_2)$, $\Theta_n(\pi)$ is a finite length $H_n$-module where 
$H_n= \SO_{2n-7}$. 
\end{prop} 

\begin{proof} 
 We shall use the fact that the minimal representation of $\SO_{2n}$ ($n = 5 $ or $6$) is the big theta lift of the trivial representation of $\SL_2$ (see \cite[Prop. 8.4]{Y} for the irreducibility of this big theta lift)  and then appeal to the seesaw identity arising from the seesaw diagram:
 \[
 \xymatrix{
 \tilde{SL}_2 \times \tilde{SL}_2  \ar@{-}[dr] \ar@{-}[d] & \SO_{2n}
     \ar@{-}[d] \\
  \SL_2 \ar@{-}[ur] &   G_2 \times \SO_{2n-7}}
\]
From this, we see that
\begin{equation} \label{E:ssid}
  \Theta_n(\pi)^* = \Hom_{G_2}(\Pi_n, \pi) = \Hom_{\SL_2}(\Omega_{2n-7, \psi} \otimes  \tilde{\Theta}_{\bar{\psi}}(\pi),\mathbb{C})\cong \Hom_{\tilde{SL}_2}(\Omega_{2n-7,\psi}, \tilde{\Theta}_{\psi}(\pi)) \end{equation}
as $H_n$-modules, where 
\begin{itemize}
\item $\Omega_{2n-7,\psi}$ is the Weil representation of $\tilde{\SL}_2 \times \SO_{2n-7}$ (with respect to a nontrivial additive character $\psi$ of $F$);
\item  $\tilde{\Theta}_{\psi}(\pi)$ denotes the big $\psi$-theta lift of $\pi$ to $\tilde{\SL}_2$, with respect to the Weil representation 
$\Omega_{\psi}$ of $\tilde{SL}_2 \times \SO_7 \supset \tilde{\SL}_2 \times G_2$. 
\end{itemize}
We see in particular that if $\Theta_n(\pi)$ is nonzero, then $\pi$ has nonzero $\psi$-theta lift to $\tilde{\SL}_2$. 
 Moreover, it remains now to show that $\tilde{\Theta}_{\psi}(\pi)$ has finite length as an $\tilde{\SL}_2$-module; the desired result would then follow from this and the Howe duality theorem for $\tilde{\SL}_2 \times H_n$.
 \vskip 5pt
 
 Now the theta correspondence for $\tilde{\SL}_2\times G_2$ has been completely determined in \cite{GG06}, though the finiteness of $\tilde{\Theta}_{\psi}(\pi)$ was not formally stated there.   Let us see how this finiteness can be deduced from \cite{GG06}. 
 \vskip 5pt
 
 As before, let us write $\tilde{\Theta}_{\psi}(\pi) = \tilde{\Theta}(\pi)_c \oplus \tilde{\Theta}(\pi)_{nc}$ as a sum of its cuspidal and noncuspidal component. To show that $\tilde{\Theta}(\pi)_{nc}$ has finite length as a $\tilde{SL}_2$-module, it suffices to show that  the Jacquet module of $\tilde{\Theta}(\pi)$ with respect to a Borel subgroup $\tilde{B} = \tilde{T} \cdot N$ of $\tilde{\SL}_2$ has finite length as a $\tilde{T}$-module. Now \cite[Prop. 8.1]{GG06} gives a short exact sequence of $G_2$-modules:
 \[ \begin{CD}
 0 @>>>   {\rm Ind}_Q^{G_2} C_c(\GL_1) @>>> \Omega_N @>>> \mathbb{C} @>>> 0 
 \end{CD} \]
 where the action of $L \cong \GL_2$ on $\GL_1$ is via $\det$.  The finite length of $\Theta(\pi)_N$ follows from this via a similar argument as above, by examining $\Hom_{G_2}(\Omega_N, \pi)$. 
 \vskip 5pt

 It remains to show that $\tilde{\Theta}(\pi)_c$ has finite length. In fact, it was shown in \cite[Thm. 9.1(c) and (d)]{GG06} that for genuine supercuspidal representations $\sigma_1 \ncong \sigma_2$ of $\tilde{\SL}_2$, one has $\tilde{\theta}_{\psi}(\sigma_1) \ncong \tilde{\theta}_{\psi}(\sigma_2)$.
  In other words,   $\tilde{\Theta}(\pi)_c$ is irreducible or $0$. This shows that $\tilde{\Theta}_{\psi}(\pi)$ has finite length.

\end{proof}

 \vskip 5pt
 
 \subsection{\bf Finiteness of $\Theta(\sigma)_{nc}$.}  \label{S:finite2}
 For tempered $\sigma \in {\rm Irr}(\PGSp_6)$,  the finite length of $\Theta(\sigma)_{nc}$ as a $G_2$-module is shown in the same way, using Propositions \ref{P:jf_P} and \ref{P:jf_Q}.
 We leave the details to the reader and only consider the top pieces in the filtration of the two Jacquet modules. 
 \vskip 5pt
 
 \begin{itemize}
 \item For the maximal parabolic subgroup $P$, we have to consider the theta correspondence for $\PGSp_6 \times \PGL_2$ with respect to the minimal representation $\Pi_6$ of $\PGSO_{12}$.  For the purpose of showing finiteness, there is no harm in working with $\Sp_6 \times \SL_2$. Hence, the theta correspondence in question arises as follows. If $V_2$ and $V_6$ denote the 2-dimensional and 6-dimensional symplectic vector spaces, then we are considering the map
 \[  \Sp(V_2) \times \Sp(V_6) \longrightarrow \SO(V_2 \otimes V_6) \]
 and pulling back the minimal representation $\Pi_6$ of $\SO(V_2 \otimes V_6)$. As before, we shall use the fact that this minimal representation is the big theta lift of the trivial representation of $\SL_2$. More precisely, let $V_2'$ be another symplectic space of dimension $2$, then we have the map
 \[  \Sp(V'_2) \times \Sp(V_2) \times \Sp(V_6) \longrightarrow \Sp(V'_2) \otimes \SO(V_2 \otimes V_6)  \longrightarrow \Sp(V_2' \otimes V_2 \otimes V_6). \]
 Given the Weil representation $\Omega$ of $\Sp(V'_2) \times \SO(V_2 \otimes V_6)$ and $\sigma \in {\rm Irr}(\Sp(V_6))$, we have
 \[  \Theta(\sigma)^* \cong \Hom_{\Sp(V_6)}( \Pi_6 , \sigma) \cong \Hom_{\Sp(V'_2) \times \Sp(V_6)} (\Omega, 1_{\Sp(V'_2)} \otimes \sigma) \cong  \Hom(\Theta'(\sigma), 1_{\Sp(V'_2)}) \]
 where $\Theta'(\sigma)$ is the big theta lift of $\sigma$ to $\SO(V_2' \otimes V_2)$. Note that there is a natural isogeny 
 \[  \Sp(V'_2) \times \Sp(V_2) \longrightarrow  \SO(V'_2 \otimes V_2) \]
 whose image is of finite index. Hence, by the classical Howe duality theorem, $\Theta'(\sigma)$ is a finite length representation of $\Sp(V'_2) \times \Sp(V_2)$. This implies that 
 $\Theta(\sigma)$ has finite length.
  
 \item For the maximal parabolic subgroup $Q$, we need to consider the restriction of $\Pi_{A_5}$, a minimal representation of $\SL_6$ to $\Sp_6$. Note that $\Pi_{A_5}$ is a degenerate principal series representation induced from a maximal parabolic subgroup which stabilizes a line in the standard representation. Since $\Sp_6$ acts transitively on such lines, we see that the restriction of $\Pi_{A_5}$ to $\Sp_6$ is simply a degenerate principal series representation of $\Sp_6$. This implies the desired finiteness. 
\end{itemize}

\vskip 5pt

We have thus completed the proofs of Lemmas \ref{L:basic} and \ref{L:basic2}.

 \vskip 15pt

\section{\bf Howe Duality for $G_2 \times \PGSp_6$: General Case}  \label{S:howegeneral}

 Finally,  by combining  Theorem \ref{T:main} and Theorem \ref{T:nont}, we can establish the Howe duality theorem for $G_2 \times \PGSp_6$.
   \vskip 5pt
   
\begin{theorem}\label{T:main2}
Let $\pi \in {\rm Irr}(G_2$).

(i) $\Theta(\pi)$ is nonzero if and only if $\pi$ has zero theta lift to $PD^{\times}$.
\vskip 5pt

(ii) If $\Theta(\pi) \ne 0$, then $\Theta(\pi)$ is a finite length representation of $\PGSp_6$ with a unique irreducible quotient $\theta(\pi)$.
\vskip 5pt

(iii) For $\pi_1, \pi_2 \in {\rm Irr}(G_2)$, 
\[  \theta(\pi_1) \cong \theta(\pi_2) \ne 0 \Longrightarrow \pi_1 \cong \pi_2. \]
\vskip 5pt

(iv) If $\Theta(\pi) \ne 0$, then $\theta(\pi)$ is tempered if and only if $\pi$ is tempered.
 \vskip 5pt
 
 (v) If $\pi$ is non-tempered, then $\theta(\pi)$ is nonzero and the L-parameter of $\theta(\pi)$ is obtained from that of $\pi$ by composing with the natural inclusion
  $G_2(\mathbb C) \subset {\rm Spin}_7(\mathbb C)$. 
  
 \end{theorem}

\vskip 10pt

\subsection{\bf Explicit correspondence.}

\vskip 10pt

We  can in fact  determine the theta lift $\theta(\pi)$ explicitly if $\pi$ is tempered and noncuspidal. Indeed, we may also determine $\theta(\pi)$ for those tempered $\pi$ which has nonzero theta lift to $\PGL_3$. To achieve this, we shall use the following four facts:

\vskip 5pt

\begin{itemize} 
\item If $\pi$ does not appear in the correspondence with $PD^{\times}$, then $\theta(\pi)\neq 0$ (Theorem \ref{T:main2}(i)).  
\vskip 5pt 
\item If $\pi$ is tempered and $\theta(\pi) \ne 0$, then $\theta(\pi)$ is  irreducible and tempered (Theorem \ref{T:main2}(v)). 
\vskip 5pt 
\item  If $\pi$ is nongeneric, then $\theta(\pi)$ is nongeneric (Corollary \ref{C:GS-Whit}(ii)). 
\vskip 5pt 
\item The cuspidal support of $\theta(\pi)$ can be computed (from the Jacquet module computations of \S \ref{S:jacquet}). 
\end{itemize} 

 More precisely we have: 

\begin{theorem} \label{T:pgl3}
 Let $\pi$ be an irreducible tempered representation of $G_2$. Assume that $\pi$ is a lift of a (necessarily tempered) representation $\tau$ 
of $\PGL_3$, that is, $\pi=\theta_B(\tau^{\epsilon})$ for some $\epsilon = \pm$. Then we have the following: 

\vskip 5pt 
\noindent (i) If $\tau \ncong \tau^{\vee}$ then $\theta(\pi) \cong I_3(\tau)\cong I_3(\tau^{\vee}) \in {\rm Irr}(\PGSp_6)$. 

\vskip 5pt 

\noindent (ii) If $\tau \cong \tau^{\vee}$ and the parameter of $\tau$  contains a trivial summand,  then $\theta(\pi) \cong I_3(\tau)$.  

\vskip 5pt 

\noindent (iii) If $\tau \cong \tau^{\vee}$ and the parameter of $\tau$ does not contain a trivial summand, then $\pi$ is one of the two representations 
$\pi_{\rm gen}= \theta_B(\tau^+)$ and $\pi_{\rm deg}= \theta_B(\tau^-)$. In this case $I_3(\tau)=I_3(\tau)_{\rm gen} \oplus I_3(\tau)_{\rm deg}$, 
\[ 
\theta(\pi_{\rm gen})= I_3(\tau)_{\rm gen} \text{ and } \theta(\pi_{\rm deg})= I_3(\tau)_{\rm deg}. 
\] 

\end{theorem} 

We now deal with the remaining tempered representations of $G_2$. Non-supercuspidal representations are mostly 
constituents of the principal series $I_Q(\tau)$ where 
$\tau$ is a discrete series representation. These representations lift to constituents of  the principal series $I_2(\tau\otimes\tau)$. More precisely, 
we have:

\begin{theorem} Let $\pi$ be an irreducible tempered representation of $G_2$ which is not a lift from 
 $\PGL_3$. Then we have the following: 

\vskip 5pt 
\noindent (i) Let $\tau$ be a unitary discrete series representation of $\GL_2$. Then $I_Q(\tau)$ is irreducible if and only if 
$I_2(\tau\otimes\tau)$ is irreducible. We have: 
\begin{itemize} 
\item If $I_Q(\tau)$ is irreducible  then 
\[ 
\theta(I_Q(\tau)) \cong I_2(\tau\otimes\tau). 
\] 
\vskip 5pt 
\item If $I_Q(\tau)$ is reducible  then  
\[ 
\theta(I_Q(\tau)_{\rm gen} ) \cong I_2(\tau\otimes\tau)_{\rm gen} \text { and } \theta(I_Q(\tau)_{\rm deg} ) \cong I_2(\tau\otimes\tau)_{\rm deg} . 
\] 
\end{itemize}

\vskip 5pt 

\noindent (ii) Assume that $\tau \cong \tau^{\vee}$ is a supercuspidal representation of $\GL_2$ with the trivial central character. 
Let $\delta_Q(\tau)$ and $\delta_P(\tau)$ be the square integrable constituents of $I_Q(1/2,\tau)$ and $I_P(1/2,\tau)$. Then 
\[ 
\theta(\delta_Q(\tau)) \cong \delta_2(\tau) \text { and } \theta(\delta_P(\tau)) \cong  \delta_{13}(\tau) 
\]
where $\delta_2(\tau)$ and $\delta_{13}(\tau)$ are the square integrable constituents of $I_2(1/2,\tau\otimes\tau)$ and $I_{13}(1/2,\tau\otimes 1)$.

\vskip 5pt 

\noindent (iii) Assume that $\tau \cong \tau^{\vee}$ is a supercuspidal representation of $\GL_2$ whose Langlands parameter has 
the image $S_3$. Recall that $I_Q(1,\tau)$ has a square integrable constituent denoted by $\pi_{\rm gen}[\tau]$. 
 Then 
\[ 
\theta(\pi_{\rm gen}[\tau]) \cong \sigma_{\rm gen}[\tau] 
\]
where $\sigma_{\rm gen}[\tau]$ is the square integrable constituent of $I_2(1,\tau\otimes\tau)$. 

\vskip 5pt 

\noindent (iv) Assume that $\chi^2=1$ and $\chi\neq 1$.  Recall that $I_Q(1/2,{\rm st}_{\chi})$ has a square integrable constituent denoted by $\pi_{\rm gen}[\chi]$. 
 Then 
\[ 
\theta(\pi_{\rm gen}[\chi]) \cong \sigma_{\rm gen}[\chi] 
\]
where $\sigma_{\rm gen}[\chi]$ is the square integrable constituent of $I_2(1/2,{\rm st}_{\chi}\otimes {\rm st}_{\chi})$. 

\vskip 5pt

\noindent (v) Steinberg lifts to Steinberg: 
\[ \theta({\rm St}_{G_2}) = {\rm St}_{{\rm PGSp}_6}. \]  

\end{theorem} 

Finally we need to deal with supercuspidal representations. In view of Theorem \ref{T:main2}(i)  and Theorem \ref{T:pgl3}, we only need to consider those supercuspidal representations which do not lift to $\PGL_3$ or $PD^{\times}$. We first introduce a thin family of supercuspidal representatons of $G_2$, namely those which participate in the theta correspondence for $\tilde{\SL}_2 \times G_2$. We have already encountered this theta correspondence in the proof of Proposition \ref{P:Hn}. As mentioned there, this theta correspondence has been studied in detailed in \cite{GG06}. 

\vskip 5pt

We first introduce some notation. For each  cuspidal representation $\rho$ of $\PGL_2 \cong \SO_3$, let $JL(\rho)$ be its Jacquet-Langlands lift to the an isotropic inner form $PB^{\times} = \SO_3^*$ (where $B$ is the quaternion division algebra) and let $\sigma_{\rho}$ be the $\psi$-theta lift of $JL(\tau)$ to $\tilde{\SL}_2$ (where $\psi$ is a fixed nontrivial additive character of $F$). Then $\sigma_{\rho}$ is an irreducible supercuspidal genuine representation of $\tilde{\SL}_2$. Consider now the $\psi$-theta lift 
\[  \pi_{\rho} := \theta(\sigma_{\rho}) \in {\rm Irr}(G_2) \]
of $\sigma_{\rho}$ from $\tilde{SL}_2$ to $G_2$. Now we  recall some results from \cite[Thm. 9.1]{GG06}:
\vskip 5pt

\begin{lemma} \label{L:GG06}
With the above notations, we have:
\vskip 5pt

\noindent (i) The representation $\pi_{\rho}$ is nonzero irreducible supercuspidal. Moreover, $\Theta(\pi_{\rho}) = \sigma_{\rho}$ under the theta correspondence for $\tilde{\SL}_2 \times G_2$. 
\vskip 5pt

\noindent (ii) The map $\rho \mapsto \pi_{\rho}$ is an injective map from the set ${\rm Irr}_{sc}(\PGL_2)$ of supercuspidal representations of $\PGL_2$ to ${\rm Irr}_{sc}(G_2)$. 

\vskip 5pt

\noindent (iii) Any $\pi \in {\rm Irr}_{sc}(G_2)$ which lifts to $\tilde{SL}_2$ but not $\PGL_3$ or $PD^{\times}$ is of the form $\pi_{\rho}$ for some $\rho \in {\rm Irr}_{sc}(\PGL_2)$.
 
\end{lemma}
\vskip 5pt

For $\sigma_{\rho} \in {\rm Irr}(\tilde{\SL}_2)$ as above, we may also consider its $\psi$-theta lift from $\tilde{\SL}_2$ to $\SO_5$ and set
\[  \tau_{\rho} = \Theta(\sigma_{\rho}) = \theta(\sigma_{\rho}) \in {\rm Irr}(\SO_5). \]
Then $\tau_{\rho}$ is a nongeneric supercuspidal representation of $\SO_5$ belonging to a so-called Saito-Kurokawa A-packet. The representations $\pi_{\rho}$ and $\tau_{\rho}$ are related as follows:
\vskip 5pt

\begin{lemma} \label{L:GG062}
Consider the restriction of the minimal representation of $\SO_{12}$ to $G_2 \times \SO_5$. Then for $\rho \in {\rm Irr}_{sc}(\PGL_2)$, 
\[  \Theta(\pi_{\rho}) = \tau_{\rho}. \]
\end{lemma}

\begin{proof}
We shall use the seesaw diagram in the proof of Proposition \ref{P:Hn}. The ensuing seesaw identity (\ref{E:ssid})  and Lemma  \ref{L:GG06}(i) give:
\[ \Theta(\pi_{\rho})^* \cong \Hom_{\tilde{\SL}_2}(\Omega_5, \sigma_{\rho})  = \tau_{\rho}^*. \]
Hence $\Theta(\pi_{\rho}) = \tau_{\rho}$. 
\end{proof}
\vskip 5pt
Now we have:
\vskip 5pt

\begin{prop} Let $\pi$ be an irreducible  supercuspidal representation of $G_2$ that is not a lift from $\PGL_3$ or $PD^{\times}$. Then we have the following two possibilities:
\vskip 5pt

\begin{itemize}
\item If $\pi = \pi_{\rho}$ for some $\rho \in {\rm Irr}_{sc}(\PGL_2)$ (as in Lemma \ref{L:GG06}),  then
\[ 
\theta(\pi_{\rho})=\delta_1(\tau_{\rho}) 
\] 
where $\delta_1(\tau_{\rho})$ is the square integrable subquotient of  $I_1(1/2,\tau_{\rho})$ given in Proposition \ref{P:P1}.

\vskip 5pt

\item  If $\pi$ is not of the above form, then $\theta(\pi)$ is supercuspidal.
\end{itemize}
\end{prop} 

\begin{proof} Let $\Pi$ be the minimal representation of $E_7$. Recall that $r_{P_i}$ is the normalized Jacquet functor with respect to the maximal 
parabolic $P_i$ in $\PGSp_6$. Then $\pi\otimes r_{P_i}(\theta(\pi))$ is a quotient of $r_{P_i}(\Pi)$. By the assumption that $\pi$ does not lift to $\PGL_3$, it follows that 
$r_{P_i}(\theta(\pi))=0$ for $i=2,3$. Thus either $r_{P_1}(\theta(\pi)) =0$, in which case $\theta(\pi)$ is supercuspidal, or $r_{P_1}(\theta(\pi))$ is a supercuspidal 
representation of the Levi factor $L_1=\GSp_4$. In fact, from Proposition \ref{P:jf_p1}, it follows that $r_{P_1}(\theta(\pi))= \tau\otimes | \nu|^{1/2}$ where 
$\tau$ is a (possibly reducible) supercuspidal representation of $\PGSp_4\cong \SO_5$ such that $\pi \otimes \tau$ appears as a quotient of the minimal representation of $\SO_{12}$.   
By the seesaw in the proof of Proposition \ref{P:Hn}, we see that $\pi$ must have nonzero theta lift to $\tilde{\SL}_2$ and hence is of the form  $\pi_{\rho}$ for some $\rho \in {\rm Irr}(\PGL_2)$ by Lemma \ref{L:GG06}(iii). Then Lemma  \ref{L:GG062} implies that $\tau = \tau_{\rho}$.  By Frobenius reciprocity and the fact that $\theta(\pi_{\rho})$ is tempered, we see that $\theta(\pi_{\rho}) = \delta_1(\tau_{\rho})$, as desired.
\end{proof} 
\vskip 5pt

As a consequence of the explicit results in this section, we have:
\vskip 5pt

\begin{cor}
If $\pi \in {\rm Irr}(G_2)$ is a discrete series representation which does not lift to $\PGL_3$ or $PD^{\times}$, then $\theta(\pi)$ is an irreducible discrete series representation of $\PGSp_6$. As a result, any discrete series representation of $G_2$ lifts to a discrete series of exactly one of $PD^{\times}$, $\PGL_3$ or $\PGSp_6$. That lift is Whittaker generic iff and only if $\pi$ is. 
\end{cor}

%At this point it is natural to ask if the situation in the previous proposition occurs, that is, are there $\SO_5$-supercuspidal quotients of the 
%minimal representation $\Pi_{D_6}$ of $\SO_{12}$? The answer is yes, in fact, this can be made very precise. To that end, let $K$ be a quadratic 
 %etale algebra over $F$. Let $E=F+K$ be a cubic etale algebra and $\psi_E$ the corresponding character of $N$, the unipotent radical of the maximal 
 %parabolic $P\subset G_2$. 
 %Then 
 %\[ 
 %(\Pi_{D_6})_{N, \psi_E}\cong C_c^{\infty}(\SO_5/\SO_4^K) 
 %\] 
% where $\SO_4^K=\SL_2(K)/\mu_2$.  But representations of split $\SO_5$ with the $\SO_4^K$-period appear in the theta correspondence 
 %for the dual pair $\SO_5 \times \widetilde{\SL}_2$. From the theory of classical theta correspondences, an irreducible representation of 
 %$\widetilde{\SL}_2$ lifts to a supercuspidal representation of $\SO_5$ precisely when the lift to the split $\SO_3$ vanishes, and such representations of 
 %$ \widetilde{\SL}_2$ are in correspondence with irreducible representations of $\SO_3^*$, the anisotropic form of $\SO_3$. 
 
  \vskip 15pt

\noindent{\bf Acknowledgments:} 
The authors would like to thank MFI in Oberwolfach for hospitality during a conference in October of 2019 when some of the ideas needed to finish this work emerged. 
Thanks are due to Petar Baki\'c, Baiying Liu and Yiannis Sakellaridis for help with some finer points. 
W.T. Gan is partially supported by an MOE Tier 1 grant 
R-146-000-320-114.  G. Savin is partially supported by 
 a   National Science Foundation grant DMS-1901745.


\vskip 15pt
\begin{thebibliography}{KMRT}


\bibitem[G99]{G}  W. T. Gan, 
{\em Exceptional Howe correspondences over finite fields.}  Compositio Math. 118 (1999), no. 3, 323--344. 
  
  
  
\bibitem[GG06]{GG06} W.T. Gan and N. Gurevich, {\em Nontempered $A$-packets of $G_2$: Liftings from $\widetilde{\SL}_2$.}  American J. Math. 170 (2006), 1105--1185.

\bibitem[GG09]{GG09} W.T. Gan and N. Gurevich, {\em CAP representations of $G_2$ and the Spin L-function of $\PGSp_6$.} Israel J. Math. 170 (2009), 1-52.




%\bibitem[K]{GRS} D. Ginzburg, S. Rallis and D. Soudry Duke Math, J. 

\bibitem[GS99]{GS99} W. T. Gan and G. Savin, 
{\em The Dual Pair $G_2 \times \mathrm {PU}_3(D)$ ($p$-Adic Case).}  Canad. J. Math. 51 (1999), 130--146. 

\bibitem[GS04]{GS04} W. T. Gan and G. Savin, {\em Endoscopic lifts from $PGL_3$ to $G_2$.}  Compositio Math. 140, No. 3 (2004), 793--808. 


\bibitem[GS05]{GS05} W. T. Gan and G. Savin, 
{\em On minimal representations definitions and properties.} Represent. Theory 9 (2005), 46--93. 

\bibitem[GS14]{GS14} W. T. Gan and G. Savin, {\em Twisted Bhargava cubes.} Algebra Number Theory 8 (2014), no. 8, 1913--1957.

\bibitem[GS15]{GS15} W. T. Gan and G. Savin, {\em A family of Arthur packets of triality $\Spin(8)$. } Preprint 2015.  

\bibitem[GT]{GT} W. T. Gan and S. Takeda, {\em The local Langlands conjecture for $GSp(4)$.}  Ann. of Math. (2) 173 (2011), no. 3, 1841--1882. 

\bibitem[GJ]{GJ} D. Ginzburg and D. Jiang, {\em Periods and lifting from $G_2$ to $C_3$.}  Israel J. Math. 123 (2001), 29-59.



\bibitem[GrS]{GrS} B. H. Gross and G. Savin, 
{\em Motives with Galois group of type $G_2$: an exceptional theta-correspondence.}  Compositio Math. 114 (1998), no. 2, 153--217

\bibitem[HKT]{HKT} M. Harris, C. Khare and J. Thorne, {\em A local Langlands parameterization for generic supercuspidal representations of $p$-adic $G_2$.} 
arXiv :1909.05933 

\bibitem[HMS]{HMS} J. S. Huang, K. Magaard and G. Savin, 
{\em Unipotent representations of $G_2$ arising from the minimal representation of $D_4^E$.}  J. Reine Angew. Math. 500 (1998), 65--81.

\bibitem[K]{K}  S. Kudla, {\em On the local theta-correspondence},
Invent. Math. 83 (1986), 229--255.

\bibitem[KMRT]{KMRT} M.-A. Knus, A. Merkurjev, M. Rost, and J.-P. Tignol, {\em The book of involutions. } 
 American Mathematical Society Colloquium Publications 44, Amer. Math. Soc., Providence, RI, 1998. 

%\bibitem[KLS10]{KLS10} C. Khare, M. Larsen and G. Savin. 
%{\em Functoriality and the inverse Galois problem II: Groups of type $B_n$ and $G_2$.}  



\bibitem[LS]{LS} H. Y. Loke and G. Savin, {\em On minimal representations of Chevalley groups of type $D_n$, $E_n$ and $G_2$.} 
Math. Ann. 340 (2008), no. 1, 195--208.

\bibitem[LT]{LT} S. Lonka and R. Tandon. 
{\em Zero weight space for tori inside a division algebra.} 
 J. Ramanujan Math. Soc. 33 (2018), no. 4, 435--454.
 

 
 \bibitem[MS]{MS} K. Magaard and G. Savin,
{\em Exceptional $\Theta$-correspondences. I.} Compositio Math. 107 (1997), no. 1, 89--123. 

\bibitem[MVW]{MVW}  C. Moeglin, M.-F. Vigneras and J.-L. Waldspurger, {\em Correspondances de Howe sur un corps $p$-adiques.} 
 Lecture Notes in Mathematics, vol 1291 (Springer, 1987).  

\bibitem[MW]{MW}  C. Moeglin and J.-L. Waldspurger, {\em Mod\`eles de Whittaker d\'eg\'en\'er\'es pour des groupes p-adiques.} 
 Math. Z. 196 (1987), 427--452. 


\bibitem[Mu]{Mu} G. Mui\'c, {\em The unitary dual of $p$-adic $G_2$.} Duke Math. J. 90 (1997), 465--493.

\bibitem[Ro]{R96} B. Roberts, {\em The Theta Correspondences for Similitudes.}  Israel J. Math. 94 (1996), 285--317. 



%\bibitem[SV]{SV} Y. Sakellaridis and A. Venkatesh, {\em Periods and harmonic analysis on spherical varieties.} Ast\'erisque 396, 2017. 


%{\em The degenerate Eisenstein series attached to the Heisenberg parabolic subgroups of quasi-split forms of ${\rm Spin}_8$.}, Trans. Amer. Math. Soc. 370 (2018), no. 8, 5983-6039.

\bibitem[Sa]{Sa} G. Savin, {\em A class of supercuspidal representations of $G_2(k)$}, Canad. Math. Bulletin, CMB (1999), vol. 42, no. 3, 393--400.

\bibitem[SWe]{SWe} G. Savin and M. Weissman, {\em Dichotomy for generic supercuspidal representations of $G_2$.} Compositio Math. 
147 (2011), 735--783.

\bibitem[SWo]{SWo}  G. Savin and M. Woodbury,
{\em Matching of Hecke operators for exceptional dual pair correspondences.} J. Number Theory 146 (2015), 534--556.

\bibitem[S]{S} A. Segal, 
{\em The degenerate residual spectrum of quasi-split forms of ${\rm Spin}_8$  associated to the Heisenberg parabolic subgroup}, Trans. Amer. Math. Soc. 372 (2019), no. 9, 6703-6754.


 \bibitem[Sh]{Sh} F. Shahidi, {\em A proof of Langlands' conjecture on Plancherel measures; complementary series for p-adic groups.}  Ann. of Math. 132 (1990), 273--330. 

\bibitem[Ta]{Ta} M. Tadi\'c, {\em Representations of $p$-adic symplectic groups.}  Compositio Math. 90 (1994), 123--181. 

\bibitem[Va]{Va} S. Varma, {\em On a result of Moeglin and Waldspurger in residual characteristic 2.} 
Math. Z. 277 (2014), no. 3-4, 1027--1048.
  
  \bibitem[Y]{Y} S. Yamana, {\em Degenerate principal series representations for quaternionic unitary groups},  Israel J. Math. 185 (2011), 77-124. 

%\bibitem[We] {We} M. Weissman. {\em } Representation Theory (2003). 









\end{thebibliography}
\end{document}